\documentclass[aos,preprint]{imsart}
\usepackage{amsmath}
\usepackage{amsthm}
\usepackage{amssymb}
\usepackage{footmisc}
\usepackage[authoryear,round]{natbib}
\usepackage{hypernat}
\usepackage{verbatim}
\usepackage{textcomp}
\usepackage{enumerate}
\usepackage{graphicx}
\usepackage{bbm}
\usepackage{array}
\usepackage{dsfont}
\newcolumntype{H}{>{\setbox0=\hbox\bgroup}c<{\egroup}@{}}

\usepackage{fancyhdr}
 
\newcommand{\R}{{\mathbb R}}
\newcommand{\E}{{\mathbb E}}
\renewcommand{\P}{{\mathbb P}}
\newcommand{\N}{{\mathbb N}}

\newcommand{\eps}{\varepsilon}

\newcommand{\argmin}[1]{{\operatorname{argmin}}_{#1}}

\newcommand{\lmin}[1]{{\operatorname{\lambda_{\text{min}}}}{#1}}

\DeclareMathOperator{\Var}{Var}
\DeclareMathOperator{\Cov}{Cov}
\DeclareMathOperator{\trace}{trace}

\DeclareMathOperator{\diag}{diag}

\DeclareMathOperator{\s}{span}

\newtheorem{theorem}{Theorem}[section]
\newtheorem{proposition}[theorem]{Proposition}
\newtheorem{lemma}[theorem]{Lemma}

\newtheorem{remark}[theorem]{Remark}
\newtheorem*{remark*}{Remark}
\newtheorem{example}[theorem]{Example}

\newtheorem*{definition*}{Definition}
\newtheorem{definition}[theorem]{Definition}

\numberwithin{equation}{section}

\newcounter{cond}

\newcounter{rcnt}[section]

\newcommand{\rem}[1]{}

\newcounter{desccount}

\newcommand{\descref}[1]{\hyperref[#1]{#1}}

\begin{document}
\sloppy

\title{Conditional predictive inference for stable algorithms}
\runauthor{Steinberger, Leeb}

\runtitle{Conditional predictive inference for stable algorithms}

\begin{aug}
\author{\fnms{Lukas} \snm{Steinberger}\thanksref{t1}\ead[label=a1]{lukas.steinberger@univie.ac.at}}
\and
\author{\fnms{Hannes} \snm{Leeb}\thanksref{t2}\ead[label=a2]{hannes.leeb@univie.ac.at}}

\thankstext{t1}{Lukas Steinberger was supported by the Austrian Science Fund (FWF): P~28233-N32 and I~5484-N, the latter project is part of the Research Unit 5381 of the German Research Foundation. Part of this research was conducted while Lukas Steinberger was funded by the German Research Foundation (DFG): RO 3766/401.}

\thankstext{t2}{Hannes Leeb was supported by the Austrian Science Fund (FWF): P 28233-N32 and P 26354-N26.}

\affiliation{University of Vienna}

\address{
	Lukas Steinberger\\	
	Department of Statistics and OR\\
	Data Science @ Uni Vienna\\
	University of Vienna  \\
	Oskar-Morgenstern-Platz 1 \\
	1090 Vienna, Austria\\
	\printead{a1}
	}

\address{
	Hannes Leeb\\
	Department of Statistics and OR\\
	Data Science @ Uni Vienna\\
	University of Vienna  \\
	Oskar-Morgenstern-Platz 1 \\
	1090 Vienna, Austria\\
	\printead{a2}
	}
	
\end{aug}



\begin{abstract}
We investigate generically applicable and intuitively appealing prediction intervals based on $k$-fold cross validation. We focus on the conditional coverage probability of the proposed intervals, given the observations in the training sample (hence, training conditional validity), and show that it is close to the nominal level, in an appropriate sense, provided that the underlying algorithm used for computing point predictions is sufficiently stable when feature-response pairs are omitted. Our results are based on a finite sample analysis of the empirical distribution function of $k$-fold cross validation residuals and hold in non-parametric settings with only minimal assumptions on the error distribution. To illustrate our results, we also apply them to high-dimensional linear predictors, where we obtain uniform asymptotic training conditional validity as both sample size and dimension tend to infinity at the same rate and consistent parameter estimation typically fails.
These results show that despite the serious problems of resampling procedures for inference on the unknown parameters \citep[cf.][]{Bickel83, Mammen96, ElKaroui15}, cross validation methods can be successfully applied to obtain reliable predictive inference even in high dimensions and conditionally on the training data. 

\end{abstract}

\maketitle


\section{Introduction}

It is the fundamental task of (supervised) statistical learning, when given an i.i.d. training sample of feature-response pairs $(x_i,y_i)$ and an additional feature vector $x_0$, to provide a point prediction for the corresponding unobserved response variable $y_0$. In such a situation, a prediction interval that contains the unobserved response variable with a prescribed probability provides valuable additional information to the practitioner. In many applications a training sample is obtained only once and is subsequently used to repeatedly construct point and interval predictions as new measurements of feature vectors become available. In such a situation, it is desirable to control the conditional coverage probability of the prediction interval given the observations in the training sample, rather than the unconditional probability \citep[the latter is controlled, for example, by procedures discussed in][]{Vovk05,Barber19b}. We call this notion of validity of a prediction interval, where the conditional coverage probability given the training data is (approximately) controlled, \emph{training conditional validity} \citep[see also][]{Vovk13}. Notice how this is markedly different from the notion of  \emph{object conditional validity} where the conditioning is on the new feature vector $x_0$ \citep[cf.][]{Barber19}.  

We study a very simple method based on cross validation (CV) residuals, which is generic, in the sense that it can be constructed on top of any algorithm that produces a point predictor, and that yields asymptotic training conditionally valid prediction intervals provided the underlying point prediction algorithm is sufficiently stable. To illustrate the idea and keep notation simple, we begin with the case of $n$-fold cross validation, that is, leave-one-out. For an i.i.d. training sample $T_n = (x_i,y_i)_{i=1}^n$ of size $n$, consisting of $\R^p\times\R$-valued feature-response pairs, and an additional feature vector $x_0$ in $\R^p$, suppose that we have decided to use a prediction algorithm $M_{n,p}:(\R^p\times\R)^n\times \R^p\to\R$ to produce a point prediction $\hat{y}_0 = M_n(T_n,x_0)$ for the real unobserved response $y_0$. If $T_n^{[i]} = (x_j,y_j)_{j\ne i}$ is the sample without the $i$-th observation pair, compute leave-one-out residuals $\hat{u}_i = y_i - M_{n-1,p}(T_n^{[i]},x_i)$, $1\le i\le n$. Finally, to obtain a prediction interval for $y_0$, compute appropriate empirical quantiles $\hat{q}_{\alpha_1}$ and $\hat{q}_{\alpha_2}$ from the collection $\hat{u}_1, \dots, \hat{u}_n$ and report the leave-one-out prediction interval 
$$
PI_{\alpha_1,\alpha_2}^{(L1O)}(T_n,x_0) = [\hat{y}_0+ \hat{q}_{\alpha_1}, \hat{y}_0+ \hat{q}_{\alpha_2}].
$$ 
In this paper we investigate the following conditional coverage probability given the training data, where the randomness comes only from the feature-response pair $(x_0,y_0)$ in the prediction period,
$$
\P(y_0\in PI_{\alpha_1,\alpha_2}^{(L1O)}(T_n,x_0) \| T_n).
$$ 
We investigate this probability in finite samples and in more specific asymptotic settings, including those where the dimension $p$ of the feature vectors $x_i$ increases at the same rate as sample size $n$. We find that even in these challenging scenarios where both $n$ and $p$ are large, the conditional coverage probability of $PI_{\alpha_1,\alpha_2}^{(L1O)}(T_n,x_0)$ is typically close to the nominal level $\alpha_2-\alpha_1$, that is, $PI_{\alpha_1,\alpha_2}^{(L1O)}$ guarantees (approximate) training conditional validity, provided some regularity conditions on the data generating model and the predictor are satisfied. Note that the analogous procedure based on ordinary residuals $y_i - M_{n,p}(T_n,x_i)$ instead of leave-one-out residuals would, in general, not be valid in such a large-$p$ scenario \citep[cf.][]{Bickel83}. Extending these results to the general $k$-fold cross-validation case relies on a generalization of a result by \citet{Bousquet02} on the estimation of the test error of a learning algorithm by its empirical leave-one-out error, which might be of independent interest (see Lemma~\ref{lemma:BousquetCV} in the supplement).

Despite the remarkable simplicity of this method, and its apparent similarity to the Jackknife, we are not aware of any rigorous analysis of its statistical properties. Notable exceptions are \citet{Steinb16d, Steinb21}, which are precursors of this paper, and \citet{Barber19b}; the latter will be discussed in Subsection~\ref{sec:JK+}. Our approach is similar, in spirit, to the methods proposed in \citet{ButlerRoth80}, \citet{Stine85}, \citet{Schmoyer92}, \citet{Olive07} and \citet{Politis13}, in the sense that it relies on resampling and leave-one-out ideas for predictive inference. But the methods from these references, like most resampling procedures in the literature, are investigated only in the classical large sample asymptotic regime where the number of available explanatory variables is fixed. Prominent exceptions are \citet{Bickel83}, \citet{Mammen96} and, more recently, \citet{ElKaroui15}. These articles draw mainly negative conclusions about resampling methods in high dimensions, arguing, for instance, that the famous residual bootstrap in linear regression, which relies on the consistent estimation of the true unknown error distribution, is unreliable when the number of variables in the model is not small compared to sample size. In contrast, we show that under mild conditions the leave-one-out prediction interval $PI_{\alpha_1,\alpha_2}^{(L1O)}$ does not suffer from these problems because it relies on estimation of the conditional distribution of the prediction error $\P(y_0-\hat{y}_0\le t\| T_n)$ instead of an estimator for the unconditional distribution of the error term $y_0-\E[y_0\|x_0]$. That the use of leave-one-out residuals leads to more reliable methods in high dimensions was also observed by \citet{ElKaroui15}.

Our contribution is threefold. First, we show that the leave-one-out prediction interval (and its extension to $k$-fold CV) is approximately training conditionally valid given the training sample $T_n$, that is
$$
\P\left( y_0 \in PI_{\alpha_1, \alpha_2}^{(L1O)}(T_n,x_0)\Big\| T_n\right)\; \approx\; \alpha_2-\alpha_1.
$$
The error term of the above approximation can be controlled in finite samples and asymptotically, provided that the employed prediction algorithm $M_{n,p}$ is sufficiently stable under the omission of feature-response pairs and that it has a bounded (in probability) estimation error as an estimator for the true unknown regression function. It is of paramount importance, however, to point out that we do not need to assume consistent estimability of the regression function and our leading examples are such that consistency fails.

Second, we show that the required stability and approximation properties are satisfied in many cases, including many linear predictors in high dimensional regression problems and even if the true model is not exactly linear. In particular, the proposed method is always valid if the employed predictor is consistent for the unknown regression function (or for an appropriate surrogate target), and is therefore applicable to complex data structures and methods such as non-parametric regression, LASSO prediction, random forests or deep learning (see Section~\ref{sec:consistency}).

Third, we discuss issues of interval length and find that in typical situations predictors with smaller mean squared prediction error lead to shorter prediction intervals. For ordinary least squares prediction, we also investigate the impact of the dimensionality of the regression problem on the interval length and discuss the relationship between the leave-one-out method and an obvious sample splitting technique.
All our results hold uniformly over large classes of data generating processes and under weak assumptions on the unknown error distribution (e.g., the errors may be heavy tailed and non-symmetric, and the standardized design vectors $\Cov[x_i]^{-1/2}x_i$ may have dependent components and a non-spherical distribution).

Our work is greatly inspired by \cite{ElKaroui13b} and \citet{Bean13} \citep[see also][]{ElKaroui13, ElKaroui18}, who investigate efficiency of general $M$-estimators in linear regression when the number of regressors $p$ is of the same order of magnitude as sample size $n$. In particular, the $M$-estimators studied in these references provide one leading example of a class of linear predictors for which our construction of prediction intervals leads to training conditionally valid predictive inference even in high dimensions.

The remainder of the paper is organized as follows. In the following Subsection~\ref{sec:related} we give a brief overview of alternative methods from the large body of literature on predictive inference in regression. Subsection~\ref{sec:notation} introduces the notation that is used throughout the paper. Sections~\ref{sec:main} and \ref{sec:linpred} proceed along a general-to-specific scheme. We begin, in Subsection~\ref{sec:kCVPI}, by introducing the general cross validation method and the notion of training conditional validity and we relate this notion to estimation of the conditional prediction error distribution in Kolmogorov distance. In Subsection~\ref{sec:stability}, we draw the connection between training conditional validity and algorithmic stability and present our main results which provide sufficient conditions for training conditional validity. In Section~\ref{sec:linpred} we then show that these conditions can be verified in challenging statistical scenarios where regression function estimators and the bootstrap usually fail to be consistent. In particular, we consider linear predictors based on regularized $M$-estimators and based on James-Stein-type estimators in a situation where the number of regressors $p$ is not small relative to sample size $n$. We also take a closer look at the ordinary least squares estimator, because its simplicity allows for a rigorous analysis of the resulting interval length. In Section~\ref{sec:consistency}, we then also discuss the important case where the employed predictor is consistent (possibly for some pseudo target rather than the true regression function) and we provide examples on non-parametric regression, high-dimensional LASSO, random forests and deep neural networks. The case of consistency is an important test case for our method. Finally, we present the results of an extensive simulation study in Section~\ref{sec:mc}. Further discussions and possible extensions of our results as well as most of the technical proofs are deferred to the supplementary material~\citet{Steinb22supp}.

\subsection{Related work}
\label{sec:related}

In a fully parametric setting, predictive inference is essentially a special case of parametric inference \citep[see, e.g.,][Subsection 7.5]{Cox74}. Constructing valid prediction sets becomes much more challenging, however, if one is interested in a non-parametric setting. By non-parametric, we do not only mean that the regression function can not be indexed by a finite dimensional Euclidean space, but also that the random fluctuations $y_i-\E[y_i\|x_i]$ about the conditional mean function can not be described by a parametric family of distributions. 

\subsubsection{Tolerance regions}

A rather well researched and classical topic in the statistics literature is the construction of so called tolerance regions or tolerance limits, which are closely related to prediction regions. A tolerance region is a set valued estimate $TR(T_n)\subseteq\R^m$ based on i.i.d. $m$-variate data $z_1,\dots, z_n$, $T_n = (z_1,\dots, z_n)$, such that the probability of covering an independent copy $z_0$ is close to a prescribed confidence level. More precisely, a $(1-\alpha, \rho)$ tolerance region $TR$ is such that $\P(\P(z_0\in TR\|T_n)\ge 1-\alpha)=\rho$, and $TR$ is called a $(1-\alpha)$-expectation tolerance region, if $\E[\P(z_0\in TR\| T_n)] = \P(z_0\in TR) = 1-\alpha$ \citep[cf.][]{Krish09}. The study of non-parametric tolerance regions goes back at least to \citet{Wilks41, Wilks42}, \citet{Wald43} and \citet{Tukey47} \citep[see][for an overview and further references]{Krish09} and is traditionally based on the theory of order statistics of i.i.d. data. These researchers already obtained multivariate distribution-free methods, that is, tolerance regions that achieve a certain type of validity in finite samples without imposing parametric assumptions.
The connection to prediction regions is apparent: If $z_i=(x_i,y_i)$, then a tolerance region $TR(T_n)$ for $z_0 = (x_0,y_0)$ can be immediately used to obtain a prediction region for $y_0$ by setting $PR(T_n,x_0) = \{y: (x_0,y)\in TR(T_n)\}$. However, this is arguably not the most economical way of constructing a prediction region. In fact, the construction of a multivariate and possibly high-dimensional tolerance region appears to be a more ambitious goal than the construction of a prediction region for a univariate response variable. In particular, since estimation of the full density of $z_0$ -- which could be used to compute an optimal highest density region -- is usually not feasible if the dimension $m$ is non-negligible compared to sample size $n$, one has to specify a shape for the tolerance region $TR$ and it is not obvious which shapes are preferable in a non-parametric setting. For example, \citet{Bucc01} provide results for smallest possible hyperrectangles and ellipsoids, but obtain only the classical large sample asymptotic results with fixed dimension. \citet{Chatterjee80} estimate the density non-parametrically, which fails in high dimensions. \citet{Li08} use a notion of data depth to avoid the specification of the shape, but the fully data driven method, again, is only shown to be valid asymptotically, with the dimension fixed. 
Finally, numerically computing the $x_0$-cut of $TR$ to obtain $PR$ is computationally demanding and the result is sensitive to the shape of $TR$.

\subsubsection{Conformal prediction}

A strand of literature which has emerged from the early ideas of non-parametric tolerance regions, but which is more prominent within the machine learning community than the statistics community, is called conformal prediction \citep[][]{Vovk99, Vovk05, Vovk09}. Conformal prediction is a very flexible general framework for construction of prediction regions that can be used in conjunction with any learning algorithm. The general idea is to construct a pivotal $p$-value $\pi(y_\star)$ to test $H_\star: y_0=y_\star$ based on the sample $T_n$ and $x_0$, for each possible value $y_\star$ of $y_0$, and to invert the test to obtain a prediction region for $y_0$, i.e., $PR = \{y: \pi(y)\ge \alpha\}$. The method was primarily designed for an on-line learning setup \citep[cf.][]{Vovk09}, but has recently been popularized in the statistics community by \citet{Lei13, Lei16} and \citet{Lei14}, who study it as a batch method. Aside from their flexibility, conformal prediction methods have the advantage that they are valid in finite samples, in the sense that the unconditional coverage probability $\P(y_0\in PR)$ is no less than the nominal level $1-\alpha$, provided only that the feature-response pairs $(x_0,y_0), (x_1,y_1), \dots, (x_n,y_n)$ are exchangeable. On the other hand, their practical implementation is not so straight forward, because, for the test inversion, the $p$-value $\pi$ has to be evaluated on a grid of possible $y$ values, which is especially tricky if the conformal prediction region is not an interval \citep[see][for further discussion of these issues]{Chen17,Lei17}. Moreover, it is not clear if the classical conformal methods can also provide training conditional validity. In \citet{Vovk12}, a version of conformal prediction was presented that achieves also a certain type of (approximate) training conditional validity. However, the method relies on a sample splitting idea, which usually makes the prediction region unnecessarily wide (see Subsection~\ref{sec:SS} as well as Subsection~\ref{sec:naive} in the supplement for further discussion of sample splitting techniques). A different version of conditional validity (conditioning on $x_0$), is discussed in \citet{Barber19} (see also Remark~\ref{rem:ObjValid} in the supplement).

\subsubsection{The jackknife+}
\label{sec:JK+}
\citet{Barber19b} recently proposed a modification of the leave-one-out method considered here. For the modified method, which they call jackknife+, they derived a finite-sample lower bound for the unconditional coverage probability, under the assumption that the feature-response pairs are exchangeable and without requiring that the prediction algorithm is stable. For a jackknife+ interval with nominal coverage probability $1-\alpha$, the lower bound is $1-2\alpha$; if the prediction algorithm is stable (under omission of a single feature-response pair), the lower bound moves closer towards $1-\alpha$ (provided that the interval is slightly modified further). In simulations and data-examples, \citet{Barber19b} found that the jackknife+ performs essentially like the jackknife, i.e., like the method considered in this paper, unless the prediction algorithm is highly unstable; in the unstable case, jackknife+ outperforms jackknife. The use of an unstable prediction method is, of course, debatable, at least if the user is aware of the instability. The training conditional performance of the jackknife+ interval, that is, its coverage probability conditional on the training data, is yet to be analyzed.

\subsection{Preliminaries and notation}
\label{sec:notation}

For $p\in\N$, let $\mathcal Y \subseteq \R$ and $\mathcal X\subseteq \R^p$ be Borel measurable sets and let $\mathcal Z= \mathcal X\times \mathcal Y$. For $n\in\N$, $n\ge 2$, a realization of a training sample is denoted by $T_n := (z_i)_{i=1}^n\in\mathcal Z^n$, where $z_i = (x_i,y_i)\in\mathcal Z = \mathcal X\times \mathcal Y$ is a feature-response pair. The realizations of feature-response pairs in the prediction period are denoted by $z_0 = (x_0,y_0)\in\mathcal Z$. All the randomness we consider comes from drawing training samples $T_n$ and $z_0$, and we assume that all feature-response pairs are drawn independently and identically distributed from some probability distribution $P$ belonging to a class $\mathcal P$ of Borel probability measures on $\mathcal Z$. By $\P = \P(P)$ we denote the joint product measure of all occurring feature-response pairs with marginal distribution $P\in\mathcal P$ and we write $\E$ for the corresponding expectation operator. Thus, strictly speaking, $\P$ and $\E$ can change their meaning from line to line if the number of feature-response pairs changes. We use expressions like $\P_{T_n}$, $\P_{z_1}$, $P_{x_0}$, $\E_{z_0}$, etc., to denote integration only with respect to the indicated variables (or marginal distributions). 
By $P_{y_0\|x_0}$ we denote a regular conditional distribution of $y_0$ given $x_0$ under $P$.
By $\mu_P(x) := \E_P[y_0\|x_0=x] := \int_{\mathcal Y} y\,P_{y_0\|x_0}(dy\|x)$, $\mu_P:\mathcal X\to\R$, we denote (a version of) the true unknown regression function, if it exists. 
We sometimes express the training data $T_n$ as $(X, Y)$, where $X = [x_1,\dots, x_n]'\in\R^{n\times p}$ and $Y = (y_1,\dots, y_n)'\in\R^n$. Moreover, $X'$ denotes the transpose of $X$, and we write $(X'X)^\dagger$ for the Moore-Penrose inverse of $X'X$. Similarly, we write $X_{[i]} = [x_1,\dots, x_{i-1},x_{i+1},\dots, x_n]'$ and $Y_{[i]} = (y_1,\dots, y_{i-1},y_{i+1},\dots, y_n)'$ for the corresponding leave-one-out quantities.

Next, we formally define the notion of a (learning) algorithm, of a predictor (or estimator) $\hat{\mu}_n$ and of the cross validation residuals. For $n,p\in\N$, write $[n]:=\{1,\dots, n\}$ and consider a collection of measurable functions $(M_{n',p})_{n'\in[n]}$, $M_{n',p} : \mathcal Z^{n'} \times \mathcal X \to \R$. The collection $(M_{n',p})_{n'\in[n]}$ is called a \emph{learning algorithm}. For $k\in[n]$, let $K_1,\dots, K_k\subseteq[n]$ be a partition of $[n]$ and abbreviate $m_l := |K_l|$, $l\in[k]$. For a vector $x\in\mathcal X$, we set $\hat{\mu}_n(x) := M_{n,p}(T_n,x)$ and $\hat{\mu}_n^{(l)}(x) := M_{n-m_l,p}(T_n\setminus K_l,x)$, where $T_n\setminus K_l := (z_j)_{j\notin K_l}$ denotes the reduced training sample where the observations with indices in $K_l$ have been excluded. We call $\hat{\mu}_n:\mathcal X\to\R$ a predictor or estimator and say that it, or the underlying learning algorithm, is \emph{symmetric} if for every $n'\in[n]$, for every choice of $z_1,\dots, z_{n'}\in\mathcal Z$, every $x\in\mathcal X$ and every permutation $\pi$ of $n'$ elements, $M_{n',p}((z_i)_{i=1}^{n'},x) = M_{n',p}((z_{\pi(i)})_{i=1}^{n'},x)$.\footnote{Some of the predictors that appear in this paper happen to be symmetric, but symmetry is actually never imposed as a condition in our general theory. We mention this concept only because it can sometimes simplify things and also appears in the literature.} Finally, we define $k$-fold cross validation ($k$-CV) residuals for $l\in[k]$ and $i\in K_l$, by $\hat{u}_i := y_i - \hat{\mu}^{(l)}_n(x_i)$.

If $f:D\to\R$ is a real valued function on some domain $D$, then $\|f\|_\infty = \sup_{s\in D}|f(s)|$. For $a,b\in\R$, we also write $a\lor b = \max(a,b)$, $a\land b = \min(a,b)$ and $a_+ = a\lor 0$, and let $\lceil \delta\rceil$ denote the smallest integer no less than $\delta\in\R$. If $F:\R\to[0,1]$ is a cumulative distribution function, we write $\alpha\mapsto F^\dagger(\alpha) := \inf\{s\in\R: \alpha\le F(s)\}$ for the corresponding quantile function. We write $U\stackrel{\mathcal L}{=} V$, if the random quantities $U$ and $V$ are equal in distribution and the underlying probability spaces are clear from the context. By a slight abuse of notation, we also write $U\stackrel{\mathcal L}{=} \mathcal L_0$ if the random variable $U$ is distributed according to the probability law $\mathcal L_0$ and, again, the underlying probability space is clear from the context.

For our asymptotic statements, we will also adopt the following array setup: Unless stated otherwise, all quantities introduced above are allowed to also depend on sample size $n$. We consider sequences $(p_n)_{n\in\N}$ and $(k_n)_{n\in\N}$ of positive integers and a collection of data generating distributions $(P_n)_{n\in\N}$, where $P_n$ is a probability measure on $\mathcal Z_n = \mathcal X_n\times\mathcal Y_n\subseteq\R^{p_n+1}$. We use symbols like $\P_n$, $\E_n$, $x_{0,n}$, $y_{0,n}$, etc., to emphasize dependence on $n$. For a collection of measurable functions $\phi_n:\mathcal Z_n^{n+1} \to \R$, we say that $\phi_n$ is $P_n$-bounded in probability if $\limsup_{n\to\infty} \P_n(|\phi_n|>M) \to 0$, as $M\to\infty$, and write $\phi_n = O_{P_n}(1)$. If $\P_n(|\phi_n|>\eps) \to 0$, as $n\to\infty$, for every $\eps>0$, then we say that $\phi_n$ converges in $P_n$-probability to zero and write $\phi_n = o_{P_n}(1)$. Similarly, we say that $\phi_n$ converges in $P_n$-probability to $\psi_n:\mathcal Z_n^{n+1}\to\R$, which is also assumed to be measurable, if $|\phi_n-\psi_n| = o_{P_n}(1)$. Since the probability measure $P_n$ will be chosen arbitrarily from some class $\mathcal P_n$ of probability measures on $\mathcal Z_n$, our asymptotic results will be uniform over the respective classes.


\section{Main results}
\label{sec:main}
\subsection{$k$-CV prediction intervals and training conditional validity}
\label{sec:kCVPI}

For $\alpha\in(0,1)$, we want to construct a prediction interval $PI_{\alpha}(T_n,x_0) = (\hat{\mu}_n(x_0) + L_\alpha(T_n), \hat{\mu}_n(x_0) + U_\alpha(T_n)]$ for $y_0$, where $L_\alpha$ and $U_\alpha$ are measurable functions on $\mathcal Z^n$, such that
\begin{align}\label{eq:valid}
\E \left[ 
	\left| \P \left( y_0 \in PI_{\alpha}(T_n,x_0) \Big\| T_n \right) - (1-\alpha) \right| \right]
\end{align}
is small uniformly over a class $\mathcal P$ of data generating distributions $P$ for the feature-response pairs $(x_i,y_i)$ and $(x_0,y_0)$.
We can not expect the expression in \eqref{eq:valid} to be equal to zero for some fixed $n$ and a reasonably large class $\mathcal P$ (see Remark~\ref{rem:no-exact-validity} in the supplement). Therefore, we are content with \eqref{eq:valid} being close to zero as $n$, and possibly also $p$, is large. A similar but slightly different notion of training conditional validity is studied by \citet{Vovk13}, and is closely related to the conventional notion of a $(1-\alpha, \rho)$ tolerance region for $\rho$ close to $1$ \citep[cf.][]{Krish09}. However, these conventional definitions require only that the conditional coverage probability $\P(y_0\in PI_\alpha(T_n,x_0)\|T_n) = P(y_0\in PI_\alpha(T_n,x_0))$ is no less than the prescribed confidence level $1-\alpha$, with high probability, whereas the requirement that \eqref{eq:valid} is small also excludes overly conservative procedures. Nevertheless, we also refer to our requirement in \eqref{eq:valid} as \emph{training conditional validity}. Note that if \eqref{eq:valid} is small, then also
\begin{align*}
&\left| \P \left( y_0 \in PI_{\alpha}(T_n,x_0) \right) - (1-\alpha) \right|\\
&\quad\quad=
 \left| \E\left[ \P \left( y_0 \in PI_{\alpha}(T_n,x_0) \Big\| T_n \right) - (1-\alpha) \right]\right|\\
 &\quad\quad\le
\E\left[ \left|  \P \left( y_0 \in PI_{\alpha}(T_n,x_0) \Big\| T_n \right) - (1-\alpha) \right| \right]
\end{align*}
will be small. Hence, the prediction interval is then also approximately unconditionally valid, uniformly over $P\in\mathcal P$.

If the conditional distribution function $s\mapsto\tilde{F}_n(s) := \P(y_0 - \hat{\mu}_n(x_0) \le s\| T_n)$ is continuous, then, for $0\le \alpha_1 < \alpha_2\le 1$ fixed, there is an optimal shortest but infeasible interval 
\begin{align}
PI_{\alpha_1,\alpha_2}^{(OPT)} = [\hat{\mu}_n(x_0) + \tilde{q}_{\alpha_1}, \hat{\mu}_n(x_0) + \tilde{q}_{\alpha_2}] \label{eq:OPTPI}
\end{align} 
in the set of all prediction intervals $PI$ of the form $PI = PI(T_n, x_0) = [\hat{\mu}_n(x_0) + L(T_n), \hat{\mu}_n(x_0) + U(T_n)]$ which also satisfy 
\begin{align}
\P\left(y_0 \le \inf PI\Big\| T_n \right) &= \alpha_1, \quad\text{and}\label{eq:alpha1}\\
\P\left(y_0 \ge \sup PI\Big\| T_n \right) &= 1 - \alpha_2: \label{eq:alpha2}
\end{align}
Simply choose $\tilde{q}_{\alpha_1}$ to be the largest $\alpha_1$-quantile of $\tilde{F}$ and $\tilde{q}_{\alpha_2}$ to be the smallest $\alpha_2$-quantile of $\tilde{F}_n$. This gives the user the flexibility to choose precisely what error probability of under- and over-prediction she is willing to accept. This also shows that what we are really interested in is the unknown conditional distribution of the prediction error $y_0 - \hat{\mu}_n(x_0)$.
Thus, for $PI_{\alpha_1,\alpha_2}^{(OPT)}$, \eqref{eq:valid} is actually equal to zero (for $\alpha_1 + 1-\alpha_2=\alpha$), at least if $\mathcal P$ contains only probability distributions on $\mathcal Z$ for which $\tilde{F}_n:\R\to[0,1]$ is almost surely continuous. 

We propose the following simple cross validation idea to approximate the optimal infeasible procedure: 
Consider the weighted empirical distribution function 
\begin{align}\label{eq:Fhat}
\hat{F}_n(s) := \hat{F}_n(s; T_n) := \sum_{l=1}^k\sum_{i\in K_l} \frac{1}{k m_l} \mathbbm{1}_{(-\infty,s]}(\hat{u}_i)
\end{align}
of the $k$-CV residuals $\hat{u}_i = y_i - \hat{\mu}_n^{(l)}(x_i)$, $i\in K_l$, where $\hat{\mu}_n^{(l)}(x_i) = M_{n-m_l,p}(T_n\setminus K_l,x_i)$ is the prediction of the learning algorithm at $x_i$ when all observations from the fold $K_l$ are removed from the training data $T_n$.
For $\alpha\in[0,1]$, let $\hat{q}_{\alpha} := \hat{F}_n^\dagger(\alpha)$ denote the empirical $\alpha$-quantile. Then the $k$-CV prediction interval is given by
\begin{align}\label{eq:kCVPI}
PI_{\alpha_1,\alpha_2}^{(kCV)}(T_n,x_0) \;=\; \hat{\mu}_n(x_0) \;+\; \Big[ \hat{q}_{\alpha_1}, \;\hat{q}_{\alpha_2}\Big].
\end{align}

The idea behind the $k$-CV procedure is remarkably simple. To estimate the conditional distribution $\tilde{F}_n$  of the prediction error $y_0 - \hat{\mu}_n(x_0)$, i.e., 
\begin{align}\label{eq:Ftilde}
\tilde{F}_n(s) := \tilde{F}_n(s; T_n) := \P(y_0 - \hat{\mu}_n(x_0)\le s\|T_n),\quad s\in\R,
\end{align}
we simply use the (weighted) empirical distribution $\hat{F}_n$ of the $k$-CV residuals $\hat{u}_i = y_i -\hat{\mu}_n^{(l)}(x_i)$, $i\in K_l$. Notice that $\hat{\mu}_n$ is independent of $(x_0,y_0)$, and $\hat{\mu}_n^{(l)}$ is independent of $(x_i,y_i)$, since $i\in K_l$, and thus $\hat{u}_i$ has almost the same distribution as the prediction error, except that $\hat{\mu}_n^{(l)}$ is calculated from $m_l=|K_l|$ observations less than $\hat{\mu}_n$. In many cases this difference turns out to be negligible if $n$ is large, even if $p$ is relatively large too, provided that $|K_l|$ is chosen to be sufficiently small. Note, however, that the $k$-CV residuals $(\hat{u}_i)_{i=1}^n$ are not independent, which leads to substantial technical challenges.

The following result shows that, indeed, a sufficient condition for asymptotic training conditional validity \eqref{eq:valid} of the $k$-CV prediction interval in \eqref{eq:kCVPI} is consistent estimation of $\tilde{F}_n$ in Kolmogorov distance. 

\begin{proposition}\label{prop:basic}
Fix $0\le\alpha_1<\alpha_2\le1$, a training sample $T_n\in\mathcal Z^n$, an arbitrary cumulative distribution function $\hat{F}$ (possibly depending on $T_n$) and write $s_{max}(\hat{F}) := \sup_{s\in\R} [\hat{F}(s) - \hat{F}(s^-)]$ for the size of the largest jump of $\hat{F}$. Define the interval
\begin{equation}\label{eq:PI}
PI_{\alpha_1,\alpha_2}(T_n,x_0) := \hat{\mu}_n(x_0) + [\hat{q}_{\alpha_1}, \hat{q}_{\alpha_2}],
\end{equation}
where $\hat{q}_{\alpha}:=\hat{F}^\dagger(\alpha)$. Let $\tilde{F}_n$ be as in \eqref{eq:Ftilde}.

\begin{enumerate}[i)]
\item\label{prop:basic.i} We have
\begin{align*}
\left| P_{(x_0,y_0)} \left( y_0 \in PI_{\alpha_1,\alpha_2}(T_n,x_0) \right) \;-\; (\alpha_2-\alpha_1)\right| \le
2\| \tilde{F}_n-\hat{F}\|_\infty + 2s_{max}(\hat{F}).
\end{align*}

\item\label{prop:basic.ii} If the cdf $s\mapsto \tilde{F}_n(s;T_n)$ is continuous, then the prediction interval defined in \eqref{eq:PI} satisfies
\begin{align*}
\left| P_{(x_0,y_0)} \left( y_0 \in PI_{\alpha_1,\alpha_2}(T_n,x_0) \right) \;-\; (\alpha_2-\alpha_1)\right| \quad\le\quad
4\|\tilde{F}_n-\hat{F}\|_\infty.
\end{align*}
\end{enumerate}
\end{proposition}

\begin{remark}\normalfont
Note that the inequalities of Proposition~\ref{prop:basic} are purely algebraic statements for a fixed training set $T_n$. Also note that the coverage probability $P_{(x_0,y_0)}( y_0 \in PI_{\alpha_1,\alpha_2}(T_n,x_0) )$ is a version of the conditional probability $\P( y_0 \in PI_{\alpha_1,\alpha_2}(T_n,x_0)\| T_n)$.
\end{remark}

\begin{proof}[Proof of Proposition~\ref{prop:basic}]
For arbitrary $\alpha\in[0,1]$, we have $\alpha \le \hat{F}\circ \hat{F}^\dagger(\alpha) \le \alpha + s_{max}(\hat{F})$ and therefore
\begin{equation}\label{eq:GG}
|\hat{F}\circ \hat{F}^\dagger(\alpha_2) - \hat{F}\circ \hat{F}^\dagger(\alpha_1) - (\alpha_2-\alpha_1)| \le s_{max}(\hat{F}). 
\end{equation}
Now Part~\ref{prop:basic.i}) follows upon noticing that
\begin{align*}
&P_{(x_0,y_0)} \left( y_0 \in PI_{\alpha_1,\alpha_2}(x_0) \right) 
\;=\;
\tilde{F}_n(\hat{q}_{\alpha_2}) - \tilde{F}_n(\hat{q}_{\alpha_1}^-) \\
&\quad\;=\;
\tilde{F}_n(\hat{q}_{\alpha_2}) - \hat{F}(\hat{q}_{\alpha_2})
+ \hat{F}(\hat{q}_{\alpha_1}^-) - \tilde{F}_n(\hat{q}_{\alpha_1}^-) 
+\hat{F}(\hat{q}_{\alpha_2}) - \hat{F}(\hat{q}_{\alpha_1}) + \\
&\quad\quad+\hat{F}(\hat{q}_{\alpha_1}) - \hat{F}(\hat{q}_{\alpha_1}^-),
\end{align*}
and $|\hat{F}(\hat{q}_{\alpha_1}) - \hat{F}(\hat{q}_{\alpha_1}^-)| \le s_{max}(\hat{F})$.

For Part~\ref{prop:basic.ii}) let $s_*\in\R$ denote a point at which the largest jump of $\hat{F}$ occurs, that is, $s_{max}(\hat{F}) = \hat{F}(s_*) - \hat{F}(s_*^-)$. 
Then either $\tilde{F}_n(s_*)\le \hat{F}(s_*) - s_{max}/2$ or $\tilde{F}_n(s_*) > \hat{F}(s_*) - s_{max}/2 = \hat{F}(s_*^-) + s_{max}/2$. The former case implies $\tilde{F}_n(s_*) - \hat{F}(s_*) \le - s_{max}/2$ whereas, by continuity of $\tilde{F}_n$, the latter implies $\tilde{F}_n(s_*^-) - \hat{F}(s_*^-) > s_{max}/2$. Either way, we have $\|\hat{F}-\tilde{F}_n\|_\infty\ge s_{max}/2$.
Using \eqref{eq:GG} again, we obtain
\begin{align*}
\left|\hat{F}(\hat{q}_{\alpha_2}) - \hat{F}(\hat{q}_{\alpha_1}) - (\alpha_2-\alpha_1)\right| 
\le s_{max}(\hat{F}) \le 2\|\hat{F}-\tilde{F}_n\|_\infty,
\end{align*}
which concludes the proof, since for continuous $\tilde{F}$, we have
\begin{align*}
&P_{(x_0,y_0)} \left( y_0 \in PI_{\alpha_1,\alpha_2}(x_0) \right) 
\;=\;
\tilde{F}_n(\hat{q}_{\alpha_2}) - \tilde{F}_n(\hat{q}_{\alpha_1}^-) 
= \tilde{F}_n(\hat{q}_{\alpha_2}) - \tilde{F}_n(\hat{q}_{\alpha_1}) \\
&\quad\;=\;
\tilde{F}_n(\hat{q}_{\alpha_2}) - \hat{F}(\hat{q}_{\alpha_2})
+ \hat{F}(\hat{q}_{\alpha_1}) - \tilde{F}_n(\hat{q}_{\alpha_1}) 
+\hat{F}(\hat{q}_{\alpha_2}) - \hat{F}(\hat{q}_{\alpha_1}).
\end{align*}
\end{proof}

By virtue of Proposition~\ref{prop:basic}, most of what follows will be concerned with the analysis of $\|\tilde{F}_n-\hat{F}_n\|_\infty$.
We are particularly interested in situations where, for a fixed $x\in\mathcal X$, $\hat{\mu}_n(x)$ does not concentrate around $\mu_P(x) = \E[y_0\|x_0=x]$ with high probability but remains random (cf. Remark~\ref{rem:Dicker} in the supplement). In such cases, the unconditional distribution function of the prediction error $\P(y_0-\hat{\mu}_n(x_0)\le s) = \E[\tilde{F}_n(s)]$, the empirical distribution function of the ordinary residuals $s\mapsto\frac{1}{n}\sum_{i=1}^n \mathbbm 1_{(-\infty,s]}(y_i-\hat{\mu}_n(x_i))$ and the true error distribution function $P(y_0-\mu_{P}(x_0)\le s)$ need not be close to one another, because $\hat{\mu}_n$ may not contain enough information about the true regression function $\mu_P$ \citep[see, for instance,][for a linear regression example where $\mu_P(x) = x'\beta_P$]{Bickel83,Bean13}\footnote{It turns out, however, at least in the linear model $\mu_P(x) = x'\beta_P$ and for appropriate estimators of $\beta_P$, that the conditional distribution of the prediction error $\tilde{F}_n$ does concentrate at its mean, i.e., the unconditional distribution, even if $n$ and $p$ are of the same order of magnitude (cf. Subsection~\ref{sec:PIlength} and Lemma~\ref{lemma:UnifWeak} in the proof of Theorem~\ref{thm:PIlength}).}. Nevertheless, we will see that even in such a challenging scenario, it is often possible to consistently estimate the conditional distribution $\tilde{F}_n$ of $y_0 - \hat{\mu}_n(x_0)$, given the training sample $T_n$, by the (weighted) empirical distribution $\hat{F}_n$ of the $k$-CV residuals.

\subsection{The role of algorithmic stability}
\label{sec:stability}

In this section we present general results that relate the uniform estimation error $\|\tilde{F}_n-\hat{F}_n\|_\infty$ to a measure of stability of the underlying predictor $\hat{\mu}_n$. For our first result, sample size $n\ge2$ and dimension $p\ge 1$ are fixed. We only need the following condition on the data generating distribution.

\begin{enumerate}
        \setlength\leftmargin{-20pt}
\setcounter{cond}{\theenumi}       
 
\renewcommand{\theenumi}{(C\arabic{enumi})}
\renewcommand{\labelenumi}{\textbf{\theenumi}} 

\item \label{c.density} 
The conditional distribution $P_{y_0\|x_0}$ has a Lebesgue density $f_{y_0\|x_0}$.
\stepcounter{cond}
\end{enumerate}


Notice that, although the training data $(x_i,y_i)$, $i\in[n]$, are assumed to be realizations of i.i.d. random vectors, Condition~\ref{c.density} allows to model heteroskedasticity, because the conditional variance of the response $y_i$ given $x_i$ is allowed to depend on the value of the feature vector $x_i$.
Building on terminology from \citet{Bousquet02} \citep[see also][]{Devroye79}, we use the following probabilistic notion to quantify algorithmic stability, which is adjusted to the general case of $k$-fold cross validation and which depends on the true data generating distribution $P$ on $\mathcal Z$.

\begin{definition}\label{DEF:STABLE}
The $k$-stability coefficient of the predictor $\hat{\mu}_n$ is defined as
\begin{align*}
\eta_{n,k} := \eta_{n,k}(P) := \frac1k\sum_{l=1}^k \E\left[\left(\|f_{y_0\|x_0}\|_\infty \left|\hat{\mu}_n(x_0) - \hat{\mu}_n^{(l)}(x_0)\right|\right) \land 1\right].
\end{align*} 
\end{definition}

The stability coefficient $\eta_{n,k}$ measures the average change of the prediction of $\hat{\mu}_n$ as the observations of each of the $k$-folds are removed in turn for training. A highly stable algorithm has small stability coefficient. In the extreme case of an algorithm that does not even make use of the training data at all, we have $\eta_{n,k} = 0$. 
Since the feature-response pairs are assumed to be i.i.d. under $\P$, it is easy to see that in case of equal fold sizes $m_l = \frac{n}{k}$, $l\in[k]$, a symmetric predictor has stability coefficient $\eta_{n,k}=\E[(\|f_{y_0\|x_0}\|_\infty|\hat{\mu}_n(x_0) - \hat{\mu}_n^{(1)}(x_0)|)\land 1]$. Also note that a predictor with $\eta_{n,k}=0$ can not depend on the training data in a non-trivial way (cf. Lemma~\ref{lemma:0stable} in the supplement).
We are now in the position to state our main result on the estimation of $\tilde{F}_n(s) = \P(y_0 - \hat{\mu}_n(x_0)\le s \| T_n)$ by $\hat{F}_n(s) = \sum_{l=1}^k\sum_{i\in K_l}\frac{1}{k\cdot m_l}\mathbbm 1_{(-\infty, s]}(\hat{u}_i)$. To that end, we adopt the triangular array setup described in Subsection~\ref{sec:notation}. 

\begin{theorem}\label{thm:densityAsymp}
For $n\in\N$, let $p=p_n$ and $k=k_n$ be positive integers, with $k_n\to\infty$ as $n\to\infty$, and let $P_n$ be a data generating distribution on $\mathcal Z_n = \mathcal X_n\times\mathcal Y_n$, with $\mathcal X_n\subseteq\R^{p_n}$, $\mathcal Y_n\subseteq \R$, that satisfies \ref{c.density}. Furthermore, suppose that there exists a scaling constant $\sigma_n^2\in(0,\infty)$ and a measurable function $g_n:\mathcal X_n\to \R$, such that the following hold true:
\begin{align}
\label{eq:densityBound}
&\sigma_n\|f_{y_{0,n}\|x_{0,n}}\|_\infty = O_{P_n}(1),\\
\label{eq:errorBound}
&\frac{|y_{0,n}-g_n(x_{0,n})|}{\sigma_n} = O_{P_n}(1),\\
\label{eq:estError}
&\frac{|g_{n}(x_{0,n})-\hat{\mu}_n(x_{0,n})|}{\sigma_n} = O_{P_n}(1),\\
\label{eq:etan}
\eta_{n,k_n} = \frac{1}{k_n}\sum_{l=1}^{k_n} &\E_n\left[\left(\|f_{y_{0,n}\|x_{0,n}}\|_\infty \left|\hat{\mu}_n(x_{0,n}) - \hat{\mu}_n^{(l)}(x_{0,n})\right|\right) \land 1\right] \xrightarrow[]{n\to\infty} 0.
\end{align}
Then, the (weighted) empirical cdf $\hat{F}_n$ of the $k_n$-fold cross validation residuals satisfies
$$
\E_n\left[\|\tilde{F}_n-\hat{F}_n\|_\infty \right] \xrightarrow[n\to\infty]{} 0.
$$
Moreover, for $0\le \alpha_1<\alpha_2\le 1$, the $k_n$-fold cross validation interval is asymptotically training conditionally valid, i.e.,
\begin{equation}\label{eq:AsympVal}
\E_n \left[ 
	\left| \P_n \left( y_{0,n} \in PI_{\alpha_1,\alpha_2}^{(k_nCV)}(T_n,x_{0,n}) \Big\| T_n \right) - (\alpha_2-\alpha_1) \right| \right] 
	\xrightarrow[n\to\infty]{}\;0.
\end{equation}
\end{theorem}

\begin{remark}[Uniform training conditional validity]\normalfont
For each $n$, let $p_n$ and $k_n$ be as in Thorem~\ref{thm:densityAsymp}, and let $\mathcal P_n$ be a class of probability measures on $\mathcal Z^n$. If Theorem~\ref{thm:densityAsymp} applies for any sequence $P_n\in\mathcal P_n$, $n\ge1$, then the $k_n$-fold cross validation interval is uniformly asymptotically training conditionally valid, i.e., the left-hand-side of \eqref{eq:AsympVal} converges to zero even when taking the supremum over all $P_n\in\mathcal P_n$. 
\end{remark}

In Theorem~\ref{thm:densityAsymp}, one should think of the scaling constant $\sigma^2_n$ and the function $g_n$ as the error variance and the regression function in a non-linear regression model $y_{0} = g_n(x_{0}) + u_{0}$, with additive error $u_{0}$ that is independent of $x_{0}$ and has variance $\sigma_n^2$. However, they can be taken as appropriate pseudo parameters if needed. Also notice that Condition~\ref{c.density} allows for a heteroskedastic error term.
Theorem~\ref{thm:densityAsymp} is actually a simple consequence of the following finite sample bound.

\begin{theorem}
\label{thm:density}
Fix positive finite constants $\eps, c_1,c_2$ and a measurable function $g:\mathcal X\to\R$. If Condition~\ref{c.density} holds, then
\begin{align*}
\E\left[\|\hat{F}_n - \tilde{F}_n\|_\infty \right]
\;&\le\;
P(|y_0-g(x_0)|\ge c_1) + \P(|\hat{\mu}_n(x_0)-g(x_0)| > c_2)\\
&\quad + \E_{x_0}\left[(\eps\|f_{y_0\|x_0}\|_\infty)\land1\right]\\
&\quad+
\sqrt{\left(\frac{2(c_1+c_2)}{\eps} + 2 \right)\left[\frac{1}{4(k-1)} + 5\eta_{n,k}\right]},
\end{align*}
where $\eta_{n,k}$ is the stability coefficient of the predictor $\hat{\mu}_n$.

\end{theorem}

\begin{proof}[Proof of Theorem~\ref{thm:densityAsymp}]
Set $\nu_n:= \frac{1}{4(k_n-1)} + 5\eta_{n,k_n} = o(1)$ and apply Theorem~\ref{thm:density} with $c_1=c_2=\sigma_n\nu_n^{-1/3}$, $\eps=\sigma_n\nu_n^{1/3}$. For the second claim, note that under \ref{c.density} we have $\P_n(y_{0,n}-\hat{\mu}_n(x_{0,n})=c\| T_n) = \E_{x_{0,n}}[\P_n(y_{0,n}-\hat{\mu}_n(x_{0,n})=c\| T_n, x_{0,n})] = 0$, for all $c\in\R$ and all $T_n\in\mathcal Z_n^n$, which means $s\mapsto \tilde{F}_n(s;T_n)$ is continuous. Now apply Proposition~\ref{prop:basic}.\eqref{prop:basic.ii}.
\end{proof}

Theorem~\ref{thm:density} provides an upper bound on the risk of estimating the conditional prediction error distribution $\tilde{F}_n$ by the (weighted) empirical distribution of the cross validation residuals $\hat{F}_n$. The upper bound crucially relies on the properties of the chosen estimator $\hat{\mu}_n$ for the (pseudo) regression function $g$. If the sample size is sufficiently large and if the estimator is sufficiently stable and has a moderate estimation error, then the parameters $\eps, c_1, c_2$ can be chosen such that the upper bound is small. This is what we do in Theorem~\ref{thm:densityAsymp}. It is important to note that Theorem~\ref{thm:densityAsymp} and Theorem~\ref{thm:density} are informative also in case the estimator $\hat{\mu}_n$ is not consistent for $g$, as is often the case when $p/n \nrightarrow 0$. The bound of Theorem~\ref{thm:density} also exhibits an interesting trade-off between the $\eta$-stability of $\hat{\mu}_n$ and the magnitude of its estimation error. More stable estimators are allowed to be less accurate whereas less stable estimators need to achieve higher accuracy in order to be as reliable for predictive inference purposes as a more stable algorithm.

Theorem~\ref{thm:densityAsymp} and Theorem~\ref{thm:density} show that the $k$-CV prediction interval in \eqref{eq:kCVPI} is approximately uniformly training conditionally valid, i.e., has the property that \eqref{eq:valid} is uniformly small at least for large $n$, provided that the underlying estimator $\hat{\mu}_n$ has two essential properties: First, the $k$-stability coefficient of the estimator must be (uniformly) small if $n$ is large, as in \eqref{eq:etan}.
This is an intuitively appealing assumption since otherwise the $k$-CV residuals $\hat{u}_i = y_i - \hat{\mu}_n^{(l)}(x_i)$ may not be well suited to estimate the distribution of the prediction error $y_0 - \hat{\mu}_n(x_0)$. Second, the scaled estimation error $(\mu_{P}(x_0) - \hat{\mu}_n(x_0))/\sigma_P$ at the new observation $x_0$ must be $P_n$-bounded (cf. \eqref{eq:estError} with $g_n = \mu_{P_n}$).
This is used to guarantee that the conditional distribution $\tilde{F}_n$ of the prediction error $y_0 - \hat{\mu}_n(x_0)$ given the training data $T_n$ is tight in an appropriate sense (cf. Lemma~\ref{lemma:Ftilde}\ref{l:Tightness} in the supplement), so that a pointwise bound on $|\hat{F}_n(t)-\tilde{F}_n(t)|$ can be turned into a uniform bound.
In the following sections we demonstrate that these two conditions on the estimator $\hat{\mu}_n$ are satisfied in several different contexts. From now on, as in Theorem~\ref{thm:densityAsymp}, we will take on an asymptotic point of view.


\section{Linear prediction with many variables}
\label{sec:linpred}

In this section we investigate a scenario in which both consistent parameter estimation as well as bootstrap consistency fail \citep[cf.][]{Bickel83, ElKaroui15}, but the $k$-CV prediction interval is still asymptotically uniformly training conditionally valid. See Section~\ref{sec:consistency} for a discussion of scenarios where consistent parameter estimation is possible. For simplicity, in this section we consider only the leave-one-out case, that is, $k_n = n$. For $\kappa\in[0,1)$, we fix a sequence of positive integers $(p_n)$, such that $p_n/n\to \kappa$ as $n\to\infty$ and $n>p_n+1$ for all $n\in \N$. In case $\kappa>0$, this type of `large $p$, large $n$' asymptotics has the advantage that certain finite sample features of the problem are preserved in the limit, while offering a workable simplification. It turns out that conclusions drawn from this type of asymptotic analyses often provide remarkably accurate descriptions of finite sample phenomena.

We work with linear predictors of the form $\hat{\mu}_n(x_0) = x_0'\hat{\beta}_n$ for various estimators $\hat{\beta}_n$. Suppose the second moment matrix $\Sigma_P = \E_{x_0}[x_0x_0']$ exists. Then the conditions \eqref{eq:etan} and \eqref{eq:estError} can be verified as follows:
For $\eps>0$,
\begin{align*}
&\E\left[\left( \|f_{y_{0}\|x_{0}}\|_\infty |\hat{\mu}_n(x_0) - \hat{\mu}_n^{(1)}(x_0)|\right)\land 1\right] \\
&\quad\le 
\P\left( \frac{|x_0'\hat{\beta}_n - x_0'\hat{\beta}_{n,[1]}|}{\sigma_n} > \eps\right) +
\E_{x_0}\left[\left( \eps \|\sigma_nf_{y_0\|x_0}\|_\infty\right)\land 1\right]\\
&\quad\le 
\E\left[ \left(\frac{1}{\eps^2} \left\|\Sigma_P^{1/2}\left( \hat{\beta}_n-\hat{\beta}_{n,[1]}\right)\right\|_2^2/\sigma_n^2\right)\land 1\right]
+
\E_{x_0}\left[\left( \eps \|\sigma_nf_{y_0\|x_0}\|_\infty\right)\land 1\right]
\end{align*}
where, for the second inequality, we have used the conditional Markov inequality along with independence of $x_0$ and $T_n$ under $\P$, and $\hat{\beta}_{n,[1]}$ denotes the estimator based on the reduced training sample $(X_{[1]},Y_{[1]})$. Thus \eqref{eq:etan} follows if 
\begin{align}
\|\sigma_n f_{y_{0,n}\|x_{0,n}}\|_\infty &= O_{P_n}(1) \quad\text{and} \label{eq:densBound}\\
\left\|\Sigma_{P_n}^{1/2}\left( \hat{\beta}_n-\hat{\beta}_{n,[1]}\right)\right\|_2/\sigma_n^2 &= o_{P_n}(1). \label{eq:etanLin}
\end{align} 
Notice that \eqref{eq:densBound} coincides with \eqref{eq:densityBound}.
By a similar argument, we find that \eqref{eq:estError} with $g_n(x_0) = x_0'\beta_n$ follows if
\begin{align}
\left\|\Sigma_{P_n}^{1/2}\left( \hat{\beta}_n-\beta_n\right)\right\|_2/\sigma_n &= O_{P_n}(1), \label{eq:estErrorLin}
\end{align}
for some vectors $\beta_n\in\R^{p_n}$, and \eqref{eq:errorBound} can be expressed as
\begin{align}\label{eq:errBound}
|y_{0,n} - x_{0,n}'\beta_n|/\sigma_n = O_{P_n}(1).
\end{align}


\subsection{Regularized $M$-estimators}
\label{sec:ElKaroui}

An important class of linear predictors for which our theory on the leave-one-out prediction interval applies are those based on regularized $M$-estimators investigated by \citet{ElKaroui18} in the challenging scenario where $p/n$ is not close to zero \citep[see also][]{ElKaroui13b, Bean13, ElKaroui13}. For a given convex loss function $\rho:\R\to\R$ and a fixed tuning parameter $\gamma\in(0,\infty)$ (both not depending on $n$), consider the estimator
\begin{equation}\label{eq:robust}
\hat{\beta}_n^{(\rho)} := \argmin{b\in\R^{p}} \frac{1}{n} \sum_{i=1}^n \rho(y_i-x_i'b) + \frac{\gamma}{2}\|b\|_2^2.
\end{equation}
In a remarkable tour de force, \citet{ElKaroui18} studied the estimation error $\|\hat{\beta}_n^{(\rho)} - \beta\|_2$ as $p/n\to \kappa\in(0,\infty)$, in a linear model $y_i=x_i'\beta + u_i$, allowing for heavy tailed errors (including the Cauchy distribution) and non-spherical design \citep[see Subsection~2.1 in][for details on the technical assumptions]{ElKaroui18}. In particular, the author shows that $\|\hat{\beta}_n^{(\rho)} - \beta\|_2$ converges in probability to a deterministic positive and finite quantity $r_\rho(\kappa)$ and characterizes the limit through a system of non-linear equations. On the way to this result, \citet[][Theorem~3.9 together with Lemma~3.5 and the ensuing discussion]{ElKaroui18} also establishes the stability property $\|\hat{\beta}_n^{(\rho)} - \hat{\beta}_{n,[1]}^{(\rho)}\|_2\to0$ in probability. Thus, under the assumptions maintained in that reference, \eqref{eq:etanLin}, \eqref{eq:estErrorLin} and \eqref{eq:errBound} hold, and the leave-one-out prediction interval based on the linear predictor $\hat{\mu}_n(x_0)=x_0'\hat{\beta}_n^{(\rho)}$ is asymptotically training conditionally valid, provided that also the boundedness condition \eqref{eq:densBound} is satisfied. 
Finally, we note that a detailed assessment of the predictive performance of $\hat{\beta}_n^{(\rho)}$ in dependence on $\rho$ requires a highly non-trivial analysis of $r_\rho(\kappa)$. For the asymptotic validity of the leave-one-out prediction interval, however, all the information needed on $r_\rho(\kappa)$ is that it is finite.


\subsection{James-Stein type estimators}

Another important example is the class of linear predictors $\hat{\mu}_n(x_0) = x_0'\hat{\beta}_n^{(JS)}$ based on James-Stein type estimators $\hat{\beta}_n^{(JS)}$ defined below. Here, we can allow for the class of data generating processes defined in Condition~\ref{c.non-linmod} below. The feature-response pairs are realizations of a non-Gaussian random design non-linear homoskedastic regression model with regression function $\mu_P$ and error variance $\sigma_P$. Moreover, the feature vectors $x_i$ are allowed to have a complex geometric structure, in the sense that the standardized design vector $\Sigma_P^{-1/2} x_1$ is not necessarily concentrated on a sphere of radius $\sqrt{p_n}$, as would be the case if $\mathcal L_l$ in Condition~\ref{c.non-linmod} was supported on $\{-1,1\}$ (see, e.g., Subsection~3.2 in \citealp[][]{ElKaroui10} and Subsection~2.3.1 in \citealp[][]{ElKaroui18} for further discussion of this point).

\begin{enumerate}
        \setlength\leftmargin{-20pt}
\setcounter{enumi}{\thecond}   
\renewcommand{\theenumi}{(C\arabic{enumi})}
\renewcommand{\labelenumi}{\textbf{\theenumi}} 

\item \label{c.non-linmod} 
	Fix finite constants $C_0>0$ and $c_0>0$ and probability measures $\mathcal L_l$ and $\mathcal L_w$ on $(\R,\mathcal B(\R))$, such that $\mathcal L_w$ has mean zero, unit variance and finite fourth moment, $\int s^2 \mathcal L_l(ds) = 1$ and $\mathcal L_l((-c_0,c_0)) = 0$.

For every $n\in\N$, the class $\mathcal P_n = \mathcal P_n(\mathcal L_l, \mathcal L_w, C_0)$ consists of all probability measures $P$ on $\mathcal Z\subseteq\R^{p_n+1}$, such that the following hold: 
\begin{enumerate}
\item The $x_0$-marginal distribution of $P$ is given by 
$$P_{x_0}\; \stackrel{\mathcal L}{=}\; l_0 \Sigma_P^{1/2}(w_1,\dots, w_{p_n})',$$ 
where $w_1,\dots, w_{p_n}$ are i.i.d. according to $\mathcal L_w$, $l_0\stackrel{\mathcal L}{=} \mathcal L_l$ is independent of the $w_j$ and $\Sigma_P^{1/2}$ is the unique symmetric positive definite square root of a positive definite $p_n\times p_n$ covariance matrix $\Sigma_P$. 

\item The conditional distribution of the response given the regressors is
$$P_{y_0 \|x_0} \;\stackrel{\mathcal L}{=}\; \mu_P(x_0) + \sigma_Pv_{P},$$ 
where $v_P$ is independent of $x_0$ and has mean zero, unit variance and fourth moment bounded by $C_0$, where $\mu_P:\R^{p_n}\to \R$ is some measurable regression function with $\E_{x_0}[\mu_P(x_0)]=0$ and $\sigma_P\in(0,\infty)$.
\end{enumerate}
\stepcounter{cond}
\end{enumerate}

The model $\mathcal P_n$ in \ref{c.non-linmod} is non-parametric, because the regression function $\mu_P$ is unrestricted, up to being centered, and the independent error distribution is arbitrary, up to the requirements $\E_P[v_P] =0$, $\E_P[v_P^2]=1$ and $\E_P[v_P^4]\le C_0$.

To predict the value of $y_0$ from $x_0$ and a training sample $T_n = (x_i,y_i)_{i=1}^n$ with $n\ge p_n+2$, generated from $\P$, we consider linear predictors $\hat{\mu}_n(x_0)=x_0'\hat{\beta}_n(c)$, where $\hat{\beta}_n(c)$ is a James-Stein-type estimator given by
\begin{align*}
\hat{\beta}_n(c) \;=\;\begin{cases}
	 \left( 1 - \frac{cp_n\hat{\sigma}_n^2}{\hat{\beta}_n'X'X\hat{\beta}_n}\right)_+\hat{\beta}_n, \quad&\text{if } \hat{\beta}_n'X'X\hat{\beta}_n >0,\\
	 0, &\text{if } \hat{\beta}_n'X'X\hat{\beta}_n = 0,
 \end{cases}	 
\end{align*}
for a tuning parameter $c\in[0,1]$. Here $\hat{\beta}_n = (X'X)^\dagger X'Y$, $\hat{\sigma}_n^2 = \|Y-X\hat{\beta}_n\|_2^2/(n-p_n)$. The corresponding leave-one-out estimator $\hat{\beta}_n^{[i]}(c)$ is defined equivalently, but with $X$ and $Y$ replaced by $X_{[i]}$ and $Y_{[i]}$. Note that the leave-one-out equivalent of $\hat{\sigma}_n^2 = \hat{\sigma}_n^2(X,Y)$ is given by 
$$
\hat{\sigma}_{n,[i]}^2(X_{[i]},Y_{[i]}) = \hat{\sigma}_{n-1}^2(X_{[i]},Y_{[i]}) = \|Y_{[i]}-X_{[i]}\hat{\beta}_n^{[i]}\|_2^2/(n-1-p_n).
$$ 
The ordinary least squares estimator $\hat{\beta}_n$ belongs to the class of James-Stein estimators. In particular, $\hat{\beta}_n(0) = \hat{\beta}_n$, because, with $P_X:=X(X'X)^\dagger X'$, we have $\|P_XY\|_2^2 = \hat{\beta}_n'X'X\hat{\beta}_n=0$ if, and only if, $Y\in\s(P_X)^\perp=\s(X)^\perp$, and the latter clearly implies $\hat{\beta}_n=0$.

Using James-Stein type estimators for prediction is motivated, e.g., by the optimality results of \citet{Dicker13} and the discussion in \citet{Huber13}. The next result shows that in the model~\ref{c.non-linmod} with $p_n/n\to\kappa\in(0,1)$ and if the deviation from a linear model is not too severe, the James-Stein-type estimators are sufficiently stable and their estimation errors are uniformly bounded in probability, just as required in \eqref{eq:etanLin} and \eqref{eq:estErrorLin}.

\begin{theorem}\label{thm:JSmisspec}
For every $n\in\N$, let $P_n\in \mathcal P_n = \mathcal P_n(\mathcal L_l, \mathcal L_w, C_0)$ as in Condition~\ref{c.non-linmod} and suppose that under $P_n$, the error term $v_{P_n}$ in \ref{c.non-linmod} has a Lebesgue density. Define $\beta_n = \beta_{P_n}$ to be the minimizer of $\beta \mapsto \E_{(x_0,y_0)}[(y_0-\beta'x_0)^2]$ over $\R^{p_n}$. If $p_n/n\to \kappa \in [0,1)$, $c_n\in[0,1]$ for all $n\in\N$, and
\begin{align}\label{eq:JSMisspecBound}
\limsup_{n\to\infty} \E_{x_{0,n}}\left[\left(\frac{\mu_{P_n}(x_{0,n}) - x_{0,n}'\beta_n}{\sigma_{P_n}} \right)^2\right] \;<\;\infty,
\end{align}
then the positive part James-Stein estimator $\hat{\beta}_n(c_n)$ satisfies \eqref{eq:estErrorLin}, i.e.,
$$
\limsup_{n\to\infty} \P_n \left( \left\|\Sigma_{P_n}^{1/2}(\hat{\beta}_n(c_n) - \beta_n)/\sigma_{P_n}\right\|_2 > M \right) \quad\xrightarrow[M\to \infty]{} \quad 0.
$$
If, in addition, $\kappa>0$, then \eqref{eq:etanLin} is also satisfied, i.e., for every $\eps>0$,
$$
\P_n\left( \left\|\Sigma_{P_n}^{1/2}(\hat{\beta}_n(c_n) - \hat{\beta}_n^{[1]}(c_n))/\sigma_{P_n}\right\|_2 > \eps \right) \quad\xrightarrow[n\to \infty]{} \quad 0.
$$
\end{theorem}

\begin{remark} \normalfont
Under the assumptions of Theorem~\ref{thm:JSmisspec}, uniform asymptotic training conditional validity of the leave-one-out prediction interval follows, provided that, in addition, the errors $v_P$ in Condition~\ref{c.non-linmod} have uniformly bounded densities so that \eqref{eq:densBound} holds. To see this, note that \eqref{eq:etanLin} and \eqref{eq:estErrorLin} are conclusions of the theorem, that uniformly bounded fourth moment of the error implies $P_n$-uniform boundedness, such that \eqref{eq:errBound} is a consequence of assumption~\eqref{eq:JSMisspecBound}.
\end{remark}

\begin{remark} \normalfont
The last statement of Theorem~\ref{thm:JSmisspec} can also be established for the case $\kappa=0$ but would require a different proof strategy. Since this case is statistically less interesting we omit it for the sake of brevity.
\end{remark}

\subsection{Ordinary least squares and interval length}
\label{sec:PIlength}

We investigate the special case of the ordinary least squares predictor $\hat{\mu}_n(x) = x'\hat{\beta}_n = x'(X'X)^\dagger X'Y$ in some more detail, because here also the length 
$$
\left|PI_{\alpha_1,\alpha_2}^{(L1O)}\right| = \hat{q}_{\alpha_2} - \hat{q}_{\alpha_1},
$$ 
of the leave-one-out prediction interval (cf. \eqref{eq:kCVPI} with $k=n$) admits a reasonably simple asymptotic characterization. We consider a class $\mathcal P_n^{(lin)} = \mathcal P_n^{(lin)}(\mathcal L_l,\mathcal L_w, \mathcal L_v)$ which is a subset of the one of Condition~\ref{c.non-linmod}, with the additional assumption that the regression function $\mu_P$ is linear and that the error distribution is fixed (up to arbitrary scaling).

\begin{enumerate}
        \setlength\leftmargin{-20pt}
\setcounter{enumi}{\thecond}   
\renewcommand{\theenumi}{(C\arabic{enumi})}
\renewcommand{\labelenumi}{\textbf{\theenumi}} 

\item \label{c.linmod} 
	Fix a finite constant $c_0>0$ and probability measures $\mathcal L_l$, $\mathcal L_w$ and $\mathcal L_v$ on $(\R,\mathcal B(\R))$, such that $\mathcal L_w$ and $\mathcal L_v$ have mean zero, unit variance and finite fourth moment, $\int s^2 \mathcal L_l(ds) = 1$ and $\mathcal L_l((-c_0,c_0)) = 0$.

For every $n\in\N$, the class $\mathcal P_n^{(lin)} = \mathcal P_n^{(lin)}(\mathcal L_l, \mathcal L_w, \mathcal L_v)$ consists of all probability measures $P$ on $\mathcal Z\subseteq \R^{p_n+1}$, such that the following hold: 
\begin{enumerate}
\item The $x_0$-marginal distribution of $P$ is given by 
$$P_{x_0}\; \stackrel{\mathcal L}{=}\; l_0 \Sigma_P^{1/2}(w_1,\dots, w_{p_n})',$$ 
where $w_1,\dots, w_{p_n}$ are i.i.d. according to $\mathcal L_w$, $l_0\stackrel{\mathcal L}{=} \mathcal L_l$ is independent of the $w_j$ and $\Sigma_P^{1/2}$ is the unique symmetric positive definite square root of a positive definite $p_n\times p_n$ covariance matrix $\Sigma_P$. 

\item The conditional distribution of the response given the regressors is
$$P_{y_0 \|x_0} \;\stackrel{\mathcal L}{=}\; x_0'\beta_P + \sigma_Pv_{0},$$ 
where $v_0\stackrel{\mathcal L}{=} \mathcal L_v$ is independent of $x_0$, and where $\beta_P\in\R^{p_n}$ and $\sigma_P\in(0,\infty)$.
\end{enumerate}
\stepcounter{cond}
\end{enumerate}

Note that under \ref{c.linmod}, the distributions $\mathcal L_l$, $\mathcal L_w$ and $\mathcal L_v$ are fixed, so that $\mathcal P_n^{(lin)}$ is a parametric model indexed by $\beta_P$, $\Sigma_P$ and $\sigma_P$. However, these parameters may depend on sample size $n$, and the dimension $p_n$ of $\beta_P\in\R^{p_n}$ and of $\Sigma_P\in \R^{p_n\times p_n}$ may increase with $n$. Subsequently, we aim at uniformity in these parameters. 

\begin{theorem}\label{thm:PIlength}
Fix $\alpha\in[0,1]$. For every $n\in\N$, let $P_n\in \mathcal P_n$ with $\mathcal P_n = \mathcal P_n^{(lin)}(\mathcal L_l, \mathcal L_w, \mathcal L_v)$ as in \ref{c.linmod}. 
If $p_n/n\to \kappa \in (0,1)$ then the scaled empirical $\alpha$-quantile $\hat{q}_\alpha/\sigma_{P_n}$ of the leave-one-out residuals $\hat{u}_i = y_i - x_i'\hat{\beta}_n^{[i]}$ based on the OLS estimator $\hat{\beta}_n = (X'X)^\dagger X'Y$ converges in $P_n$-probability to the corresponding $\alpha$-quantile $q_\alpha$ of the asymptotic distribution of the scaled prediction error $(y_0-x_0'\hat{\beta}_n)/\sigma_{P_n}$. The latter is equal in distribution to
$$l N \tau + v,$$ 
where $l, N, \tau$ and $v$ are defined as follows:
$l\stackrel{\mathcal L}{=} \mathcal L_l$, $N\stackrel{\mathcal L}{=} \mathcal N(0,1)$, and $v\stackrel{\mathcal L}{=}\mathcal L_v$ are independent, and $\tau= \tau(\mathcal L_l,\kappa)$ is non-random. 

The same statement holds also for $\kappa=0$, provided that, in addition, $\mathcal L_v$ has a continuous and strictly increasing cdf and $p_n\to\infty$ as $n\to\infty$.

Here, the function $\kappa \mapsto \tau(\mathcal L_l, \kappa)\in [0,\infty)$ defined on $[0,1)$ has the following properties: For any $\mathcal L_l$ as in \ref{c.linmod}, $\tau(\mathcal L_l, \kappa)=0$ if, and only if, $\kappa=0$. If $\mathcal L_l(\{-1,1\})=1$, then $\tau(\mathcal L_l, \kappa) = \sqrt{\kappa/(1-\kappa)}$.
\end{theorem}

Theorem~\ref{thm:PIlength} shows how the length $\hat{q}_{\alpha_2} - \hat{q}_{\alpha_1}$ of the leave-one-out prediction interval for the OLS predictor depends (asymptotically) on $\mathcal L_l$, $\mathcal L_v$ and $\kappa = \lim_{n\to\infty} p_n/n$. For simplicity, let $\mathcal L_l(\{-1,1\}) = 1$ and consider an equal tailed interval, i.e., $\alpha_1 = \alpha/2 = 1-\alpha_2$. Figure~\ref{fig:PIlengths} shows asymptotic interval lengths as functions of $\kappa\in[0,1]$ for different values of error level $\alpha$ in the cases $\mathcal L_v = \text{Unif}\{-1,1\}$ and $\mathcal L_v = \mathcal N(0,1)$. For a wide range of $\kappa$ values ($\kappa\in[0,0.8]$), the interval length is almost constant. However, for high dimensional problems ($\kappa>0.8$) the interval length increases dramatically, as expected, because here the asymptotic estimation error $\tau=\sqrt{\kappa/(1-\kappa)}$ explodes. We also get an idea about the impact of the error distribution, on which the practitioner has no handle. In particular, for large error levels ($\alpha=0.6$) we even observe a non-monotonic dependence of the interval length on $\kappa$, which seems rather counterintuitive. This results from the non-monotonicity of $\tau^2 \mapsto IQR_\alpha(\mathcal N(0,\tau^2) * \mathcal L_v) = q_{1-\alpha/2}-q_{\alpha/2}$, where $*$ denotes convolution. This non-monotonicity may only occur if the error distribution $\mathcal L_v$ is not log-concave (e.g., the blue curve for $\alpha=0.6$ in Figure~\ref{fig:PIlengths}; cf. the discussion in Subsection~\ref{sec:efficiency} of the supplement). Finally, for large values of $\kappa$, and thus, for large values of $\tau$, the error distribution has little effect on the interval length, because in that case the term $N\tau$ dominates the distribution of $N\tau + v$.

The result of Theorem~\ref{thm:PIlength} can be intuitively understood as follows. If the true model $\mathcal P_n^{(lin)}$ is linear and satisfies \ref{c.linmod} then the scaled prediction error under $\P=P^{n+1}, P\in\mathcal P_n^{(lin)}$, is distributed as
$$
\frac{y_0- x_0'\hat{\beta}_n}{\sigma_P} \stackrel{\mathcal L}{=} l_0 (w_1,\dots, w_{p_n}) \Sigma_P^{1/2}(\beta_P-\hat{\beta}_n)/\sigma_P + v_0,
$$
and for $n$ large, $\|\Sigma_P^{1/2}(\beta_P-\hat{\beta}_n)/\sigma_P\|_2 \approx \tau$ is approximately non-random, so that $(w_1,\dots, w_{p_n}) \Sigma_P^{1/2}(\beta_P-\hat{\beta}_n)/\sigma_P \approx w_0' Z \tau$, where $Z:= \Sigma_P^{1/2}(\beta_P-\hat{\beta}_n)/\|\Sigma_P^{1/2}(\beta_P-\hat{\beta}_n)\|_2$ is a random unit vector that is independent of $w_0:=(w_1,\dots, w_{p_n})'$. Thus, if $p_n$ is large and $Z$ satisfies the Lyapounov condition $\|Z\|_{2+\delta}\to 0$, then $w_0'Z \approx \mathcal N(0,1)$ (see Lemma~\ref{lemma:UnifWeak}\eqref{l:UnifWeakT>0} in the supplement).
This effect of additional Gaussian noise in the prediction error was also observed by \citet{ElKaroui13, ElKaroui13b, ElKaroui15, ElKaroui18}. Note that the conditions $\|\Sigma_P^{1/2}(\beta_P-\hat{\beta}_n)/\sigma_P\|_2 \approx \tau$ and $\|Z\|_{2+\delta}\to0$ need not be satisfied by a given estimator $\hat{\beta}_n$. However, the former condition is always satisfied by the robust $M$-estimators
$$
\hat{\beta}_n^{(\rho)} = \argmin{b\in\R^p} \frac{1}{n}\sum_{i=1}^n \rho(y_i-x_i'b),
$$
considered in \citet{ElKaroui18, ElKaroui13} and under the model assumptions in that reference (cf. Subsection~\ref{sec:ElKaroui}). Here, $\rho:\R\to\R$ is an appropriate convex loss function. 
The Lyapounov condition $\|Z\|_{2+\delta}\to0$ is also satisfied by $\hat{\beta}_n^{(\rho)}$, provided that the standardized design vectors $\Sigma_P^{-1/2}x_i$ follow an orthogonally invariant distribution, because then one easily sees that under~\ref{c.linmod}
$$
\hat{\beta}_n^{(\rho)} \;=\; \beta_P + \Sigma_P^{-1/2}\tilde{\beta}_n^{(\rho)}
\;\stackrel{\mathcal L}{=}\; \beta_P + \|\tilde{\beta}_n^{(\rho)}\|_2 \Sigma_P^{-1/2}U, 
$$
where $\tilde{\beta}_n^{(\rho)} = \argmin{b\in\R^p} \frac{1}{n}\sum_{i=1}^n \rho(\sigma_Pv_i-x_i'\Sigma_P^{-1/2}b)$ and $U$ is uniformly distributed on the unit sphere and independent of $\|\tilde{\beta}_n^{(\rho)}\|_2 = \|\Sigma_P^{1/2}(\beta_P-\hat{\beta}_n^{(\rho)})\|_2$. 
Therefore, $\|Z\|_{2+\delta} \stackrel{\mathcal L}{=} \|U\|_{2+\delta} \stackrel{\mathcal L}{=} \left( \|V\|_{2+\delta}^{2+\delta}/\|V\|_2^{2(1+\delta/2)}\right)^{\frac{1}{2+\delta}}\to 0$, in probability, for $V\thicksim \mathcal N(0,I_{p_n})$, $p_n\to\infty$ as $n\to\infty$.
However, the Lyapounov property is not satisfied, e.g., by the James-Stein estimators (cf. Lemma~\ref{lemma:JSneg} in the supplement). If the mentioned conditions are not satisfied, much more complicated limiting distributions of the prediction error than the one of Theorem~\ref{thm:PIlength} may arise.

\begin{figure}[htbp]
\includegraphics[width=\textwidth]{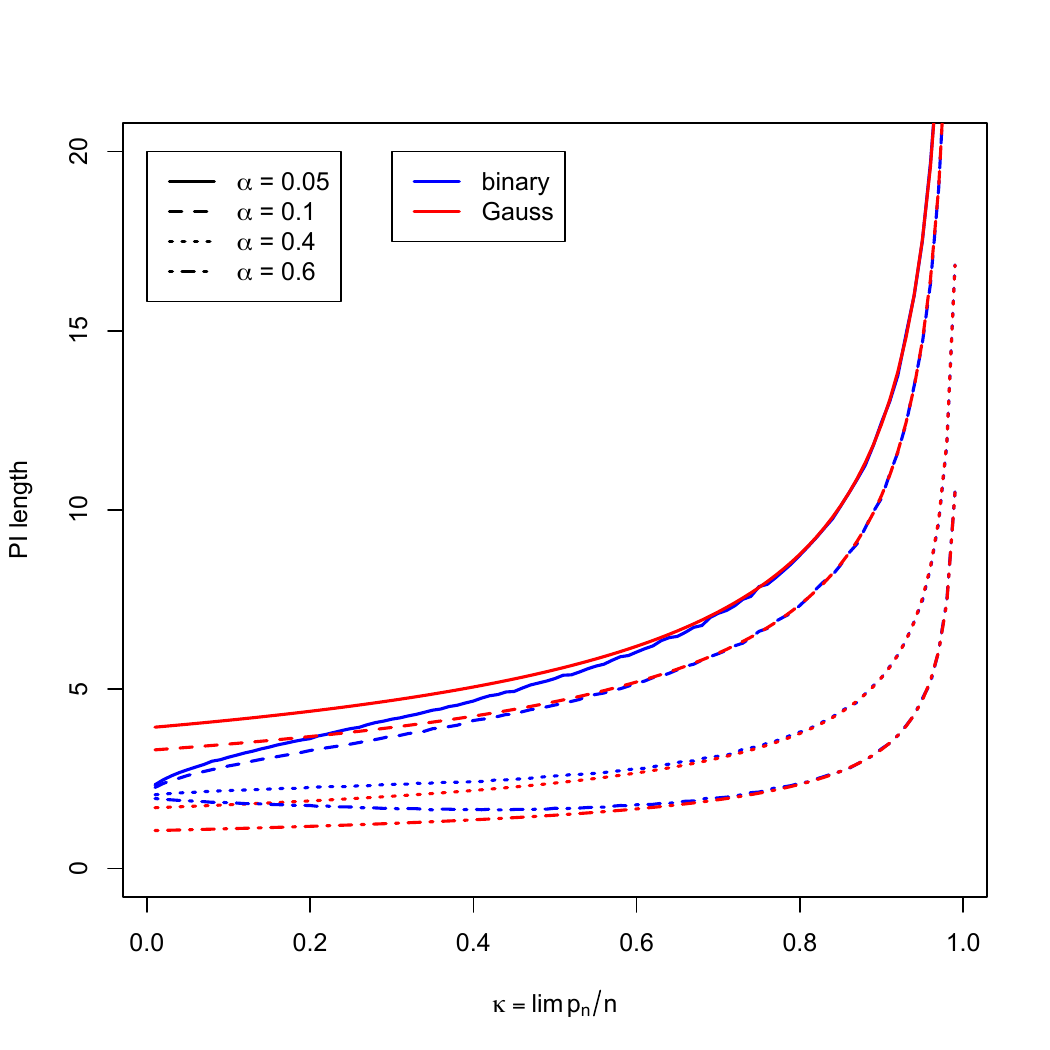} 
\caption{Lengths of leave-one-out prediction intervals as a function of $\kappa = \lim_{n\to\infty} p_n/n$ for confidence level $1-\alpha$ and with $\text{Unif}\{-1,1\}$ (binary) and $\mathcal N(0,1)$ (Gauss) errors.}
\label{fig:PIlengths}
\end{figure}

\subsection{Sample splitting}
\label{sec:SS}

An obvious alternative to the $k$-fold cross validation prediction interval \eqref{eq:kCVPI} is to use a sample splitting method as follows. Decide on a fraction $\nu\in(0,1)$ and use only a number $n_1 = \lceil \nu n\rceil$ of observation pairs $(x_i, y_i)$, $i\in S_\nu\subseteq \{1,\dots, n\}$, $|S_\nu|=n_1$, to compute an estimate $\hat{\mu}_{n_1}$. Now use the remaining $n-n_1$ observations to compute residuals $\hat{u}_i^{(\nu)} = y_i - \hat{\mu}_{n_1}(x_i)$, $i\in\{1,\dots, n\}\setminus S_\nu$. Since, conditionally on the observations corresponding to $S_\nu$, these residuals are i.i.d. and distributed as $y_0-\hat{\mu}_{n_1}(x_0)$, constructing a prediction interval of the form $[\hat{\mu}_{n_1}(x_0) + L, \hat{\mu}_{n_1}(x_0) + U]$ for $y_0$ is now equivalent to constructing a tolerance interval for $y_0-\hat{\mu}_{n_1}(x_0)$ based on i.i.d. observations with the same distribution. One can now simply use appropriate empirical quantiles $L=\hat{q}_{\alpha_1}^{(\nu)}$ and $U=\hat{q}_{\alpha_2}^{(\nu)}$ from the sample splitting residuals $\hat{u}_i^{(\nu)}$ (see also Subsection~\ref{sec:naive} in the supplement). Such a procedure is suggested and studied, e.g., by \citet{Vovk12} and \citet{Lei16}. Notice, however, that this prediction interval is not centered at the more accurate point predictor $\hat{\mu}_n$ that is computed from the full sample of size $n$.

In order to formally study the length of this sample splitting interval we restrict to the case of OLS estimation, i.e., $\hat{\mu}_{n_1}(x) = x'\hat{\beta}_{n_1}$, where $\hat{\beta}_{n_1}$ is the OLS estimator based on the sample corresponding to $S_\nu$. Note that in this case, the estimator will not be unique if $n_1< p_n$, so one usually requires $n_1\ge p_n$. Now, in view of Theorem~\ref{thm:PIlength}, the empirical quantiles of the residuals $\hat{u}_i^{(\nu)}$, $i\in S_\nu^c$, converge (unconditionally) to the quantiles of $l N \tau' + u$, where now $\tau'$ is the non-random limit of $\|\Sigma_P^{1/2}(\beta_P-\hat{\beta}_{n_1})/\sigma_P\|_2$. In particular, if $\mathcal L_l$ degenerates to $\{-1,1\}$, then $\tau' = \sqrt{\kappa'/(1-\kappa')}$, where $\kappa'=\lim_{n\to\infty} p_n/n_1 = \kappa/\nu$. Thus, we can read off the asymptotic interval length of the sample splitting procedure from Figure~\ref{fig:PIlengths} by simply adjusting the value of $\kappa$ to $\kappa/\nu$. For instance, in the binary error case with $\alpha=0.05$, if $\kappa=0.4$ and we use sample splitting with $\nu=1/2$, then $\kappa'=0.8$ and the asymptotic length of the leave-one-out prediction interval is about $4.7$, while the asymptotic length of the sample splitting interval is about $9$, so almost twice as wide.


\section{Asymptotically degenerate (non-random) estimators}
\label{sec:consistency}

Another important class of problems, where the conditions \eqref{eq:estError} and \eqref{eq:etan} of Subsection~\ref{sec:stability} are satisfied, are those where the estimator $\hat{\mu}_n$ (and those computed without the observations in the $l$-th fold $\hat{\mu}_n^{(l)}$) asymptotically degenerates to some non-random function which need not be the true regression function $\mu_P(x_0) = \E_P[y_0\|x_0]$. We point out that in the scenario considered in this section, the naive approach that tries to estimate the true unknown distribution of the errors $u_i = y_i - \mu_P(x_i)$ in the additive error model \ref{c.non-linmod} based on the ordinary residuals $y_i-\hat{\mu}_n(x_i)$ is often successful (asymptotically) for constructing training conditionally valid prediction intervals, provided that consistent estimation of $\mu_P$ is possible. This less challenging but more classical setting of asymptotically non-random predictors is an important test case for the CV method.
We still consider asymptotic results where the number of explanatory variables $p=p_n$ and the number of folds $k=k_n$ can grow with sample size $n$. Thus, we consider sequences $(p_n)_{n\in\N}$ and $(k_n)_{n\in\N}$, and a sequence $(P_n)_{n\in\N}$ of probability measures, where $P_n$ is supported on $\mathcal Z_n \subseteq \R^{p_n+1}$. Moreover, we have to slightly extend the usual definition of uniform consistency of an estimator sequence to cover also the estimates $\hat{\mu}_n^{(l)}$ omitting the $l$-th fold and to allow the possibility of an asymptotically non-vanishing bias.

\begin{definition}[Uniform Asymptotic Degeneracy (UAD)]
For every $n\in\N$, let $k_n, p_n\in\N$, let $P_n$ be a probability measure on $\mathcal Z_n$ and let $\sigma^2_n\in(0,\infty)$. We say that a sequence of predictors $\hat{\mu}_n$ is uniformly asymptotically degenerate (UAD) with respect to $(P_n)_{n\in\N}$ and relative to $(\sigma^2_n)_{n\in\N}$, if there exists measurable functions $g_n:\R^{p_n}\to\R$, such that for every $\eps>0$, 
\begin{align}
&|g_n(x_{0,n}) - \hat{\mu}_n(x_{0,n})| = o_{P_n}\left(\sigma_n\right)\quad\text{and} \label{eq:defCons:a}\\
&\frac{1}{k_n}\sum_{l=1}^{k_n}\E_n\left[\left(\frac{|g_n(x_{0,n}) - \hat{\mu}_n^{(l)}(x_{0,n})|}{\sigma_n}\right)\land 1\right] = o\left(1\right). \label{eq:defCons:b}
\end{align}
\end{definition}

In the classical case of consistent estimation, the function $g_n$ would be the true unknown regression function $g_n(x_{0,n})=\mu_{P_n}(x_{0,n})$ and the constant $\sigma_n^2$ can be thought of as the error variance $\sigma_n^2=\Var_n[y_{0,n}-\mu_{P_n}(x_{0,n})]$, if it exists, or can be chosen equal to $1$. 

It is convenient and common practice to choose the $k$-folds $K_1,\dots, K_k$ of approximately the same size. We consider the following convention: In case $n/k\in\N$, choose all folds of equal size $\frac{n}{k}$ and otherwise choose $\ell:= n-\lfloor \frac{n}{k}\rfloor k \in[k-1]$ and $m_1=m_2=\dots = m_\ell = \lfloor\frac{n}{k}\rfloor + 1$ and $m_{\ell+1}=\dots=m_k = \lfloor\frac{n}{k}\rfloor$. With this convention, the left hand side of \eqref{eq:defCons:b} can be written as
\begin{align*}
&\frac{\ell}{k_n}\E_n\left[\left(\frac{|g_n(x_{0,n}) - \hat{\mu}_n^{(1)}(x_{0,n})|}{\sigma_n}\right)\land 1\right]\\
&\quad\quad+
\frac{k_n-\ell}{k_n}\E_n\left[\left(\frac{|g_n(x_{0,n}) - \hat{\mu}_n^{(k_n)}(x_{0,n})|}{\sigma_n}\right)\land 1\right],
\end{align*}
irrespective of whether $\hat{\mu}_n$ is symmetric or not.
Hence, \eqref{eq:defCons:b} follows if
\begin{equation}\label{eq:defCons:c}
|g_n(x_{0,n}) - \hat{\mu}_n^{(1)}(x_{0,n})| + |g_n(x_{0,n}) - \hat{\mu}_n^{(k_n)}(x_{0,n})| = o_{P_n}\left(\sigma_n\right).
\end{equation}
Thus, under this construction it actually suffices to verify \eqref{eq:defCons:a} and \eqref{eq:defCons:c} to establish UAD. Furthermore, in many cases \eqref{eq:defCons:c} will even be equivalent to \eqref{eq:defCons:a}, provided that $n-m_1$ is growing sufficiently fast with $n$. 

It is easy to see that if $\hat{\mu}_n$ is UAD with respect to $(P_n)_{n\in\N}$ and relative to $(\sigma_n^2)_{n\in\N}$ and if \ref{c.density} holds, then the sequence of stability constants $\eta_{n,k_n}$ satisfies \eqref{eq:etan}, i.e.,
\begin{align*}
\eta_{n,k_n} &= \frac{1}{k_n}\sum_{l=1}^{k_n}\E_n\left[\left( \|f_{y_{0,n}\|x_{0,n}}\|_\infty |\hat{\mu}_n(x_{0,n}) - \hat{\mu}_n^{(l)}(x_{0,n})|\right)\land 1\right] 
\xrightarrow[n\to\infty]{} 0,
\end{align*}
provided that $\sigma_n \|f_{y_{0,n}\|x_{0,n}}\|_\infty = O_{P_n}(1)$. Simply use the inequality $(a+b)\land 1 \le (a\land1)+(b\land 1)$, for $a,b\ge0$. Also note that $\|\sigma_n f_{y_{0,n}\|x_{0,n}}\|_\infty$ is the sup-norm of the conditional density of the scaled error term $(y_{0,n} - g_n(x_{0,n}))/\sigma_n$ given $x_{0,n}$. 

In the remainder of this subsection we list a number of examples where the UAD property of $\hat{\mu}_n$ holds. Therefore (assuming \ref{c.density} and the boundedness conditions of Theorem~\ref{thm:densityAsymp}) also asymptotic training conditional validity of the $k$-CV prediction interval follows. We emphasize that the conditions on the statistical model $\mathcal P$, that are imposed in the subsequent examples, are taken from the respective reference and we do not claim that they are minimal.

\begin{example}[Non-parametric regression estimation]\label{ex:nonparametric}\normalfont
Consider a constant sequence of dimension parameters $p_n=p\in\N$. For positive finite constants $L$ and $C$, let $\mathcal P(L,C)$ denote the class of probability distributions $P$ on $\mathcal Z = \mathcal X\times \mathcal Y\subseteq\R^{p+1}$ such that $P_{y_0}(|y_0|\le L) = 1 = P_{x_0}(\|x_0\|_2\le L)$ and whose corresponding regression function $\mu_P:\R^p\to \R$ is $C$-Lipschitz, i.e., $|\mu_P(x_1) - \mu_P(x_2)| \le C\|x_1-x_2\|_2$ for all $x_1,x_2\in\mathcal X$. \citet[][Chapter~7]{Gyorfi02} show that if $\hat{\mu}_n$ is either an appropriate kernel estimate, a partitioning estimate or a nearest-neighbor estimate, all with fully data driven choice of tuning parameter, then, if $P_n\in\mathcal P(L,C)$,
\begin{align*}
\P_n(|\hat{\mu}_n(x_{0,n}) - \mu_{P_n}(x_{0,n})|>\eps) \xrightarrow[n\to\infty]{} 0
\end{align*}
for every $\eps>0$. Thus \eqref{eq:defCons:a} holds with $g_n=\mu_{P_n}$ and $\sigma_n=1$, and under the (approximate) equal fold size construction mentioned above, \eqref{eq:defCons:c} also holds provided $n-m_1\to\infty$, because here $p_n=p$ is constant, and UAD follows.
\end{example}

\begin{example}[Deep neural networks]\label{ex:DNN}\normalfont
\citet{SchmidtH20} considers a nonparametric regression model similar to \ref{c.non-linmod}, but with $\sigma_P=1$ and $v_P\thicksim \mathcal N(0,1)$. The distribution of the regressors $x_i$ is supported on $\mathcal X=[0,1]^p$ but is otherwise unrestricted. He assumes a smooth compositional structure for the true unknown regression function, that is, $\mu_P \in \mathcal G(q,d,t,\beta,K)$, where
\begin{align*}
\mathcal G(q,d,t,\beta,K) =
\Big\{ 
f = &g_q\circ\dots\circ g_0 \Big| g_i : [a_i,b_i]^{d_i} \to [a_{i+1}, b_{i+1}]^{d_{i+1}}, \\
&\quad g_i\in \mathcal C_{t_i}^{\beta_i}([a_i,b_i]^{t_i},K), \text{ for some } |a_i|, |b_i| \le K
\Big\},
\end{align*}
$q\in\N$, $d = (d_0,\dots, d_{q+1})'\in\N^{q+2}$, $d_0=p$, $d_{q+1}=1$, $t = (t_0,\dots, t_q)'\in\N^{q+1}$, $t_i\le d_i$, for $i\le q+1$, $\beta = (\beta_0,\dots, \beta_q)'\in(0,\infty)^{q+1}$, $\mathcal C_{t_i}^{\beta_i}([a_i,b_i]^{t_i},K)$ is the set of H\"older functions on $[a_i,b_i]^{t_i}$ with H\"older-exponent $\beta_i$ and H\"older norm bounded by $K$, and it is implicitly assumed that $g_i$ depends only on $t_i$ of its $d_i$ possible input variables. Consider a predictor $\hat{\mu}_n$ that is obtained by empirical risk minimization over a certain class $\mathcal F(L, (p_i)_{i=0}^{L+1}, s, F)$ of $F$-bounded (i.e., $\|\hat{\mu}_n\|_\infty\le F$) ReLU networks with $L+1$ layers where the $i$-th layer has output dimension $p_i$, $p_0=p$ and $p_{L+1}=1$, and which have bounded and $s$-sparse network parameters. If the tuning parameters $L, (p_i)_{i=0}^{L+1}, s$ and $F$ are chosen appropriately, and if, in particular, the network depth $L = L_n$ is logarithmically growing with $n$, \citet[][Corollary~1]{SchmidtH20} shows that 
$$
\E\left[\left(\hat{\mu}_n(x_0) - \mu_P(x_0)\right)^2\right] \le C'L_n \phi_n \log^2n,
$$
where $C' = C'_p$ does not depend on $n$ (but on $p$), 
$$
\phi_n = \max_{i=0,\dots, q} n^{-\frac{2\beta_i^*}{2\beta_i^*+t_i}}
$$
and $\beta_i^*= \beta_i\prod_{\ell=i+1}^q (\beta_\ell\land1)$. Thus, \eqref{eq:defCons:a} and \eqref{eq:defCons:b} are satisfied with $g_n = \mu_{P_n}$ and $\sigma_n=1$, provided that $C_{p_n}'L_n \phi_{n^*} \log^2n^* = o(1)$, where $n^* = n - \max_l m_l$ is the number of observed feature-response pairs excluding the ones in the largest fold. 
\end{example}

\begin{example}[Random forests]\label{ex:randForest}\normalfont
\citet{Scornet15} consider an additive regression model similar to \ref{c.non-linmod}, but with $v_P\thicksim\mathcal N(0,1)$, $P_{x_0}$ the uniform distribution on $[0,1]^p$ and additive regression function
$$
\mu(x) = \mu_P(x) =  \sum_{j=1}^p \mu_j(x^{(j)}),\quad x=(x^{(1)},\dots, x^{(p)})'\in\R^p,
$$
where the $\mu_j:[0,1]\to\R$ are continuous. A random forest predictor $\hat{\mu}_n$ is trained in the following way: Each tree is grown by first randomly subsampling a number of $a_n\in[n]$ feature-response pairs from the training set and then growing a tree according to the CART-split criterion \citep[see][Section~2, for details]{Scornet15}. The tree growing process is terminated when a number of $t_n\in[n]$ leaves is reached. In this way, $M$ random regression trees are grown, each of which predicts $y_0$ by averaging all the $y_i$ whose features $x_i$ belong to the leaf containing $x_0$. \citet{Scornet15} study the idealized predictor $\hat{\mu}_n(x_0)$ obtained from averaging these $M$ predictions and letting $M\to\infty$. They show that if $a_n,t_n\to \infty$ in such a way that $t_n(\log a_n)^9/a_n\to 0$, we have for every fixed data generating distribution $P$ as above,
$$
\E[(\hat{\mu}_n(x_0)-\mu(x_0))^2] \to 0,\quad\text{as } n\to\infty.
$$
Thus \eqref{eq:defCons:a} holds, and under the (approximate) equal fold size construction mentioned above, \eqref{eq:defCons:c} also holds provided $n-m_1\to\infty$, because here $p_n=p$ is constant.
\end{example}

\begin{example}[High-dimensional linear regression with the LASSO]\label{ex:LASSO}\normalfont

Consider a non-decreasing sequence $(K_n)_{n\in\N}$ of positive real numbers and a sequence of dimension parameters $(p_n)_{n\in\N}$ such that $K_n^4\log(p_n)/n\to 0$ as $n\to\infty$. For a positive finite constant $M$, let $\mathcal P_n(M)$ denote the class of probability distributions on $\R^{p_n+1}$, such that  under $P\in\mathcal P_n(M)$, the pair $(x_0,y_0)$ has the following properties: 
\begin{itemize}
\item $\|x_0\|_\infty \le M$, almost surely.
\item Conditional on $x_0$, $y_0$ is distributed as $\mathcal N(x_0'\beta_P, \sigma_P^2)$, for some $\beta_P\in\R^{p_n}$ and $\sigma_P^2\in(0,\infty)$.
\item The parameters $\beta_P$ and $\sigma_P^2$ satisfy $\max(\|\beta_P\|_1, \sigma_P)\le K_n$.
\end{itemize}
In particular, we have $\mu_P(x_0) = x_0'\beta_P$.
\citet[][Theorem~1]{Chatt13} shows that any estimate $\hat{\beta}_{n,K_n}$ which minimizes
$$
\beta \quad\mapsto \quad \sum_{i=1}^n (y_i - \beta'x_i)^2 \quad \text{subject to }\quad \|\beta\|_1\le K_n
$$
satisfies
\begin{align*}
\P_n\left(\left| x_{0,n}'\hat{\beta}_{n,K_n} - \mu_{P_n}(x_{0,n})\right|>\eps\right) \xrightarrow[n\to\infty]{} 0
\end{align*}
for every $\eps>0$ and if $P_n\in\mathcal P_n(M)$. Clearly, here the estimate $\hat{\beta}^{(1)}_{n,K_n}$ excluding the data from the first fold has the same asymptotic property, provided that $K_{n-m_1}^4\log p_n / (n-m_1) \to 0$, with $m_1=\lfloor\frac{n}{k_n}\rfloor$. Thus, under the (approximate) equal fold size construction, UAD holds. Note that in this example, consistent estimation of the parameters $\beta_P$ and $\sigma_P^2$ would require additional assumptions on the distribution of the feature vector $x_0$ \citep[so called `compatibility conditions', see][]{Buehlmann11}, and therefore, it is not immediately clear whether the usual Gaussian prediction interval, based on estimates $\hat{\beta}_n$ and $\hat{\sigma}_n^2$ and a Gaussian quantile, is asymptotically valid in the present setting. Furthermore, the result of \citet{Chatt13} can be extended also to the non-Gaussian case, where the usual Gaussian prediction interval certainly fails.
\end{example}

\begin{example}[Ridge regression with many variables]\label{ex:RIDGE}\normalfont
 A qualitatively different parameter space is considered in \citet{Lopes15}, who shows uniform consistency of ridge regularized estimators in a linear model under a boundedness assumption on the regression parameter $\beta_P$ and a specific decay rate of eigenvalues of $\E_{x_0}[x_0x_0']$. 
\end{example}

\begin{example}[Misspecified regression estimation]\normalfont
A classical strand of literature on the asymptotics of Maximum-Likelihood under misspecification has established various conditions under which the MLE is not consistent for the true unknown parameter, but for a pseudo parameter that corresponds to the projection of the true data generating distribution onto the maintained working model. See, for example, \citet{Huber67}, \citet{White80a,White80b} or \citet{Fahrmeir90}. A common pseudo target in random design regression is the minimizer of $\beta\mapsto\E_{z_0}[(y_0-\beta'x_0)^2]$.
\end{example}


\section{Numerical results}
\label{sec:mc}

We conduct an extensive simulation study to assess the quality of our theoretical approximations in a small sample situation. For our numerical experiments we closely follow the setup of \citet{Barber19b}. We compute prediction intervals with nominal coverage probability of $1-\alpha = 0.9$ using i.i.d. training samples of different sizes $n$ and $p=n$ explanatory variables. Each data point $(x_i,y_i)$ is generated as
$$
x_i \thicksim \mathcal N(0,I_p)\quad\text{and}\quad y_i\|x_i \thicksim \mathcal N(\beta'x_i,1),\quad i=1,\dots, n.
$$
We also generate $100$ test data points from the same distribution.
The true coefficient vector $\beta$ is drawn once for each parameter configuration as $\beta=\sqrt{10}\cdot u$ for a uniform random unit vector $u$ in $\R^p$ and is then fixed throughout the Monte Carlo iterations. Instead of the least squares estimator (or the minimum norm interpolator in case $p>n$) studied in the simulation section~7.1 of \citet{Barber19b}, which was exemplified to be unstable in the present regime of $n\approx p$ and results in anti-conservative performance, we here investigate the LASSO as our prediction algorithm. For simplicity and reproducibility we take the default implementation of the LASSO in the R-package \emph{glmnet} \citep{Friedman10} with data driven choice of tuning parameter determined by $5$-fold cross validation\footnote{\emph{cv.glmnet(x=X, y=Y, family="gaussian", intercept=FALSE, nfolds=5, alpha=1)} in \emph{glmnet} version~4.1-2.}. 

First, we try to experimentally confirm our theoretical prediction that the cross-validation prediction intervals are asymptotically training conditionally valid even in scenarios where no consistent parameter estimation is possible. For simplicity, we consider only leave-one-out cross validation ($k=n$). Recall that in the challenging scenario where $n=p$, stability of the OLS predictor fails. Notice also that the condition $\|x_0\|_\infty\le M$ for consistency of the LASSO predictor of Example~\ref{ex:LASSO} is not satisfied in the current setup. For each $n$, using the $100$ test data points, we compute the empirical probability of coverage. For one fixed training sample $T_n$, this gives us a Monte Carlo estimate of the conditional coverage probability 
$\P(y_0\in PI_{\alpha}^{(L1O)}(T_n,x_0) \| T_n)$. In Figure~\ref{fig:Asymp} we show histograms of these $100$ Monte Carlo estimates, their sample means (black solid lines) and the nominal coverage level $1-\alpha = 0.9$ (red dashed lines). For small samples we observe an anti-conservative bias as well as a considerable variability in the true conditional coverage probabilities, both of which are reduced as sample size $n$ increases, even though $p=n$ increases as well.

In Figure~\ref{fig:Asymp-Cov-Len-Stab} we show average coverage (i.e., Monte Carlo estimates of the unconditional coverage probabilities $\P(y_0\in PI_{\alpha}(T_n,x_0))$), average length (i.e., Monte Carlo estimates of the expected length) and a Monte Carlo estimate $\hat{\eta}_{n,k}$ of the stability coefficient $\eta_{n,k}$ (cf. Definition~\ref{DEF:STABLE}) as functions of $n$ (here, still, $n=k=p$). We estimate the stability of the predictor $\hat{\mu}_ n$ on a given training sample $T_n$, by
$$
\hat{\eta}_{n,k}(T_n) = \frac1k\sum_{l=1}^k \frac{1}{100}\sum_{j=1}^{100}\left[\left(\frac{1}{\sqrt{2\pi}}\left|\hat{\mu}_n(x_{0,j}) - \hat{\mu}_n^{(l)}(x_{0,j})\right|\right) \land 1 \right], 
$$
where $x_{0,j}$, $j=1,\dots, 100$ are the regressors in the test set, and then average over the $100$ Monte Carlo samples $T_n$ to approximate $\eta_{n,k}$. As we already saw in Figure~\ref{fig:Asymp} (black bold vertical lines), Figure~\ref{fig:Asymp-Cov-Len-Stab} also shows that the unconditional coverage probabilities increase towards the nominal level with increasing $n$ and $p$. Simultaneously, the expected interval lengths and the stability coefficients decrease with increasing $n$. Notice that a feasible but trivial $1-\alpha$ prediction (or tolerance) interval for $y_0$ in the present scenario is $TI_\alpha = \pm q_{1-\alpha/2}^N\cdot \sigma$, with $\sigma^2 = \Var[y_0] = \|\beta\|_2^2 + 1 = 11$, which has length $2 \cdot q_{0.95}^N\cdot \sqrt{11} = 10.91$. This is well outperformed by the leave-one-out prediction interval even in small samples. For reference, the oracle prediction interval for $y_0$ which uses knowledge of the true coefficient vector $\beta$, has length $2 \cdot q_{0.95}^N = 3.29$.

\begin{figure}[tbp]
\includegraphics[width=\textwidth]{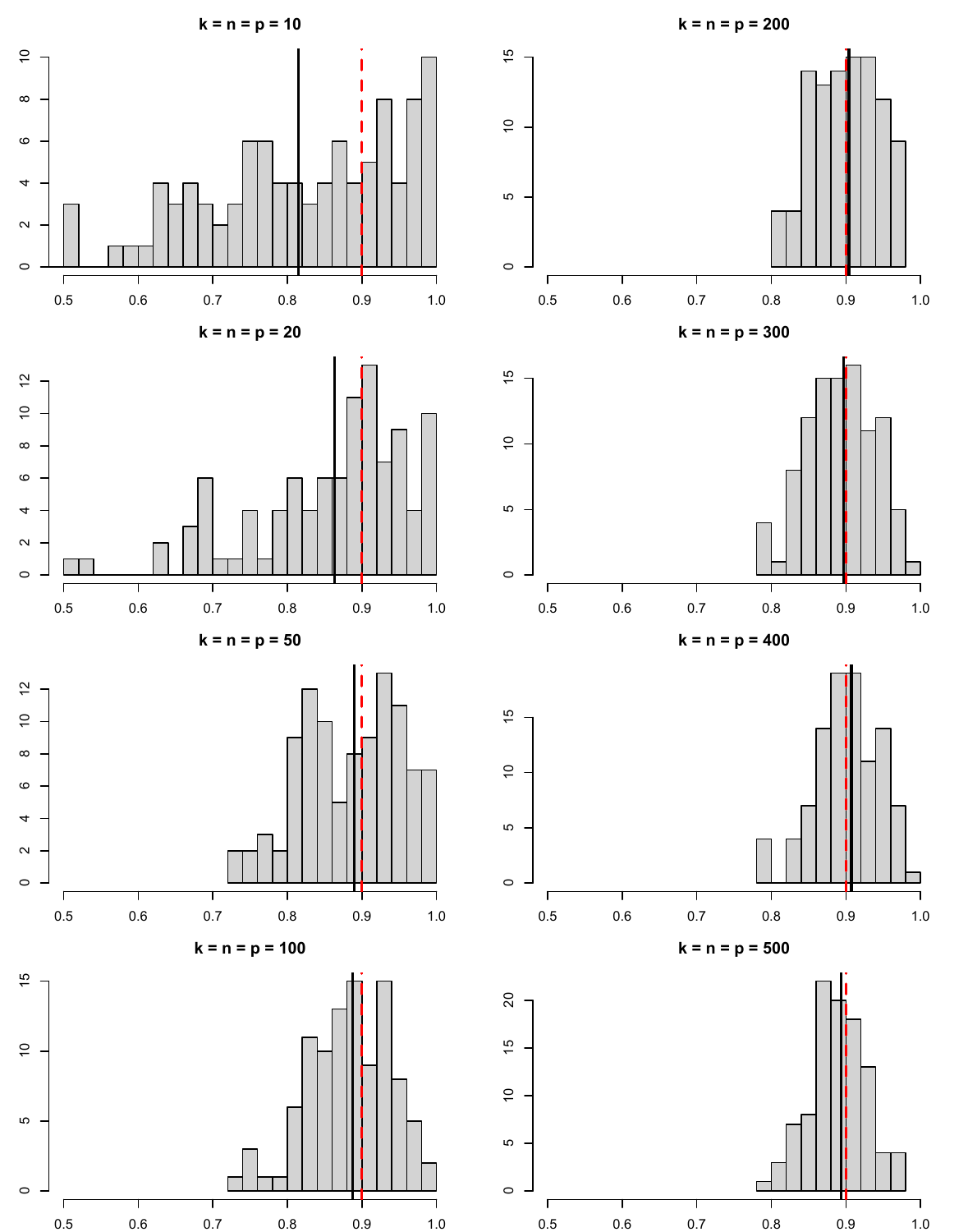} 
\caption{Histograms of conditional coverage probabilities $\P(y_0\in PI^{(L1O)}_{\alpha}(T_n,x_0) \| T_n)$ from $100$ Monte Carlo training samples $T_n$ for different values of $n$ at confidence level of $1-\alpha = 0.9$ (red dashed line). Here $k=p=n$.}
\label{fig:Asymp}
\end{figure}

\begin{figure}[tbp]
\includegraphics[width=\textwidth]{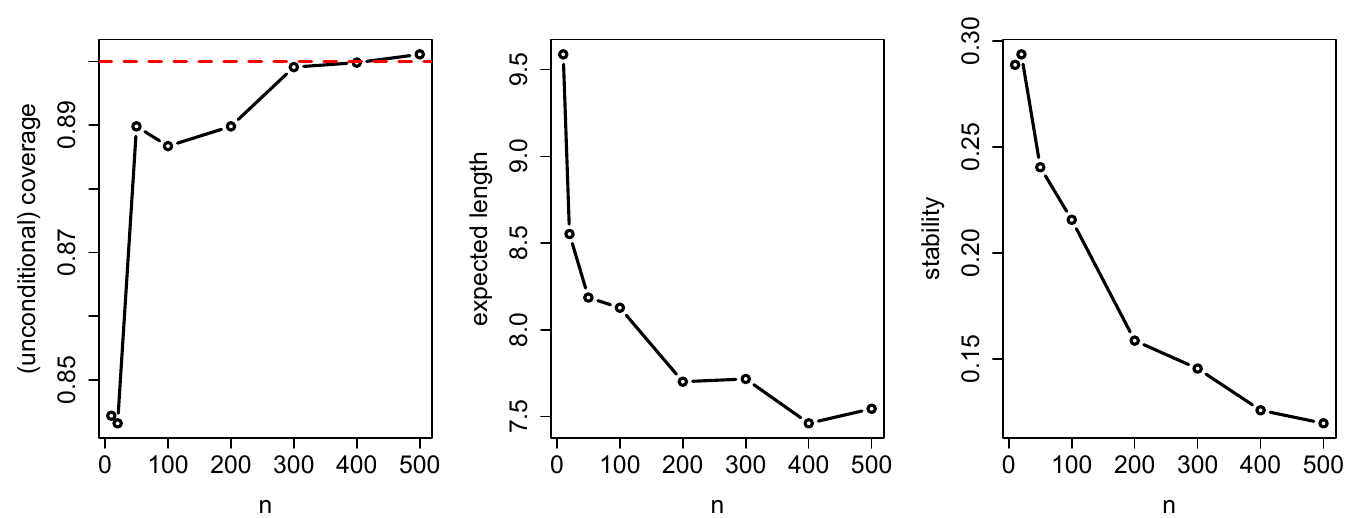} 
\caption{Unconditional coverage probability, expected interval length and estimated stability for different sample sizes $n$ from $100$ Monte Carlo simulations with $k=n=p$ (i.e., leave-one-out).}
\label{fig:Asymp-Cov-Len-Stab}
\end{figure}

Next, we investigate the impact of the number of folds $k$ on the $k$-CV prediction interval. We compute \eqref{eq:kCVPI} for a range of $k$ from $2$ up to sample size $n$. In Figure~\ref{fig:Hists-np100-HOM} we plot histograms of conditional coverage probabilities for different values of $k$, in the case $n=p=100$. We find that for small values of $k$, most training samples produce conservative prediction intervals, whereas for larger $k$ the conditional coverages are more symmetrically spread around the nominal level with a slight anti-conservative bias (as was already observed in Figure~\ref{fig:Asymp}). The number of folds $k$ does not seem to have a strong impact on the spread of theses conditional probabilities. The conservative behavior for small $k$ can be explained in a similar way as for the case of sample splitting (cf. Subsection~\ref{sec:SS} and Subsections~\ref{sec:efficiency} and \ref{sec:naive} in the supplement). In a nutshell, if $k$ is small, then $\hat{\mu}_n^{(l)}$, $l=1,\dots, k$, will be based on much fewer observations than $\hat{\mu}_n$ and will thus be less accurate for prediction, leading to residuals that will have much larger variance than the true prediction error $y_0-\hat{\mu}_n(x_0)$. Hence, the resulting prediction interval will be too wide. This reasoning is also supported by Figure~\ref{fig:Cov-Len-Stab}.

\begin{figure}[tbp]
\includegraphics[width=\textwidth]{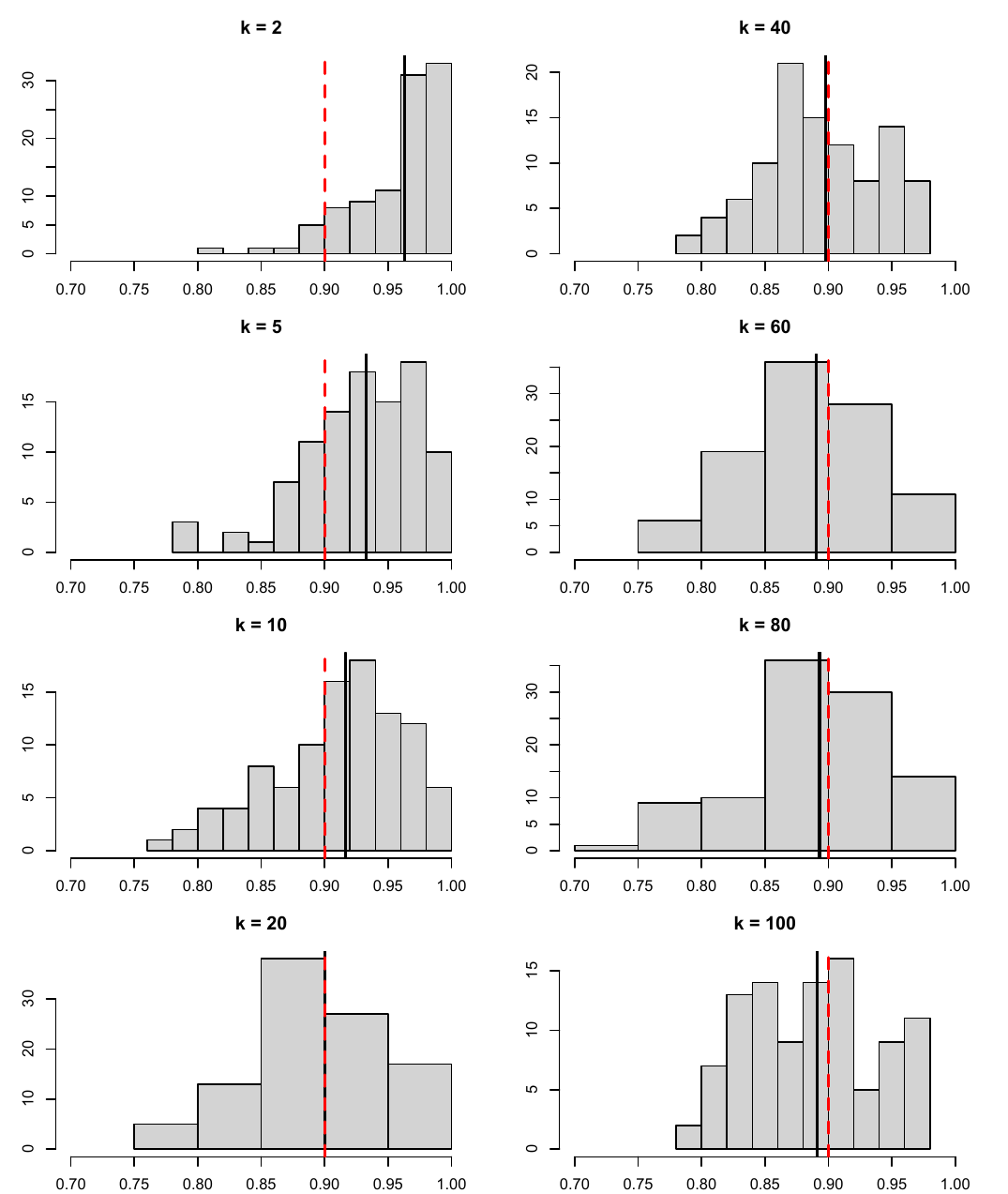} 
\caption{Histograms of conditional coverage probabilities $\P(y_0\in PI^{(kCV)}_{\alpha}(T_n,x_0) \| T_n)$ from $100$ Monte Carlo training samples $T_n$ for different values of $k$  at confidence level of $1-\alpha = 0.9$ (red dashed line). Here $n=p=100$.}
\label{fig:Hists-np100-HOM}
\end{figure}

%
%

From Figure~\ref{fig:Cov-Len-Stab} we learn that coverage probabilities, expected interval lengths and stability all decrease with increasing $k$ in a quite similar fashion. This was to be expected for the stability, because for nearly equal fold sizes, the number of observations in each fold decreases as $k$ increases, which means that, on average, $|\hat{\mu}_n(x_0) - \hat{\mu}_n^{(l)}(x_0)|$ will decrease for each $l$, as the two predictors are computed on nearly the same observations. We also see in Figure~\ref{fig:Cov-Len-Stab} that apparently for smaller values of $k$ below $n/2$, the $k$-CV prediction interval is too wide, as we observe both, large interval lengths as well as over-coverage (cf. the discussion in the previous paragraph). On the other hand, there does not seem to be any practical reason to choose $k$ substantially larger than $n/2$, as there is no benefit in terms of validity or interval length for increasing the number of folds above that threshold. However, keep in mind that the runtime of the $k$-CV prediction interval increases (roughly) linearly in $k$ (cf. Remark~\ref{rem:computation} in the supplement). 

\begin{figure}[tbp]
\includegraphics[width=\textwidth]{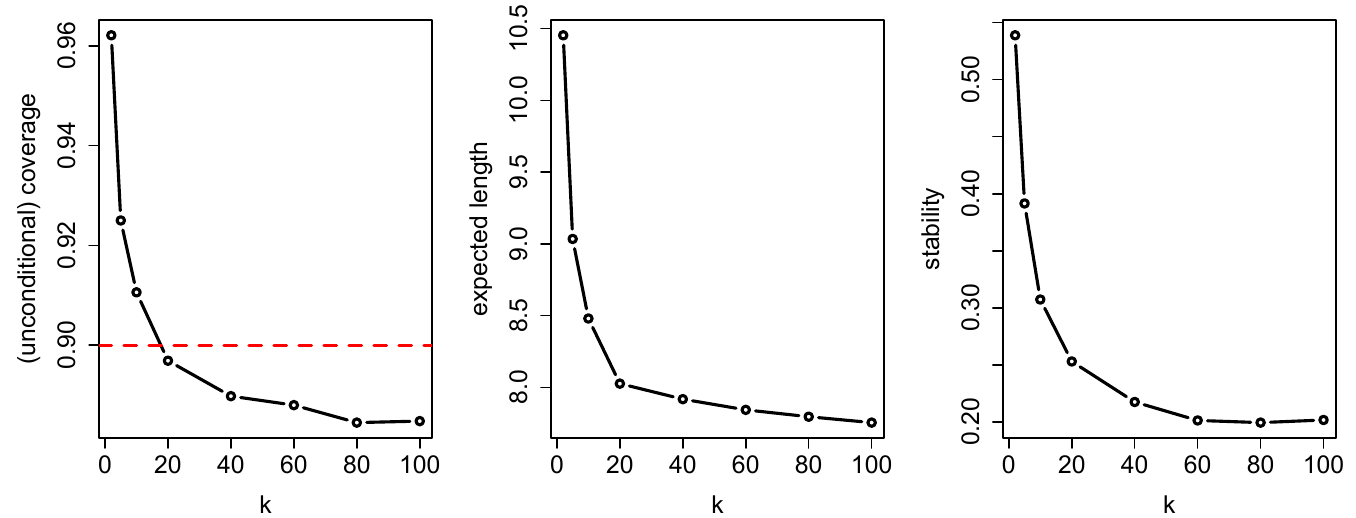} 
\caption{Unconditional coverage probability, expected interval length and estimated stability for different numbers of folds $k$ from $100$ Monte Carlo simulations with $n=p=100$.}
\label{fig:Cov-Len-Stab}
\end{figure}

Finally, we repeated the simulations for heteroskedastic $y_i$ generated as
$$
y_i\|x_i \thicksim \mathcal N\left(\beta'x_i,\sigma^2(x_i)\right), \quad \sigma^2(x_i) = 1+\frac{|\beta'x_i|}{\sqrt{10}} - \frac{\sqrt{2}}{\Gamma(\frac12)}.
$$
For comparability we have chosen the conditional variance of $y_i$ such that $\E[\sigma^2(x_i)] = 1$ in order to achieve $\Var(y_i) = 11$, as in the homoskedastic case. The resulting plots are very similar to those of the homoskedastic case and are deferred to Section~\ref{sec:mcApp} in the supplement. For heteroskedastic data we observe slightly larger spread of conditional probabilities around the nominal level, slightly wider intervals and slightly slower asymptotic convergence of average (unconditional) coverage to the nominal level.

\section*{Acknowledgements}

The authors thank the participants of the ``ISOR Research Seminar in Statistics and Econometrics" at the University of Vienna for discussion of an early version of the paper. In particular, we want to thank Benedikt M. P\"otscher and David Preinerstorfer for valuable comments. We are also grateful to three anonymous referees and an associate editor for their constructive feedback to improve the paper.

{\small
\bibliographystyle{chicago}
\bibliography{PredInter.bbl}

\begin{thebibliography}{}

\bibitem[\protect\citeauthoryear{Aven}{Aven}{1985}]{Aven85}
Aven, T. (1985).
\newblock Upper (lower) bounds on the mean of the maximum (minimum) of a number
  of random variables.
\newblock {\em J. Appl. Probab.\/}~{\em 22\/}(3), 723--728.

\bibitem[\protect\citeauthoryear{Bai and Silverstein}{Bai and
  Silverstein}{2010}]{BaiSilv10}
Bai, Z. and J.~W. Silverstein (2010).
\newblock {\em Spectral Analysis of Large Dimensional Random Matrices\/} (2nd
  ed.).
\newblock Springer Series in Statistics. New York: Springer.

\bibitem[\protect\citeauthoryear{Bai and Yin}{Bai and Yin}{1993}]{Bai93}
Bai, Z.~D. and Y.~Q. Yin (1993).
\newblock Limit of the smallest eigenvalue of a large dimensional sample
  covariance matrix.
\newblock {\em Ann. Probab.\/}~{\em 21\/}(3), 1275--1294.

\bibitem[\protect\citeauthoryear{Barber, Candes, Ramdas, and Tibshirani}{Barber
  et~al.}{2021}]{Barber19b}
Barber, R.~F., E.~J. Candes, A.~Ramdas, and R.~J. Tibshirani (2021).
\newblock Predictive inference with the jackknife+.
\newblock {\em Ann. Statist.\/}~{\em 49\/}(1), 486--507.

\bibitem[\protect\citeauthoryear{Bean, Bickel, El~Karoui, and Yu}{Bean
  et~al.}{2013}]{Bean13}
Bean, D., P.~J. Bickel, N.~El~Karoui, and B.~Yu (2013).
\newblock Optimal m-estimation in high-dimensional regression.
\newblock {\em Proc. Natl. Acad. Sci. USA\/}~{\em 110\/}(36), 14563--14568.

\bibitem[\protect\citeauthoryear{Bickel and Freedman}{Bickel and
  Freedman}{1983}]{Bickel83}
Bickel, P.~J. and D.~A. Freedman (1983).
\newblock Bootstrapping regression models with many parameters.
\newblock In P.~Bickel, K.~Doksum, and J.~Hodges (Eds.), {\em A Festschrift for
  Erich L. Lehmann}, pp.\  28--48. Wadsworth Inc.

\bibitem[\protect\citeauthoryear{Billingsley}{Billingsley}{1995}]{Billingsley95}
Billingsley, P. (1995).
\newblock {\em Probability and Measure\/} (3rd ed.).
\newblock New York: Wiley.

\bibitem[\protect\citeauthoryear{Bousquet and Elisseeff}{Bousquet and
  Elisseeff}{2002}]{Bousquet02}
Bousquet, O. and A.~Elisseeff (2002).
\newblock Stability and generalization.
\newblock {\em J. Mach. Learn. Res.\/}~{\em 2}, 499--526.

\bibitem[\protect\citeauthoryear{Bucchianico, Einmahl, and
  Mushkudiani}{Bucchianico et~al.}{2001}]{Bucc01}
Bucchianico, A.~D., J.~H.~J. Einmahl, and N.~A. Mushkudiani (2001).
\newblock Smallest nonparametric tolerance regions.
\newblock {\em Ann. Statist.\/}~{\em 29\/}(5), 1320--1343.

\bibitem[\protect\citeauthoryear{B\"{u}hlmann and van~de Geer}{B\"{u}hlmann and
  van~de Geer}{2011}]{Buehlmann11}
B\"{u}hlmann, P. and S.~van~de Geer (2011).
\newblock {\em Statistics for High-dimensional Data}.
\newblock Berlin: Springer.

\bibitem[\protect\citeauthoryear{Butler and Rothman}{Butler and
  Rothman}{1980}]{ButlerRoth80}
Butler, R. and E.~D. Rothman (1980).
\newblock Predictive intervals based on reuse of the sample.
\newblock {\em J. Amer. Statist. Assoc.\/}~{\em 75\/}(372), 881--889.

\bibitem[\protect\citeauthoryear{Chatterjee}{Chatterjee}{2013}]{Chatt13}
Chatterjee, S. (2013).
\newblock Assumptionless consistency of the lasso.
\newblock {\em arXiv preprint arXiv:1303.5817\/}.

\bibitem[\protect\citeauthoryear{Chatterjee and Patra}{Chatterjee and
  Patra}{1980}]{Chatterjee80}
Chatterjee, S.~K. and N.~K. Patra (1980).
\newblock Asymptotically minimal multivariate tolerance sets.
\newblock {\em Calcutta Statist. Assoc. Bull.\/}~{\em 29\/}(1-2), 73--94.

\bibitem[\protect\citeauthoryear{Chen, Chun, and Barber}{Chen
  et~al.}{2018}]{Chen17}
Chen, W., K.-J. Chun, and R.~F. Barber (2018).
\newblock Discretized conformal prediction for efficient distribution-free
  inference.
\newblock {\em Stat\/}~{\em 7\/}(1), e173.

\bibitem[\protect\citeauthoryear{Cox and Hinkley}{Cox and
  Hinkley}{1974}]{Cox74}
Cox, D.~R. and D.~V. Hinkley (1974).
\newblock {\em Theoretical Statistics}.
\newblock London: Chapman and Hall.

\bibitem[\protect\citeauthoryear{Devroye and Wagner}{Devroye and
  Wagner}{1979}]{Devroye79}
Devroye, L. and T.~J. Wagner (1979).
\newblock Distribution-free inequalities for the deleted and holdout error
  estimates.
\newblock {\em IEEE Trans. Inform. Theory\/}~{\em 25\/}(2), 202--207.

\bibitem[\protect\citeauthoryear{Dicker}{Dicker}{2012}]{Dicker12}
Dicker, L. (2012).
\newblock Optimal estimation and prediction for dense signals in
  high-dimensional linear models.
\newblock {\em arXiv:1203.4572\/}.

\bibitem[\protect\citeauthoryear{Dicker}{Dicker}{2013}]{Dicker13}
Dicker, L.~H. (2013).
\newblock Optimal equivariant prediction for high-dimensional linear models
  with arbitrary predictor covariance.
\newblock {\em Electron. J. Stat.\/}~{\em 7}, 1806--1834.

\bibitem[\protect\citeauthoryear{El~Karoui}{El~Karoui}{2010}]{ElKaroui10}
El~Karoui, N. (2010).
\newblock The spectrum of kernel random matrices.
\newblock {\em Ann. Statist.\/}~{\em 38\/}(1), 1--50.

\bibitem[\protect\citeauthoryear{El~Karoui}{El~Karoui}{2013}]{ElKaroui13}
El~Karoui, N. (2013).
\newblock Asymptotic behavior of unregularized and ridge-regularized
  high-dimensional robust regression estimators: rigorous results.
\newblock {\em arXiv preprint arXiv:1311.2445\/}.

\bibitem[\protect\citeauthoryear{El~Karoui}{El~Karoui}{2018}]{ElKaroui18}
El~Karoui, N. (2018).
\newblock On the impact of predictor geometry on the performance on
  high-dimensional ridge-regularized generalized robust regression estimators.
\newblock {\em Probab. Theory Relat. Fields\/}~{\em 170\/}(1-2), 95--175.

\bibitem[\protect\citeauthoryear{El~Karoui, Bean, Bickel, Lim, and
  Yu}{El~Karoui et~al.}{2013}]{ElKaroui13b}
El~Karoui, N., D.~Bean, P.~J. Bickel, C.~Lim, and B.~Yu (2013).
\newblock On robust regression with high-dimensional predictors.
\newblock {\em Proc. Natl. Acad. Sci. USA\/}~{\em 110\/}(36), 14557--14562.

\bibitem[\protect\citeauthoryear{El~Karoui and Purdom}{El~Karoui and
  Purdom}{2018}]{ElKaroui15}
El~Karoui, N. and E.~Purdom (2018).
\newblock Can we trust the bootstrap in high-dimensions? the case of linear
  models.
\newblock {\em J. Mach. Learn. Res.\/}~{\em 19\/}(1), 170--235.

\bibitem[\protect\citeauthoryear{Fahrmeir}{Fahrmeir}{1990}]{Fahrmeir90}
Fahrmeir, L. (1990).
\newblock Maximum likelihood estimation in misspecified generalized linear
  models.
\newblock {\em Statistics\/}~{\em 21\/}(4), 487--502.

\bibitem[\protect\citeauthoryear{Foygel~Barber, Candes, Ramdas, and
  Tibshirani}{Foygel~Barber et~al.}{2021}]{Barber19}
Foygel~Barber, R., E.~J. Candes, A.~Ramdas, and R.~J. Tibshirani (2021).
\newblock The limits of distribution-free conditional predictive inference.
\newblock {\em Information and Inference: A Journal of the IMA\/}~{\em
  10\/}(2), 455--482.

\bibitem[\protect\citeauthoryear{Friedman, Hastie, and Tibshirani}{Friedman
  et~al.}{2010}]{Friedman10}
Friedman, J., T.~Hastie, and R.~Tibshirani (2010).
\newblock Regularization paths for generalized linear models via coordinate
  descent.
\newblock {\em Journal of Statistical Software\/}~{\em 33\/}(1), 1--22.

\bibitem[\protect\citeauthoryear{Gy\"{o}rfi, Kohler, Krzy\.{z}ak, and
  Walk}{Gy\"{o}rfi et~al.}{2002}]{Gyorfi02}
Gy\"{o}rfi, L., M.~Kohler, A.~Krzy\.{z}ak, and H.~Walk (2002).
\newblock {\em {A} {D}istribution-{F}ree {T}heory of {N}onparametric
  {R}egression}.
\newblock Springer Series in Statistics. New York: Springer.

\bibitem[\protect\citeauthoryear{Huber and Leeb}{Huber and
  Leeb}{2013}]{Huber13}
Huber, N. and H.~Leeb (2013).
\newblock Shrinkage estimators for prediction out-of-sample: Conditional
  performance.
\newblock {\em Commun. Statist. - Theory Methods\/}~{\em 42\/}(7), 1246--1264.

\bibitem[\protect\citeauthoryear{Huber}{Huber}{1967}]{Huber67}
Huber, P.~J. (1967).
\newblock The behavior of maximum likelihood estimates under nonstandard
  conditions.
\newblock {\em Proceedings of the Fifth Berkeley Symposium on Mathematical
  Statistics and Probability\/}~{\em 1}, 221--233.

\bibitem[\protect\citeauthoryear{Krishnamoorthy and Mathew}{Krishnamoorthy and
  Mathew}{2009}]{Krish09}
Krishnamoorthy, K. and T.~Mathew (2009).
\newblock {\em Statistical Tolerance Regions: Theory, Applications, and
  Computation}.
\newblock Wiley Series in Probability and Statistics. Hoboken, New Jersey:
  Wiley.

\bibitem[\protect\citeauthoryear{Lei}{Lei}{2019}]{Lei17}
Lei, J. (2019).
\newblock Fast exact conformalization of the lasso using piecewise linear
  homotopy.
\newblock {\em Biometrika\/}~{\em 106\/}(4), 749--764.

\bibitem[\protect\citeauthoryear{Lei, G'Sell, Rinaldo, Tibshirani, and
  Wasserman}{Lei et~al.}{2018}]{Lei16}
Lei, J., M.~G'Sell, A.~Rinaldo, R.~J. Tibshirani, and L.~Wasserman (2018).
\newblock Distribution-free predictive inference for regression.
\newblock {\em J. Amer. Statist. Assoc.\/}~{\em 113\/}(523), 1094--1111.

\bibitem[\protect\citeauthoryear{Lei, Robins, and Wasserman}{Lei
  et~al.}{2013}]{Lei13}
Lei, J., J.~Robins, and L.~Wasserman (2013).
\newblock Distribution-free prediction sets.
\newblock {\em J. Amer. Statist. Assoc.\/}~{\em 108\/}(501), 278--287.

\bibitem[\protect\citeauthoryear{Lei and Wasserman}{Lei and
  Wasserman}{2014}]{Lei14}
Lei, J. and L.~Wasserman (2014).
\newblock Distribution-free prediction bands for non-parametric regression.
\newblock {\em J. Roy. Statist. Soc. Ser. B\/}~{\em 76}, 71--96.

\bibitem[\protect\citeauthoryear{Lewis and Thompson}{Lewis and
  Thompson}{1981}]{Lewis81}
Lewis, T. and J.~W. Thompson (1981).
\newblock Dispersive distributions, and the connection between dispersivity and
  strong unimodality.
\newblock {\em J. Appl. Probab.\/}~{\em 18\/}(1), 76--90.

\bibitem[\protect\citeauthoryear{Li and Liu}{Li and Liu}{2008}]{Li08}
Li, J. and R.~Y. Liu (2008).
\newblock Multivariate spacings based on data depth: I. construction of
  nonparametric multivariate tolerance regions.
\newblock {\em Ann. Statist.\/}~{\em 36\/}(3), 1299--1323.

\bibitem[\protect\citeauthoryear{Lopes}{Lopes}{2015}]{Lopes15}
Lopes, M.~E. (2015).
\newblock {\em Some Inference Problems in High-Dimensional Linear Models}.
\newblock Ph.\ D. thesis, UC Berkeley.

\bibitem[\protect\citeauthoryear{Mammen}{Mammen}{1996}]{Mammen96}
Mammen, E. (1996).
\newblock Empirical process of residuals for high-dimensional linear models.
\newblock {\em Ann. Statist.\/}~{\em 24\/}(1), 307--335.

\bibitem[\protect\citeauthoryear{Massart}{Massart}{1990}]{Massart90}
Massart, P. (1990).
\newblock The tight constant in the {D}voretzky-{K}iefer-{W}olfowitz
  inequality.
\newblock {\em Ann. Probab.\/}~{\em 18\/}(3), 1269--1283.

\bibitem[\protect\citeauthoryear{Olive}{Olive}{2007}]{Olive07}
Olive, D.~J. (2007).
\newblock Prediction intervals for regression models.
\newblock {\em Comput. Statist. Data Anal.\/}~{\em 51\/}(6), 3115--3122.

\bibitem[\protect\citeauthoryear{Politis}{Politis}{2013}]{Politis13}
Politis, D.~N. (2013).
\newblock Model-free model-fitting and predictive distributions.
\newblock {\em Test\/}~{\em 22\/}(2), 183--221.

\bibitem[\protect\citeauthoryear{Romano, Patterson, and Candes}{Romano
  et~al.}{2019}]{Romano19}
Romano, Y., E.~Patterson, and E.~Candes (2019).
\newblock Conformalized quantile regression.
\newblock {\em Adv. Neural Inf. Process. Syst.\/}~{\em 32}.

\bibitem[\protect\citeauthoryear{Schmidt-Hieber}{Schmidt-Hieber}{2020}]{SchmidtH20}
Schmidt-Hieber, J. (2020).
\newblock Nonparametric regression using deep neural networks with relu
  activation function.
\newblock {\em Ann. Statist.\/}~{\em 48\/}(4), 1875--1897.

\bibitem[\protect\citeauthoryear{Schmoyer}{Schmoyer}{1992}]{Schmoyer92}
Schmoyer, R.~L. (1992).
\newblock Asymptotically valid prediction intervals for linear models.
\newblock {\em Technometrics\/}~{\em 34\/}(4), 399--408.

\bibitem[\protect\citeauthoryear{Scornet, Biau, and Vert}{Scornet
  et~al.}{2015}]{Scornet15}
Scornet, E., G.~Biau, and J.-P. Vert (2015).
\newblock Consistency of random forests.
\newblock {\em Ann. Statist.\/}~{\em 43\/}(4), 1716--1741.

\bibitem[\protect\citeauthoryear{Steinberger and Leeb}{Steinberger and
  Leeb}{2016}]{Steinb16d}
Steinberger, L. and H.~Leeb (2016).
\newblock Leave-one-out prediction intervals in linear regression models with
  many variables.
\newblock {\em arXiv preprint arXiv:1602.05801\/}.

\bibitem[\protect\citeauthoryear{Steinberger and Leeb}{Steinberger and
  Leeb}{2021}]{Steinb21}
Steinberger, L. and H.~Leeb (2021).
\newblock Conditional predictive inference for high-dimensional stable
  algorithms.
\newblock {\em arXiv preprint arXiv:1809.01412\/}.

\bibitem[\protect\citeauthoryear{Steinberger and Leeb}{Steinberger and
  Leeb}{2022}]{Steinb22supp}
Steinberger, L. and H.~Leeb (2022).
\newblock Supplement to ``{C}onditional predictive inference for stable
  algorithms''.

\bibitem[\protect\citeauthoryear{Stine}{Stine}{1985}]{Stine85}
Stine, R.~A. (1985).
\newblock Bootstrap prediction intervals for regression.
\newblock {\em J. Amer. Statist. Assoc.\/}~{\em 80\/}(392), 1026--1031.

\bibitem[\protect\citeauthoryear{Tukey}{Tukey}{1947}]{Tukey47}
Tukey, J.~W. (1947).
\newblock Non-parametric estimation ii. statistically equivalent blocks and
  tolerance regions -- the continuous case.
\newblock {\em Ann. Math. Statist.\/}~{\em 18\/}(4), 529--539.

\bibitem[\protect\citeauthoryear{van~der Vaart}{van~der
  Vaart}{2007}]{vanderVaart07}
van~der Vaart, A.~W. (2007).
\newblock {\em Asymptotic Statistics\/} (8th ed.).
\newblock Cambridge Series in Statistical and Probabilistic Mathematics. New
  York: Cambridge University Press.

\bibitem[\protect\citeauthoryear{Vovk}{Vovk}{2012}]{Vovk12}
Vovk, V. (2012).
\newblock Conditional validity of inductive conformal predictors.
\newblock In {\em Asian conference on machine learning}, pp.\  475--490.

\bibitem[\protect\citeauthoryear{Vovk}{Vovk}{2013}]{Vovk13}
Vovk, V. (2013).
\newblock Conditional validity of inductive conformal predictors.
\newblock {\em Machine Learning\/}~{\em 92\/}(2), 349--376.

\bibitem[\protect\citeauthoryear{Vovk, Gammerman, and Saunders}{Vovk
  et~al.}{1999}]{Vovk99}
Vovk, V., A.~Gammerman, and C.~Saunders (1999).
\newblock Machine-learning applications of algorithmic randomness.
\newblock In I.~Bratko and S.~Dzeroski (Eds.), {\em Proceedings of the
  Sixteenth International Conference on Machine Learning}, ICML '99, pp.\
  444--453. Morgan Kaufmann Publishers Inc.

\bibitem[\protect\citeauthoryear{Vovk, Gammerman, and Shafer}{Vovk
  et~al.}{2005}]{Vovk05}
Vovk, V., A.~Gammerman, and G.~Shafer (2005).
\newblock {\em Algorithmic Learning in a Random World}.
\newblock New York: Springer.

\bibitem[\protect\citeauthoryear{Vovk, Nouretdinov, and Gammerman}{Vovk
  et~al.}{2009}]{Vovk09}
Vovk, V., I.~Nouretdinov, and A.~Gammerman (2009).
\newblock On-line predictive linear regression.
\newblock {\em Ann. Statist.\/}~{\em 37\/}(3), 1566--1590.

\bibitem[\protect\citeauthoryear{Wald}{Wald}{1943}]{Wald43}
Wald, A. (1943).
\newblock An extension of {W}ilks' method for setting tolerance limits.
\newblock {\em Ann. Math. Statist.\/}~{\em 14\/}(1), 45--55.

\bibitem[\protect\citeauthoryear{White}{White}{1980a}]{White80b}
White, H. (1980a).
\newblock A heteroskedasticity-consistent covariance matrix estimator and a
  direct test for heteroskedasticity.
\newblock {\em Econometrica\/}~{\em 48\/}(4), 817--838.

\bibitem[\protect\citeauthoryear{White}{White}{1980b}]{White80a}
White, H. (1980b).
\newblock Using least squares to approximate unknown regression functions.
\newblock {\em Internat. Econom. Rev.\/}~{\em 21\/}(1), 149--170.

\bibitem[\protect\citeauthoryear{Wilks}{Wilks}{1941}]{Wilks41}
Wilks, S.~S. (1941).
\newblock Determination of sample sizes for setting tolerance limits.
\newblock {\em Ann. Math. Statist.\/}~{\em 12\/}(1), 91--96.

\bibitem[\protect\citeauthoryear{Wilks}{Wilks}{1942}]{Wilks42}
Wilks, S.~S. (1942).
\newblock Statistical prediction with special reference to the problem of
  tolerance limits.
\newblock {\em Ann. Math. Statist.\/}~{\em 13\/}(4), 400--409.

\end{thebibliography}
}

\begin{appendix}

\section{Discussion and further remarks}
\label{sec:discussion}

In this section we collect several further thoughts on the $k$-CV prediction intervals. We discuss some properties of the proposed method that we have established above but which we believe hold in much higher generality. We also draw some further connections to other methods such as sample splitting, tolerance regions and prediction regions based on non-parametric density estimation, and we provide further intuition. Finally, we sketch possible extensions and open problems.

\subsection{Predictor efficiency and interval length}
\label{sec:efficiency}

Recall that if $T_n\in\mathcal Z^n$ and $P$ are such that 
$$
s\mapsto\tilde{F}_n(s;T_n) = \P(y_0-\hat{\mu}_n(x_0)\le s\|T_n),
$$ 
is continuous, the optimal infeasible interval 
$$
PI_{\alpha_1,\alpha_2}^{(OPT)} = \hat{\mu}_n(x_0) + [\tilde{q}_{\alpha_1}, \tilde{q}_{\alpha_2}]
$$ 
in \eqref{eq:OPTPI} is the shortest interval of the form $\hat{\mu}_n(x_0) + [L(T_n), U(T_n)]$ such that \eqref{eq:alpha1} and \eqref{eq:alpha2} are satisfied. In this infeasible scenario and for given error probabilities $\alpha_1<\alpha_2$, the only way in which the data analyst can influence the length of $PI_{\alpha_1,\alpha_2}^{(OPT)}$ is via the choice of predictor $\hat{\mu}_n$. This choice clearly affects the conditional distribution $\tilde{F}_n$ of the prediction error $y_0-\hat{\mu}_n(x_0)$ and, thus, potentially its inter-quantile-range $\tilde{q}_{\alpha_2} - \tilde{q}_{\alpha_1}$, which is the length of the infeasible prediction interval. Since we only care about minimizing the inter-quantile-range of the conditional distribution $\tilde{F}_n$, for the rest of this subsection we consider the training data $T_n$ to be fixed and non-random. Thus, the predictor $\hat{\mu}_n:\R^{p}\to \R$ is also non-random. Now we would like to use a predictor $\hat{\mu}_n$ such that the prediction error $y_0-\hat{\mu}_n(x_0)$ has short inter-quantile-range. For simplicity, assume that $y_0 = \mu_P(x_0) + u_0$, where the error term $u_0$ has mean zero and is independent of the features $x_0$. Therefore, the prediction error is given by 
$$
y_0 - \hat{\mu}_n(x_0) \;=\; \mu_P(x_0)-\hat{\mu}_n(x_0) \;+\; u_0,
$$
i.e., the sum of the estimation error $\mu_P(x_0)-\hat{\mu}_n(x_0)$ and the innovation $u_0$. Following \citet{Lewis81}, we say that a continuous univariate distribution $P_1$ is more dispersed than $P_0$ if, and only if, any two quantiles of $P_1$ are further apart than the corresponding quantiles of $P_0$. Now we note that minimizing the inter-quantile-range of the prediction error $y_0-\hat{\mu}_n(x_0)$ is, in general, not equivalent to minimizing the inter-quantile-range of $\mu_P(x_0)-\hat{\mu}_n(x_0)$, because of the effect of the error term $u_0$. However, if the distribution of the error term $u_0$ has a log-concave density, then the distribution of $y_0-\hat{\mu}_{n,1}(x_0)$ is more dispersed than that of $y_0-\hat{\mu}_{n,0}(x_0)$, if, and only if, $\mu_P(x_0)-\hat{\mu}_{n,1}(x_0)$ is more dispersed than $\mu_P(x_0)-\hat{\mu}_{n,0}(x_0)$ \citep[see Theorem~8 of][]{Lewis81}. Thus, under log-concave error distributions, the interval length of $PI_{\alpha_1,\alpha_2}^{(OPT)}$ is directly related to the prediction accuracy of the point predictor $\hat{\mu}_n$ in use. 
These considerations naturally carry over to the feasible analog $PI_{\alpha_1,\alpha_2}^{(kCV)}$ defined in \eqref{eq:kCVPI}.
In Subsection~\ref{sec:PIlength}, in the special case of a linear model and ordinary-least-squares prediction, we have discussed the issue of interval length in some more detail and provided a rigorous description of the asymptotic interval length in a high-dimensional regime.
This sheds more light on the connection between the length of $PI_{\alpha_1,\alpha_2}^{(L1O)}$ and the estimation error $\mu_P(x_0) - \hat{\mu}_n(x_0)$. However, the lessons learned from the linear model appear to be valid in a much more general situation. In particular, we see that, at least for log-concave error distributions, the lengths of $k$-CV prediction intervals can potentially be used to evaluate the relative efficiency of competing predictors.

\subsection{The case of a naive predictor and sample splitting}
\label{sec:naive}

Next, we discuss the important special case where we naively decide to work with a predictor $M_{n,p}(T_n,x_0) = \mu(x_0)$, $\mu : \mathcal X\to\R$, that does not depend on the training data $T_n$ at all.\footnote{Note that this covers, in particular, the case where we do not even use, or do not have available, the feature vectors $x_0, \dots, x_n$, i.e., $\mu\equiv 0$. In this case, a \emph{prediction interval} for $y_0$ that is only based on $y_1,\dots, y_n$ is more commonly referred to as a \emph{tolerance interval}.} In this case, the predictor and its $k$-fold CV analogs all coincide and the ($k$-CV) residuals $\hat{u}_i = y_i-\mu(x_i)$ for $i=1,\dots, n$, are actually independent and identically distributed according to the non-random distribution $\tilde{F}_n(s) = \P(y_0-\mu(x_0)\le s \| T_n) = \P_{z_0}(y_0-\mu(x_0)\le s)$ and $\hat{F}_n$ is their empirical distribution function. Therefore, by Proposition~\ref{prop:basic} and the Dvoretzky-Kiefer-Wolfowitz (DKW) inequality \citep{Massart90}, if $\tilde{F}_n$ is continuous, we get for every $\eps>0$ that
\begin{align*}
\P\left(
\left| P_{z_0} \left( y_0 \in PI_{\alpha_1,\alpha_2}^{(kCV)}(T_n,x_0) \right) \;-\; (\alpha_2-\alpha_1)\right|  > \eps
\right)
\;\le\;
4\exp\left(-\frac{n\eps^2}{2}\right).
\end{align*}
Integrating this tail probability also yields
\begin{align*}
\E\left[
\left| P_{z_0} \left( y_0 \in PI_{\alpha_1,\alpha_2}^{(kCV)}(T_n,x_0) \right) \;-\; (\alpha_2-\alpha_1)\right|  
\right]
\;\le\;
\sqrt{\frac{8\pi}{n}}.
\end{align*}
We also point out that in the present case where the predictor does not depend on $T_n$, the problem of constructing a prediction interval for $y_0$ can actually be reduced to finding a non-parametric univariate tolerance interval for $y_0-\mu(x_0)$ based on the i.i.d. copies $(y_i-\mu(x_i))_{i=1}^n$. For this problem classical solutions are available, based on the theory of order statistics of i.i.d. data \citep[cf.][Chapter~8]{Krish09}. Unfortunately, the problem changes dramatically, once we try to learn the true regression function $\mu_P$ from the training data $T_n$ and use $\hat{\mu}_n(x_0) = M_{n,p}(T_n,x_0)$ to predict $y_0$, because then the $k$-CV residuals are no longer independent and the conditional distribution function $\tilde{F}$ of the prediction error $y_0-\hat{\mu}_n(x_0)$ given $T_n$ is random. Thus, in the general case we can not expect to obtain equally powerful and elegant results as above and we can not resort to the theory of order statistics of i.i.d. data. In particular, we note that the bound of Theorem~\ref{thm:density} is still somewhat sub-optimal in this trivial case where the estimator does not depend on the training sample $T_n$. In that case, $\eta_{n,k}=0$, but the derived bound still depends on the distribution of the estimation error $\mu_P(x_0) - \mu(x_0)$, even though in that case the alternative bound obtained above by the DKW inequality does no longer involve the estimation error. It is an open problem to establish a concentration inequality for $\|\hat{F}_n-\tilde{F}_n\|_\infty$ analogous to the DKW inequality but in the general case of dependent $k$-CV residuals and random $\tilde{F}_n$.

The discussion of the previous paragraph also applies to the case where the predictor $\mu$ was obtained as an estimator for $\mu_P$, but from another independent training sample $S_{n^*} = (x_j^*, y_j^*)_{j=1}^{n^*}$ of $n^*$ i.i.d. copies of $(x_0,y_0)$. This situation can be seen as a sample splitting method, where $n^*$ of the overall $n+n^*$ observations are used to compute the point predictor $\mu=\hat{\mu}_{n^*}$ and the remaining $n$ observations in $T_n$ are used as a validation set to estimate the conditional distribution of the prediction error $y_0 - \hat{\mu}_{n^*}(x_0)$ given $S_{n^*}$ (and $T_n$), from the (conditionally on $S_{n^*}$) i.i.d. residuals $y_i-\hat{\mu}_{n^*}(x_i)$, $i=1,\dots, n$. Such a procedure is discussed, for instance, by \citet{Lei16} and \citet{Vovk12}. Note that under the assumptions of the previous paragraph, such a method is asymptotically training conditionally valid if the size $n$ of the validation set diverges to infinity. However, this method uses only $n^*$ of the $n+n^*$ available observation pairs for prediction, such that the point predictor $\hat{\mu}_{n^*}$ based on $S_{n^*}$ is not as efficient as the analogous predictor based on the full sample $S_{n^*}\cup T_n$. This typically results in a larger prediction interval than necessary, because then the conditional distribution of the prediction error $y_0 - \hat{\mu}_{n^*}(x_0)$ is usually more dispersed than that of $y_0 - \hat{\mu}_{n^*+n}(x_0)$. See also the discussion in Subsections~\ref{sec:SS} and \ref{sec:efficiency}.


\subsection{Further remarks}

\begin{remark}[On exact training conditional validity]\normalfont
\label{rem:no-exact-validity}
Suppose that the class $\mathcal P$ contains at least the data generating distributions $P_0$ and $P_1$, where for $j\in\{0,1\}$
$$
P_j \;=\; \mathcal N_{p+1}(0,\sigma_j^2 I_{p+1}), \quad \sigma_j^2>0, \;\sigma_0^2\ne\sigma_1^2,
$$
that is, the $p$ feature variables in $x_0$ and the response $y_0$ are i.i.d. $\mathcal N(0,\sigma_j^2)$. Further, suppose that we decide to predict $y_0$ by some linear predictor $\hat{\mu}_n(x_0)=x_0'\hat{\beta}_n$ for some estimator $\hat{\beta}_n = \hat{\beta}_n(T_n)$. We shall show that for every $\alpha\in(0,1/2)$, it is impossible to construct a prediction interval of the form $PI_\alpha(T_n,x_0) = x_0'\hat{\beta}_n + [L_\alpha(T_n), U_\alpha(T_n)]$ based on a finite sample $T_n$ and $x_0$, such that \eqref{eq:valid} is equal to zero.
\begin{proof}
If \eqref{eq:valid} is equal to zero, then for both $j=0,1$ and $P_j^n$-almost all samples $T_n\in \mathcal Z^n$, where $P_j^n$ is the $n$-fold product measure of $P_j$, 
\begin{align*}
1-\alpha &= P_j(y_0\in PI_\alpha(T_n,x_0)) \\
&=
P_j(L_\alpha(T_n) \le y_0 - x_0'\hat{\beta}_n(T_n) \le U_\alpha(T_n)\|T_n) \\
&= \Phi\left( \frac{U_\alpha(T_n)}{\sigma_j\sqrt{\|\hat{\beta}_n\|_2^2+1}} \right) - \Phi\left( \frac{L_\alpha(T_n)}{\sigma_j\sqrt{\|\hat{\beta}_n\|_2^2+1}} \right).
\end{align*}
Since $1-\alpha>1/2$, we must have $L_\alpha < 0 < U_\alpha$, almost surely, and it is easy to see that the function 
$$
g_{l,u}(\nu) := 
\Phi\left( \frac{u}{\nu} \right) - \Phi\left( \frac{l}{\nu} \right), \quad g_{l,u} : (0,\infty) \to (0,1),
$$
is continuous and strictly decreasing, provided that $l<0<u$, and thus, for such $l$ and $u$, $g_{l,u}$ is invertible. Therefore, for $j=0,1$ and for $P_j^{n}$-almost all samples $T_n$ (equivalently, for Lebesgue almost all $T_n\in \R^{n(p+1)}$), we have
$$
\tilde{\sigma}_n^2(T_n) := \left(\frac{g_{L_\alpha, U_\alpha}^{-1}(1-\alpha)}{\sqrt{\|\hat{\beta}_n\|_2^2+1}}\right)^2 = \sigma_j^2.
$$
In other words, there exists $T_n\in\mathcal Z^n$, such that $\sigma_0^2=\tilde{\sigma}^2_n(T_n) = \sigma_1^2$, a contradiction.
\end{proof}
\end{remark}

\begin{remark}\normalfont\label{rem:Dicker}
Consistent estimation of the true regression function $\mu_P:\mathcal X\to \R$ from an i.i.d. sample of size $n$ is usually not possible if the dimension $p$ of $\mathcal X$ is non-negligible compared to $n$, even if $\mu_P$ is very regular and the error is homoskedastic and has rapidly decaying tails. For example, in a Gaussian linear model where the only unknown parameter is the $p$-vector $\beta$ of regression coefficients, it is impossible to consistently estimate the conditional mean $\mu_P(x_0)=\E_{z_0}[y_0\|x_0] = \beta'x_0$, unless $p/n\to0$ or strong assumptions are imposed on the parameter space \citep[cf.][]{Dicker12}. 
\end{remark}

\begin{remark}\normalfont\label{rem:ObjValid}
A natural approach for constructing non-parametric prediction sets is to estimate the conditional density of $y_0$ given $x_0$ (if it exists), because, as can be easily shown, a highest density region of the conditional density of $y_0$ given $x_0$ provides the smallest (in terms of Lebesgue measure) prediction region $PR_\alpha(x_0)$ for $y_0$ that controls the conditional coverage probability given $x_0$, i.e., that satisfies
\begin{align}\label{eq:condXvalid}
P(y_0\in PR_\alpha(x) \| x_0 = x) \;\ge\; 1-\alpha \quad \text{for $P_{x_0}$-almost all $x$}.
\end{align}
This condition has been called \emph{object conditional validity} by \citet{Vovk13}. However, object conditional validity is often too much to ask for. First of all, as shown by \citet{Barber19} \citep[see also][]{Vovk13, Lei14}, for continuous distributions there are no non-trivial prediction sets based on a finite sample that satisfy \eqref{eq:condXvalid}. Moreover, even if we are content with \emph{asymptotic} object conditional validity, learning the relevant properties of the conditional density of $y_0$ given $x_0$ is typically only possible if the dimension of the feature vector $x_0$ is much smaller than the available sample size (cf. Remark~\ref{rem:Dicker}). Finally, even under stability assumptions of the underlying predictor, object conditional validity seems to be very hard. Therefore, since our focus in the present paper is on high-dimensional problems, we do not aim at object conditional validity.

\end{remark}

\begin{remark}[Prediction intervals with adaptive lengths]\normalfont
The length of the $k$-CV prediction interval in \eqref{eq:kCVPI}, as it stands, does not depend on the value of $x_0$. An immediate way to account for that is the following.
Consider, in addition, an estimator $\hat{\sigma}_n^2(x) = S(T_n,x)$ of the conditional variance $\Var[y_0\|x_0=x]$. Then a prediction interval can be computed as $\hat{\mu}_n(x_0) + [\hat{q}_{\alpha_1},\hat{q}_{\alpha_2}]\hat{\sigma}_n(x_0)$, where now, $\hat{q}_\alpha$ is an empirical $\alpha$-quantile of the $k$-CV residuals 
$$
\hat{u}_i = \frac{y_i-\hat{\mu}_n^{(l)}(x_i)}{\hat{\sigma}_{n,(l)}(x_i)}, \quad l\in[k], i\in K_l.
$$ 
An alternative construction of prediction intervals with adaptive length is given in \citet{Romano19} who consider a quantile regression procedure based on sample splitting.
\end{remark}

\begin{remark}[Computational versus statistical efficiency]\normalfont\label{rem:computation}
The computational burden of the $k$-CV prediction interval increases roughly linearly with the number of folds $k$. The model has to be re-fitted $k$-times on each of the possible reduced samples $T_n\setminus K_l$, $l=1,\dots, k$, in order to compute the residuals $\hat{u}_i = y_i - \hat{\mu}_n^{(l)}(x_i)$, $i\in K_l$. Thus, reducing the number of folds $k$ will speed up the computation of the procedure. However, choosing $k$ too small may invalidate the stability of the algorithm (i.e., $\eta_{n,k_n}\to0$ in Definition~\ref{DEF:STABLE}), because the predictor $\hat{\mu}_n$ based on the full training sample $T_n$ may be very different from the predictor $\hat{\mu}^{(l)}_n$ based on the reduced sample $T_n\setminus K_l$, if (for some $l$) $n-m_l = n-|K_l|$ is too small compared to $n$. As a consequence, the $k$-CV prediction interval may loose its approximate training conditional validity property and it may get overly wide, because the empirical distribution of the $k$-CV residuals $y_i-\hat{\mu}_n^{(l)}(x_i)$, $i\in K_l$, may be much more dispersed than the conditional distribution of the prediction error $y_0 - \hat{\mu}_n(x_0)$.

In some special cases, however, it is possible to devise a shortcut for the computation of these residuals. For example, in case of ordinary least squares prediction $\hat{\mu}_n(x) = x'\hat{\beta}_n = x'(X'X)^\dagger X'Y$, if $X_{[i]}'X_{[i]}$ has full rank, we have the well known identity for the leave-one-out residuals ($k=n$),
$$
\hat{u}_i = y_i-x_i'\hat{\beta}_n^{[i]} = \frac{y_i-x_i'\hat{\beta}_n}{1-x_i'(X'X)^{-1}x_i},
$$
such that the $n$-vector of leave-one-out residuals can be computed as
$$
\left[\diag(I_n-X(X'X)^{-1}X')\right]^{-1}(I_n - X(X'X)^{-1}X')Y.
$$
Hence, the model has to be fitted only once instead of $n$-times. 
\end{remark}


\section{Numerical results for heteroskedastic data}
\label{sec:mcApp}
We repeated the simulations from Section~\ref{sec:mc} for heteroskedastic $y_i$ generated as
$$
y_i\|x_i \thicksim \mathcal N\left(\beta'x_i,\sigma^2(x_i)\right), \quad \sigma^2(x_i) = 1+\frac{|\beta'x_i|}{\sqrt{10}} - \frac{\sqrt{2}}{\Gamma(\frac12)}.
$$
For comparability we have chosen the conditional standard deviation of $y_i$ such that $\E[\sigma^2(x_i)] = 1$ in order to achieve $\Var(y_i) = 11$, as in the homoskedastic case. Histograms of conditional coverage probabilities of the leave-one-out procedure ($k=n$) for different sample sizes are shown in Figure~\ref{fig:AsympHet}. Unconditional coverage probability, expected interval length and estimated stability are depicted in Figure~\ref{fig:Asymp-Cov-Len-StabHet}. The results are basically indistinguishable from the homoskedastic case. Studying the impact of the number of folds $k$ on the coverage probabilities (Figure~\ref{fig:Hists-np100-HET}) and on the expected interval length (Figure~\ref{fig:Cov-Len-Stab-HET}), we see a slightly worse performance of the $k$-CV procedure on heteroskadastic data compared to the homoskedastic case in the sense that conditional coverage probabilities appear to be a little bit more spread out around the nominal level and intervals tend to be somewhat wider.

\begin{figure}[htbp]
\includegraphics[width=\textwidth]{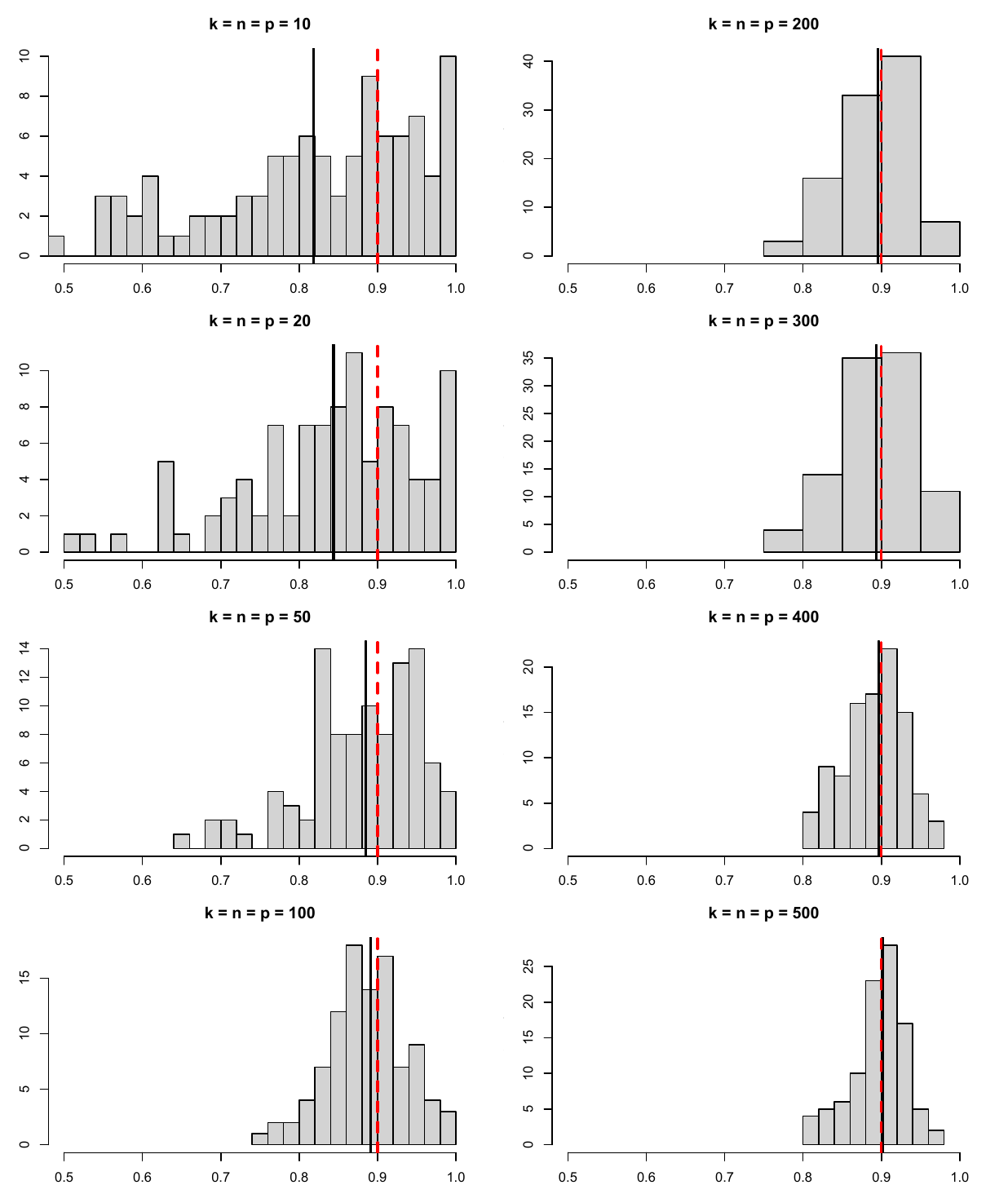} 
\caption{Histograms of conditional coverage probabilities $\P(y_0\in PI^{(L1O)}_{\alpha}(T_n,x_0) \| T_n)$ from $100$ Montecarlo training samples $T_n$ for different values of $n$ at confidence level of $1-\alpha = 0.9$ (red dashed line) based on heteroskedastic data. Here $k=p=n$.}
\label{fig:AsympHet}
\end{figure}

\begin{figure}[htbp]
\includegraphics[width=\textwidth]{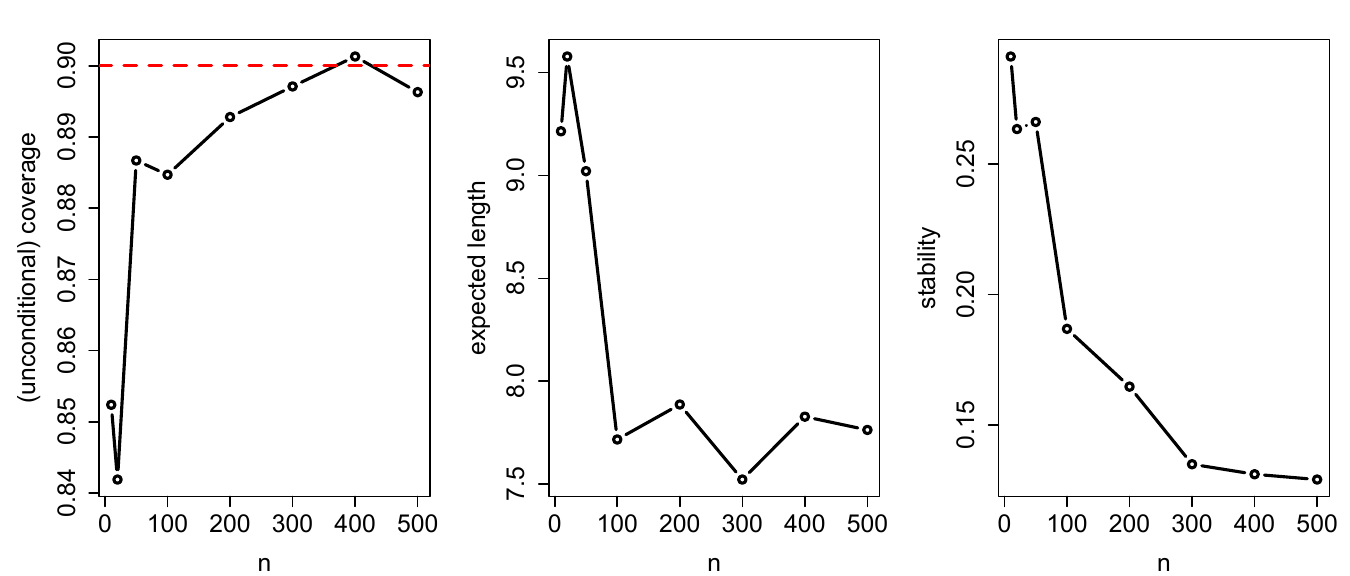} 
\caption{Unconditional coverage probability, expected interval length and estimated stability for different sample sizes $n$ from $100$ Montecarlo simulations with $k=n=p$ (i.e., leave-one-out) based on heteroskedastic data. }
\label{fig:Asymp-Cov-Len-StabHet}
\end{figure}

\begin{figure}[tbp]
\includegraphics[width=\textwidth]{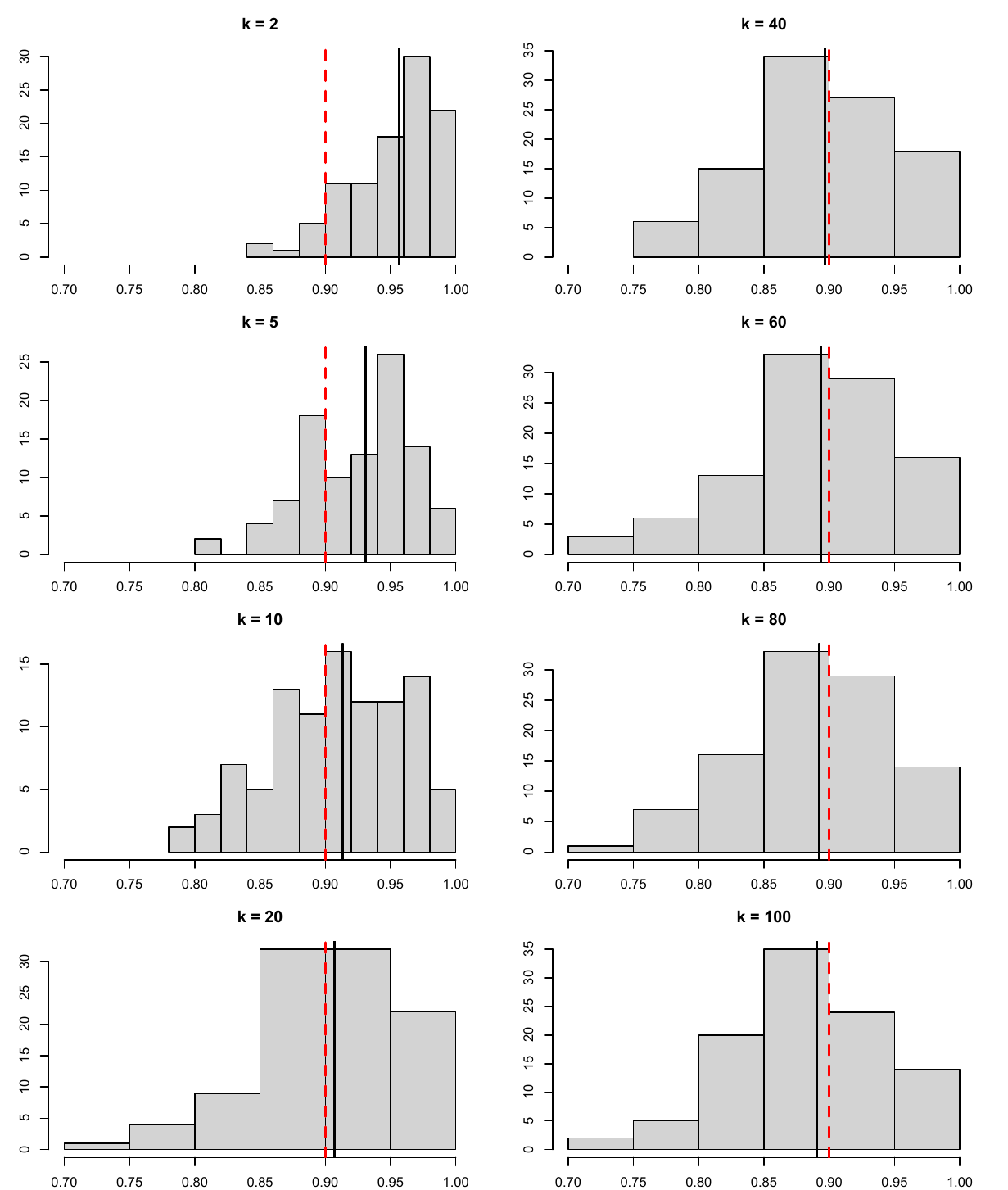} 
\caption{Histograms of conditional coverage probabilities $\P(y_0\in PI^{(kCV)}_{\alpha}(T_n,x_0) \| T_n)$ from $100$ Montecarlo training samples $T_n$ for different values of $k$  at confidence level of $1-\alpha = 0.9$ (red dashed line) based on heteroskedastic data. Here $n=p=100$.}
\label{fig:Hists-np100-HET}
\end{figure}

\begin{figure}[tbp]
\includegraphics[width=\textwidth]{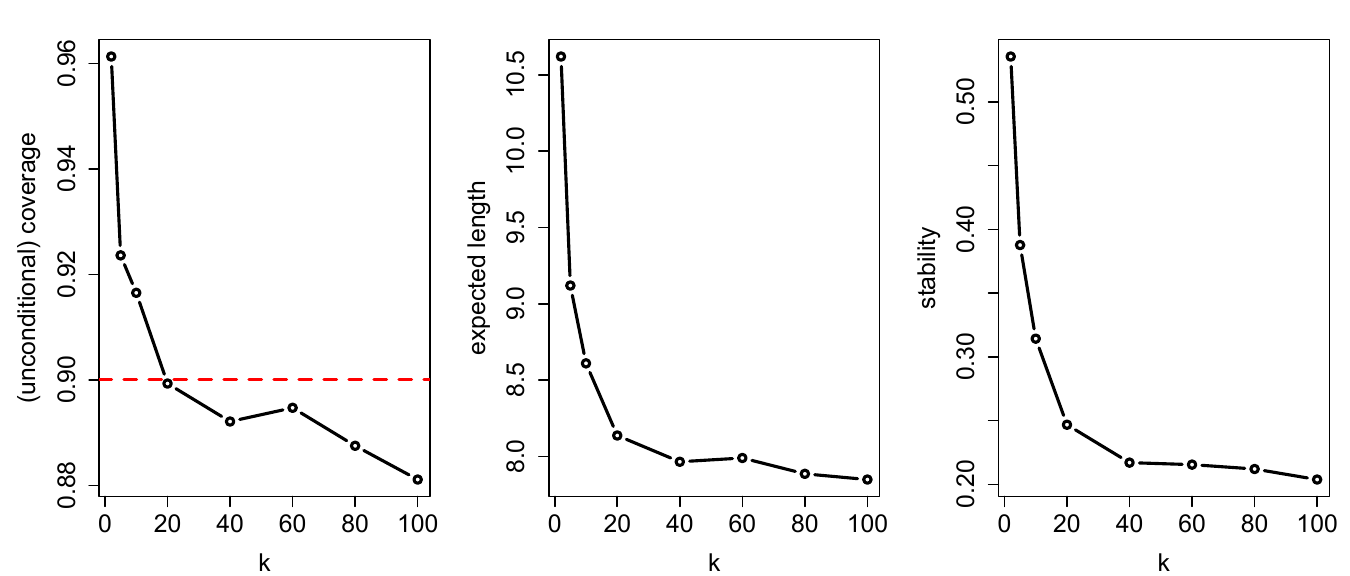} 
\caption{Unconditional coverage probability, expected interval length and estimated stability for different numbers of folds $k$ from $100$ Montecarlo simulations with $n=p=100$ based on heteroskedastic data.}
\label{fig:Cov-Len-Stab-HET}
\end{figure}


\section{Proofs of main results}
\label{sec:MainProofs}

We repeatedly use the following argument: 
Suppose that $A_n$, $n\ge 1$, are real-valued random variables on $(\Omega, \mathcal F, P)$ and $\mathcal F_n$, $n\ge 1$, are sub-sigma fields of $\mathcal F$. If $\E[|A_n|\| \mathcal F_n]$ is well-defined and converges to zero in probability, then $A_n$ converges to zero in probability. Indeed
\begin{align*}
P(|A_n|>\eps) 
= \E[ P(|A_n|>\eps\|\mathcal F_n) \land 1] 
\le
\E\left[ \left( \frac{1}{\eps}\E[|A_n|\|\mathcal F_n] \right)\land 1 \right]. 
\end{align*}
In this upper bound, the integrand converges to zero in probability by assumption, and this integrand is bounded by $1$ by construction. In view of the dominated convergence theorem, this upper bound converges to zero.
Moreover, if $\E[A_n\| \mathcal F_n]$ is well-defined and converges to $c\in\R$ in probability and if, in addition, $\Var[A_n\|\mathcal F_n]\to 0$ in probability, then $A_n\to c$
in probability. Indeed,
\begin{align*}
\E[(A_n-c)^2\|\mathcal F_n] = \Var[A_n\|\mathcal F_n] + (\E[A_n\|\mathcal F_n] - c)^2
\end{align*}
converges to zero in probability. From the preceding statement the claim follows.

The following result is an extension of Lemma~9 (Equation~(9)) in \citet{Bousquet02} \citep[see also][]{Devroye79} to the case of $k$-fold cross validation and potentially non-symmetric learning algorithms.

\begin{lemma}\label{lemma:BousquetCV}
For $n,p\in\N$, consider a learning algorithm $M_{n,p}:\mathcal Z^n\times\mathcal X\to\R$ and, for measurable $f:\mathcal X\to\R$ and $z=(x,y)\in\mathcal Z$, let $L(f,z) := c(y,f(x))$ be a loss function with $0\le c(y,y')\le C<\infty$, $y'\in\R$. For the training sample $T_n\in\mathcal Z^n$ and by a slight abuse of notation, write $M_{n,p}(T_n) := M_{n,p}(T_n,\cdot) : \mathcal X\to\R$ and define
$$
R := R(T_n) := \E_{z_0}[L(M_{n,p}(T_n), z_0)] = \E_{z_0}[c(y_0,M_{n,p}(T_n,x_0))]
$$
and
$$
R_{CV} := R_{CV}(T_n) := \frac1k \sum_{l=1}^k \frac{1}{m_l} \sum_{i\in K_l} L(M_{n-m_l,p}(T_n\setminus K_l),z_i).
$$
Then we have
\begin{align*}
\E&\left[ \left( R_{CV} - R\right)^2\right] \le \\
&\frac{C^2}{4(k-1)} 
+
5C\frac1k\sum_{l=1}^k \E\left[ |L(M_{n-m_l,p}(T_n\setminus K_l),z_0) - L(M_{n,p}(T_n),z_0)|\right]. 
\end{align*}

\end{lemma}

\begin{remark}\normalfont
Notice that the upper bound of Lemma~\ref{lemma:BousquetCV} does not simply reduce to the one of the analogous result of \citet{Bousquet02} in case $k=n$. On the one hand, the first term in the present upper bound (for $k=n$) is actually smaller than the corresponding term in \citet{Bousquet02} by a factor of $\frac{n}{2(n-1)}\le1$. This is due to a slightly more refined technique (cf. Lemma~\ref{lemma:doubleSum}) which can also be used to improve the result of \citet{Bousquet02}. On the other hand, the second term in our upper bound is larger than the corresponding one of \citet{Bousquet02} by a factor of $\frac53$. Notice, however, that unlike the result of those authors, the present formulation does not require the learning algorithm to be symmetric.
\end{remark}

\begin{proof}[Proof of Lemma~\ref{lemma:BousquetCV}]
Let $z,z'\in\mathcal Z$ and abbreviate $T=T_n=(z_1, \dots, z_n)$. For simplicity, and by another slight abuse of notation, we write $L(T,z) = L(M_{n,p}(T_n),z)$ and $L(T\setminus K_l,z)=L(M_{n-m_l,p}(T_n\setminus K_l),z)$, etc. We will repeatedly make use of the fact that the distribution of any function taking i.i.d. random vectors as inputs, is invariant under permutation of those inputs. Now
\begin{align*}
\E[R^2] &= \E_T \left[ \left( \E_{z}[L(T,z)]\right)^2\right] = \E_T \left[\E_z[L(T,z)] \E_{z'} [L(T,z')] \right] \\
&= \E[L(T,z) L(T,z')],
\end{align*}
and
\begin{align*}
\E[R\cdot R_{CV}] &= \frac1k \sum_{l=1}^k \frac{1}{m_l} \sum_{i\in K_l} \E \left[ R\cdot L(T\setminus K_l,z_i) \right]\\
&=
\frac1k \sum_{l=1}^k \frac{1}{m_l} \sum_{i\in K_l} \E[L(T,z) L(T\setminus K_l,z_{i})].
\end{align*}
For the squared CV error, we get
\begin{align*}
&\E[R_{CV}^2] = \frac{1}{k^2} \sum_{l=1}^k \E\left[ \left( \frac{1}{m_l} \sum_{i\in K_l} L(T\setminus K_l, z_i) \right)^2\right]\\
&\quad+
\frac{1}{k^2} \sum_{l_1\neq l_2} \frac{1}{m_{l_1}m_{l_2}} \sum_{i_1\in K_{l_1}} \sum_{i_2\in K_{l_2}}\underbrace{\E\left[ L(T\setminus K_{l_1},z_{i_1}) L(T\setminus K_{l_2},z_{i_2})\right]}_{=:E(l_1,l_2,i_1,i_2)}\\
&\le 
 \frac{C}{k^2} \sum_{l=1}^k \frac{1}{m_l} \sum_{i\in K_l} \E[L(T\setminus K_l, z_i)]\\
&\quad+
\frac{k(k-1)}{k^2} \frac{1}{k(k-1)}\sum_{l_1\neq l_2}\frac{1}{m_{l_1}m_{l_2}} \sum_{i_1\in K_{l_1}} \sum_{i_2\in K_{l_2}}E(l_1,l_2,i_1,i_2)\\
&=
\frac{1}{k(k-1)}\sum_{l_1\neq l_2}\frac{1}{m_{l_1}m_{l_2}} \sum_{i_1\in K_{l_1}} \sum_{i_2\in K_{l_2}}E(l_1,l_2,i_1,i_2)\\
&\quad+
\frac{1}{k} \frac{1}{k(k-1)}\sum_{l_1\neq l_2} \frac{1}{m_{l_1}m_{l_2}} \sum_{i_1\in K_{l_1}} \sum_{i_2\in K_{l_2}}\E\left[ L(T\setminus K_{l_1}, z_{i_1}) [C - L(T\setminus K_{l_2}, z_{i_2})] \right]\\
&=
\frac{1}{k(k-1)}\sum_{l_1\neq l_2}\frac{1}{m_{l_1}m_{l_2}} \sum_{i_1\in K_{l_1}} \sum_{i_2\in K_{l_2}}E(l_1,l_2,i_1,i_2)\\
&\quad+
\frac{1}{k} \frac{1}{k(k-1)}\sum_{l_1\neq l_2} \E\left[ \left(\frac{1}{m_{l_1}} \sum_{i_1\in K_{l_1}}L(T\setminus K_{l_1}, z_{i_1})\right) \left(C - \frac{1}{m_{l_2}} \sum_{i_2\in K_{l_2}} L(T\setminus K_{l_2}, z_{i_2})\right) \right].
\end{align*}
Now, using Lemma~\ref{lemma:doubleSum} together with boundedness of the loss, we see that the expression on the last line of the previous display is bounded by
\begin{align*}
\frac{C^2}{4k} \frac{k}{k-1}.
\end{align*}
Combining the previously derived facts, we therefore arrive at
\begin{align*}
&\E\left[ \left( R_{CV} - R\right)^2\right] \\
&\quad\le 
\frac{C^2}{4k} \frac{k}{k-1} + \frac{1}{k(k-1)}\sum_{l_1\neq l_2} \frac{1}{m_{l_1}m_{l_2}} \sum_{i_1\in K_{l_1}}\sum_{i_2\in K_{l_2}}\\
&\hspace{1cm} \Big(
\underbrace{\E[L(T,z) L(T,z')] - \E[L(T,z) L(T\setminus K_{l_2},z_{i_2})]}_{=:I_1} \\
&\quad\quad\quad+ \underbrace{\E\left[ L(T\setminus K_{l_1},z_{i_1}) L(T\setminus K_{l_2},z_{i_2})\right]-\E\left[L(T,z) L(T\setminus K_{l_2},z_{i_2})\right]}_{=:I_2}
\Big).
\end{align*}
Next, we repeatedly use exchangeability of $z_1,\dots, z_n, z,z'$. Moreover, by symbols like $T_{z_{i_2}\leftarrow z'}$ we denote the training data set $T$ where the observation $z_{i_2}$ was replaced by $z'$. Also recall that $i_1\in K_{l_1}$ and $i_2\in K_{l_2}$. Thus, for $I_1$ we get
\begin{align*}
I_1 &= 
\E[L(T,z) L(T,z')] - \E[L(T_{z_{i_2}\leftarrow z'},z) L(T\setminus K_{l_2},z')]\\
&=
\E\left[
\left( L(T,z) - L(T\setminus K_{l_2},z)\right) L(T,z')
\right]\\
&\quad+
\E\left[
\left(L(T\setminus K_{l_2},z) - L(T_{z_{i_2}\leftarrow z'},z)\right) L(T,z')
\right]\\
&\quad+
\E\left[
\left(L(T,z') - L(T\setminus K_{l_2},z')\right) L(T_{z_{i_2}\leftarrow z'},z)
\right]\\
&=: I_1(1) + I_1(2) + I_1(3).
\end{align*}
But exchanging $z_{i_2}$ and $z'$, we get 
$$
I_1(2) = \E\left[
\left(L(T\setminus K_{l_2},z) - L(T,z)\right) L(T_{z_{i_2}\leftarrow z'},z_{i_2})
\right],
$$
and 
\begin{align*}
I_1 &= 
I_1(3) + 
\E\left[
( L(T,z) - L(T\setminus K_{l_2},z)) ( L(T,z') - L(T_{z_{i_2}\leftarrow z'},z_{i_2}))
\right]\\
&\le 
2C\E\left[|L(T,z) - L(T\setminus K_{l_2},z)|\right],
\end{align*}
because $L$ is bounded between $0$ and $C$ and hence the modulus of the difference of two $L$ functions is bounded by $C$. 
Moreover, for $I_2$ we similarly get 
\begin{align*}
I_2 &= 
\E\left[ L(T_{z_{i_2}\leftarrow z'}\setminus K_{l_1},z) L(T_{z_{i_1}\leftarrow z}\setminus K_{l_2},z')\right]  \\
&\hspace{2cm}-\E\left[L(T_{z_{i_2}\leftarrow z'},z) L(T\setminus K_{l_2},z')\right] \\
&=
\E\left[
\left(L(T_{z_{i_1}\leftarrow z}\setminus K_{l_2},z') - L(T\setminus K_{l_1},z')\right) 
L(T_{z_{i_2}\leftarrow z'}\setminus K_{l_1},z)
\right] \\
&\quad+
\E\left[
\left(L(T\setminus K_{l_1},z') - L(T\setminus K_{l_2},z')\right) 
L(T_{z_{i_2}\leftarrow z'}\setminus K_{l_1},z)
\right] \\
&\quad+
\E\left[
\left(L(T_{z_{i_2}\leftarrow z'}\setminus K_{l_1},z) - L(T_{z_{i_2}\leftarrow z'},z)\right) L(T\setminus K_{l_2},z')
\right]\\
&=: I_2(1) + I_2(2) + I_2(3).
\end{align*}
Exchanging $z_{i_1}$ with $z$ in $I_2(1)$ and $z_{i_2}$ with $z'$ in $I_2(3)$ results in
$$
I_2(1) = \E\left[
\left(L(T\setminus K_{l_2},z') - L(T\setminus K_{l_1},z')\right) 
L(T_{z_{i_2}\leftarrow z'}\setminus K_{l_1},z_{i_1})
\right] 
$$
and
$$
I_2(3) = \E\left[
\left(L(T\setminus K_{l_1},z) - L(T,z)\right) L(T\setminus K_{l_2},z_{i_2})
\right].
$$
Thus, 
\begin{align*}
I_2 &= I_2(3) + \E\left[
\Big(L(T\setminus K_{l_2},z') - L(T\setminus K_{l_1},z')\Big) \right.\\
&\hspace{3cm}\left.\times\left(L(T_{z_{i_2}\leftarrow z'}\setminus K_{l_1},z_{i_1})-L(T_{z_{i_2}\leftarrow z'}\setminus K_{l_1},z) \right)
\right] \\
&\le
2C\E\left[
|L(T\setminus K_{l_1},z) - L(T,z)|\right]
+
C\E\left[
|L(T\setminus K_{l_2},z) - L(T,z)|
\right].
\end{align*}
Therefore, we conclude that
\begin{align*}
\E\left[ \left( R_{CV} - R\right)^2\right] \le
\frac{C^2}{4k} \frac{k}{k-1}
+
5C\frac1k\sum_{l=1}^k \E\left[ |L(T\setminus K_l,z) - L(T,z)|\right].
\end{align*}
\end{proof}


\subsection{Proof of Theorem~\ref{thm:density}}
\label{sec:proof:thm:density}

The following is an immediate consequence of Lemma~\ref{lemma:BousquetCV} applied with the loss function $L(f,z) = \mathbbm{1}_{(-\infty,s]}(y-f(x))$, $f:\mathcal X\to\mathcal Y$, $z=(x,y)\in\mathcal Z$, $s\in\R$ and $C=1$. 

\begin{lemma}
\label{lemma:Bousquet}
For every $s\in\R$, we have
\begin{align*}
&\E\left[\left(\hat{F}_n(s) - \tilde{F}_n(s)\right)^2\right] \\
&\le\; \frac{1}{4(k-1)} \;+\; \frac5k \sum_{l=1}^k\E\left[\left|\mathbbm{1}_{(-\infty, s]}(y_0-\hat{\mu}_n(x_0)) - \mathbbm{1}_{(-\infty, s]}(y_0-\hat{\mu}_n^{(l)}(x_0)) \right|\right].
\end{align*}
\end{lemma}

It is elementary to relate the upper bound of Lemma~\ref{lemma:Bousquet} to the stability coefficient (cf. Definition~\ref{DEF:STABLE}) of $\hat{\mu}_n$. We defer the proof of the next result until the end of this section.

\begin{lemma}\label{lemma:elementary}
Suppose that $y_0$ has a conditional Lebesgue density $f_{y_0\|x_0}$ given $x_0$. Then, for every $s\in\R$ and every $l=1,\dots, k$,
\begin{align*}
&\E\left[\left|\mathbbm{1}_{(-\infty, s]}(y_0-\hat{\mu}_n(x_0)) - \mathbbm{1}_{(-\infty, s]}(y_0-\hat{\mu}_n^{(l)}(x_0)) \right|\right]\\
&\quad\le
\E\left[
	\left(\|f_{y_0\|x_0}\|_\infty |\hat{\mu}_n(x_0) - \hat{\mu}_n^{(l)}(x_0)| \right)\land 1\right].
\end{align*}
\end{lemma}

To turn the pointwise bound of Lemma~\ref{lemma:Bousquet} into a uniform one, we need a certain continuity and tightness property of $\tilde{F}_n$. 

\begin{lemma}\label{lemma:Ftilde}
Fix a training sample $T_n\in\mathcal Z^n$.
\begin{enumerate}
	\setlength\leftmargin{-20pt}
	\renewcommand{\theenumi}{(\roman{enumi})}
	\renewcommand{\labelenumi}{{\theenumi}} 

\item \label{l:Lipschitz} Suppose that the conditional distribution of $y_0$ given $x_0$ has Lebesgue density $f_{y_0\|x_0}$. If $s_1, s_2\in\R$, then
$$
\left|\tilde{F}_n( s_1; T_n) - \tilde{F}_n(s_2; T_n)\right| \;\le\; \E_{x_0}\left[ (\|f_{y_0\|x_0}\|_\infty |s_1-s_2|)\land1 \right].
$$

\item \label{l:Tightness} For $c_1,c_2\in(0,\infty)$ and $g:\mathcal X\to\R$ measurable we have
\begin{align*}
&1-\tilde{F}_n( c_1+c_2; T_n) + \tilde{F}_n(-(c_1+c_2); T_n) \\
&\quad \le
P \left( |y_0-g(x_0)| \ge c_1\right) + \P_{x_0}(|\hat{\mu}_n(x_0)-g(x_0)| > c_2).
\end{align*}

\end{enumerate}
\end{lemma}

We provide the proofs of Lemma~\ref{lemma:elementary} and Lemma~\ref{lemma:Ftilde} below, after the main argument is finished.
The proof of Theorem~\ref{thm:density} is now a finite sample version of the proof of Polya's theorem. 
Fix $g:\mathcal X\to \R$ measurable, positive real numbers $\eps, c_1, c_2$, and set $\overline{s} = c_1+c_2$ and $\underline{s} = -c_1-c_2$. 
We split up the interval $[\underline{s},\overline{s})$ into $L$ intervals $[s_{j-1},s_j)$, $j=1,\dots, L$, with endpoints $\underline{s} =: s_0 < s_1 < \dots < s_L := \overline{s}$, such that $s_j - s_{j-1} \le \eps$. We may thus take $L = \lceil (\overline{s}-\underline{s})/\eps\rceil = \lceil 2(c_1+c_2)/\eps\rceil$. If $s<s_0$, then
\begin{align*}
\hat{F}_n(s) - \tilde{F}_n(s) \;&\ge\; - \tilde{F}_n(s_0) \ge - |\hat{F}_n(s_0) - \tilde{F}_n(s_0)| - \tilde{F}_n(s_0),\\
\hat{F}_n(s) - \tilde{F}_n(s) \;&\le\; \hat{F}_n(s_0) \le |\hat{F}_n(s_0) - \tilde{F}_n(s_0)| + \tilde{F}_n(s_0).
\end{align*}
Furthermore, if $s\ge s_L$, then
\begin{align*}
\hat{F}_n(s) - \tilde{F}_n(s) \;&\ge\; \hat{F}_n(s_L) - 1 \\&\ge\;  - |\hat{F}_n(s_L) - \tilde{F}_n(s_L)| - \left( 1-\tilde{F}_n(s_L)\right), \\
\hat{F}_n(s) - \tilde{F}_n(s) \;&\le\; 1 - \tilde{F}_n(s_L) \\&\le\; |\hat{F}_n(s_L) - \tilde{F}_n(s_L)| + 1- \tilde{F}_n(s_L).
\end{align*}
Finally, for $j\in\{1,\dots, L\}$ and $s\in[s_{j-1},s_j)$,
\begin{align*}
\hat{F}_n(s) - \tilde{F}_n(s) \;&\ge\; - |\hat{F}_n(s_{j-1}) - \tilde{F}_n(s_{j-1})| - \left( \tilde{F}_n(s_j) - \tilde{F}_n(s_{j-1})\right), \\
\hat{F}_n(s) - \tilde{F}_n(s) \;&\le\; |\hat{F}_n(s_j) - \tilde{F}_n(s_j)| + \left( \tilde{F}_n(s_j) - \tilde{F}_n(s_{j-1}) \right).
\end{align*}
Thus, discretizing the supremum over $\R$, we get
\begin{align*}
&\sup_{s\in\R} | \hat{F}_n(s) - \tilde{F}_n(s) | \\
&=  \sup_{s<s_0} | \hat{F}_n(s) - \tilde{F}_n(s) | \lor \sup_{s\ge s_L} | \hat{F}_n(s) - \tilde{F}_n(s) |\\
&\quad\quad\lor \max_{j=1,\dots,L} \sup_{s\in[s_{j-1},s_j)} | \hat{F}_n(s) - \tilde{F}_n(s) | \\
&\le\quad
\left( |\hat{F}_n(s_0) - \tilde{F}_n(s_0) | + \tilde{F}_n(s_0) \right) 
\lor 
\left( |\hat{F}_n(s_L) - \tilde{F}_n(s_L) | + 1 - \tilde{F}_n(s_L) \right)\\
&\quad\quad\lor \max_{j=1,\dots,L} \left( \left[|\hat{F}_n(s_{j-1}) - \tilde{F}_n(s_{j-1}) |\lor |\hat{F}_n(s_j) - \tilde{F}_n(s_j) |\right] +  \tilde{F}_n(s_j) - \tilde{F}_n(s_{j-1}) \right).
\end{align*}
Next using the abbreviation $e_n(x_0) = \hat{\mu}_n(x_0) - g(x_0)$ and both parts of Lemma~\ref{lemma:Ftilde}, we can further bound this by
\begin{align*}
&\max_{j=0,\dots,L} \left( |\hat{F}_n(s_{j}) - \tilde{F}_n(s_{j})| \right) \\
&\quad+ \max\left\{ \E_{x_0}\left[(\eps\|f_{y_0\|x_0}\|_\infty)\land1\right], \tilde{F}_n(s_0), 1-\tilde{F}_n(s_L)\right\} \\
\le&
\max_{j=0,\dots,L} \left( |\hat{F}_n(s_{j}) - \tilde{F}_n(s_{j})| \right) \\
&\quad+ \E_{x_0}\left[(\eps\|f_{y_0\|x_0}\|_\infty)\land1\right] + P \left( |y_0-g(x_0)| \ge c_1\right) + P(|e_n(x_0)| > c_2).
\end{align*}
Now, using Lemma~2.1 of \citet{Aven85}, the expectation (w.r.t. the training data $T_n$) of the maximum can be bounded by
\begin{align*}
\left(\sum_{j=0}^{L} \E\left[|\hat{F}_n(s_{j}) - \tilde{F}_n(s_{j})|^2\right] \right)^{\frac12}.
\end{align*}
Finally, applying Lemmas~\ref{lemma:Bousquet} and \ref{lemma:elementary}, the expression on the previous display is bounded by
$$
\left( (L+1) \left[\frac{1}{4(k-1)} + 5\eta_{n,k}\right]\right)^{\frac12}.
$$
We have established the bound
\begin{align*}
&\E\left[\|\hat{F}_n - \tilde{F}_n\|_\infty\right]
\le
\left( \left(\frac{2(c_1+c_2)}{\eps}+2\right) \left[\frac{1}{4(k-1)} + 5\eta_{n,k}\right]\right)^{\frac12}\\
&\quad+
 \E_{x_0}\left[(\eps\|f_{y_0\|x_0}\|_\infty)\land1\right] + P\left( |y_0-g(x_0)| \ge c_1\right) + \P(|e_n(x_0)| > c_2).
\end{align*}
\hfill \qed


\begin{proof}[Proof of Lemma~\ref{lemma:elementary}]
The integrand on the left of the desired inequality is equal to
\begin{align*}
\mathbbm{1}_{\left\{ y_0-\hat{\mu}_n(x_0) \le s < y_0 - \hat{\mu}_n^{(l)}(x_0)\right\}}
+
\mathbbm{1}_{\left\{ y_0-\hat{\mu}_n^{(l)}(x_0) \le s < y_0 - \hat{\mu}_n(x_0)\right\}}.
\end{align*}
Thus, the conditional expectation of the sum in the previous display, given the training data $T_n$ and $x_0$, can be expressed as
\begin{align*}
&\P( s+\hat{\mu}_n(x_0) < y_0 \le s+\hat{\mu}_n^{(l)}(x_0) \| T_n, x_0) \\
&\quad\quad+
\P( s+\hat{\mu}_n^{(l)}(x_0) < y_0 \le s+\hat{\mu}_n(x_0) \| T_n, x_0) \\
&=
\int\limits_{s + \hat{\mu}_n(x_0)\land \hat{\mu}_n^{(l)}(x_0)}^{s + \hat{\mu}_n(x_0) \lor \hat{\mu}_n^{(l)}(x_0)} f_{y_0\|x_0}(v)\;dv  \land 1
\le
\left(\|f_{y_0\|x_0}\|_\infty |\hat{\mu}_n(x_0) - \hat{\mu}_n^{(l)}(x_0)| \right)\land 1.
\end{align*}
\end{proof}


\begin{proof}[Proof of Lemma~\ref{lemma:Ftilde}]
For $T_n\in\mathcal Z^n$ and $s_1>s_2$, note
\begin{align*}
&\tilde{F}_n( s_1) - \tilde{F}_n( s_2)
=
\P_{z_0} \left( s_2 +  \hat{\mu}_n(x_0)< y_0 \le s_1 + \hat{\mu}_n(x_0)\right)\\
&\quad=
\E_{x_0}\left[  \int\limits_{s_2 +  \hat{\mu}_n(x_0)}^{s_1 +  \hat{\mu}_n(x_0)} f_{y_0\|x_0}(v)\;dv \land 1 \right]
\;\le\;
\E_{x_0}\left[ (\|f_{y_0\|x_0}\|_\infty (s_1-s_2))\land1 \right].
\end{align*}
So the first claim follows upon reversing the roles of $s_1$ and $s_2$.
For the second claim, set $e_n := \hat{\mu}_n(x_0) - g(x_0)$ to obtain
\begin{align*}
\tilde{F}_n( c_1+c_2) 
\;&=\;
\P_{z_0} \left( y_0 \le c_1+c_2 + \hat{\mu}_n(x_0) \right)\\
&\ge\;
\P_{z_0} \left( y_0 - g(x_0) \le c_1 + c_2 + e_n, e_n \ge -c_2 \right)\\
&\ge\;
\P_{z_0} \left( y_0 - g(x_0) \le c_1, e_n \ge -c_2 \right).
\end{align*}
Therefore, taking complements, we arrive at
$$
1- \tilde{F}_n( c_1+c_2) \le P \left( y_0 - g(x_0) > c_1\right) + \P_{z_0}\left( e_n < -c_2 \right).
$$
 The second one is obtained analogously by
\begin{align*}
\tilde{F}_n(-c_1-c_2) 
\;&=\;
P \left( y_0 - g(x_0) \le -c_1-c_2 + e_n \right)\\
&\le\;
\P_{z_0} \left( y_0 - g(x_0) \le -c_1-c_2 + e_n, e_n \le c_2 \right) + \P_{z_0}(e_n > c_2 )\\
&\le\;
P \left( y_0-g(x_0) \le -c_1\right) + \P_{z_0}(e_n > c_2).
\end{align*}
Adding both bounds finishes the proof.
\end{proof}


\subsection{Proof of Theorem~\ref{thm:JSmisspec}}

We begin by showing that under the assumptions of Theorem~\ref{thm:JSmisspec}, both of its conclusions hold with $c_n=0$ (OLS) and irrespective of $\kappa\in[0,1)$. The proof of the following result is deferred to the end of the subsection.

\begin{lemma}\label{lemma:OLSmisspec}
For every $n\in\N$, let $\mathcal P_n= \mathcal P_n(\mathcal L_l, \mathcal L_v, C_0)$ be as in Condition~\ref{c.non-linmod}. For $P_n\in \mathcal P_n$, define $\beta_{P_n}$ to be the minimizer of $\beta \mapsto \E_n[(y_0-\beta'x_0)^2]$ over $\R^{p_n}$. If $p_n/n\to \kappa \in [0,1)$ and
\begin{align}\label{eq:MisspecBound}
\limsup_{n\to\infty} \E_n\left[\left(\frac{\mu_{P_n}(x_{0,n}) - x_{0,n}'\beta_{P_n}}{\sigma_{P_n}} \right)^2\right] \;<\;\infty,
\end{align}
then the ordinary least squares estimator $\hat{\beta}_n = (X'X)^\dagger X'Y$ satisfies
$$
 \left\|\Sigma_{P_n}^{1/2}(\hat{\beta}_n - \beta_{P_n})/\sigma_{P_n}\right\|_2^2  = O_{P_n}(1),
$$
and 
$$
\left\|\Sigma_P^{1/2}(\hat{\beta}_n - \hat{\beta}_n^{[1]})/\sigma_P\right\|_2^2 = o_{P_n}(1).
$$
\end{lemma}

We proceed with the proof of Theorem~\ref{thm:JSmisspec}.
Abbreviate $\mu_n = \mu_{P_n}$, $\beta_n = \beta_{P_n}$, $\Sigma_n = \Sigma_{P_n}$ and $\sigma_n = \sigma_{P_n}$.
We have to show that $\limsup_{n\to\infty} \P_n(\|\Sigma_n^{1/2}(\hat{\beta}_n(c_n) - \beta_n)/\sigma_n\|_2^2>M)\to 0$ as $M\to\infty$, and, provided that $\kappa>0$, that $\P_n(\|\Sigma_n^{1/2}(\hat{\beta}_n(c_n) - \hat{\beta}_n^{[1]}(c_n))/\sigma_n\|_2^2>\eps)\to 0$, as $n\to\infty$, for every $\eps>0$.

Define $\delta_n^2 = \beta_n'\Sigma_n\beta_n/\sigma_n^2$, $t_n^2 = \hat{\beta}_n'X'X\hat{\beta}_n/(n\sigma_n^2)$ and
\begin{align*}
s_n = \begin{cases}
\left( 1 - \frac{p_n}{n}\frac{c_n}{t_n^2}\frac{\hat{\sigma}_n^2}{\sigma_n^2}\right)_+, \quad&\text{if } t_n^2 >0,\\
1, &\text{if } t_n^2=0,
\end{cases}
\end{align*}
such that $0\le s_n\le 1$, and $\hat{\beta}_n(c_n) = s_n \hat{\beta}_n$, because $t_n^2=0$ if, and only if, $\hat{\beta}_n=0$. 
We abbreviate $D:= \limsup_{n\to \infty} \E_{x_{0,n}}[(\mu_n(x_{0,n})-\beta_n'x_{0,n})^2]/\sigma_n^2 < \infty$. The following properties are useful and will be verified after the main argument is finished.

\begin{lemma}\label{lemma:JSmisspec}
Under the assumptions of Theorem~\ref{thm:JSmisspec} we have: $\hat{\sigma}_n^2/\sigma_n^2$ and $\sigma_n^2/\hat{\sigma}_n^2$ are $P_n$-bounded in probability, $\P_n(\hat{\sigma}_n^2=0)=0$ (provided that $n\ge p_n+2$) and $\P_n(t_n^2=0)\to 0$. Furthermore, we have $\P_n(t_n^2 \ge \kappa/2) \to 1$ if $\delta_n\to\delta\in[0,\infty)$.
All the statements of the lemma continue to hold also for the leave-one-out versions $t_{n,[1]}^2:= \hat{\beta}_n^{[1]'}X_{[1]}'X_{[1]}\hat{\beta}_n^{[1]}/(n\sigma_n^2)$ and $\hat{\sigma}_{n,[1]}^2 := \|Y_{[1]}-X_{[1]}\hat{\beta}_n^{[1]}\|_2^2/(n-1-p_n)$ of $t_n^2$ and $\hat{\sigma}_n^2$.
\end{lemma}

The quantity of interest in the first claim of the theorem can be bounded as
\begin{align}
\left\| \Sigma_n^{1/2}\left(\hat{\beta}_n(c_n) - \beta_n\right)/\sigma_n\right\|_2
&=
\left\| \Sigma_n^{1/2}s_n\left(\hat{\beta}_n - \beta_n\right)/\sigma_n + \Sigma_n^{1/2}(s_n-1)\beta_n/\sigma_n\right\|_2\notag\\
&\le
\left\| \Sigma_n^{1/2}\left(\hat{\beta}_n - \beta_n\right)/\sigma_n\right\|_2 
 +  (1-s_n) \delta_n.\label{eq:s_ndelta_n}
\end{align}
Therefore, by Lemma~\ref{lemma:OLSmisspec}, it remains to show that $\limsup_{n\to\infty}Q_n(M)\to0$ as $M\to\infty$, where $Q_n(M) = \P_n((1-s_n)\delta_n>M)$. For fixed $M\in (1,\infty)$ and fixed $n\in\N$, we distinguish the cases $\delta_n< M^{1/2}$ and $\delta_n\ge M^{1/2}$. In the former case, $Q_n(M) = 0$. In the latter case, we proceed as follows. First, notice that
\begin{align}
Q_n(M) &= \P_n\left((1-s_n)\delta_n>M, t_n^2>0\right)\le
\P_n\left( \frac{p_n}{n} \frac{c_n}{t_n^2} \frac{\hat{\sigma}_n^2}{ \sigma_n^2} \delta_n > M, t_n^2>0\right)
\notag\\
&=
\P_n\left( \frac{p_n}{n} \frac{c_n}{t_n^2/\delta_n^2} \frac{\hat{\sigma}_n^2}{ \sigma_n^2} > M, t_n^2>0\right).\label{eq:boundQn}
\end{align} 
Furthermore, we trivially have $Y = X\beta_n + \sigma_n\tilde{v}$, where $\tilde{v} := (Y-X\beta_n)/\sigma_n$ has components $\tilde{v}_i = (\mu_n(x_i)-\beta_n'x_i)/\sigma_n + (y_i-\mu_n(x_i))/\sigma_n$, and, using the reverse triangle inequality and the notation $\|a\|_{\Pi_X} = \sqrt{a'\Pi_X a}$, we have
\begin{align*}
t_n &= \sqrt{\frac{1}{n}\frac{Y'\Pi_XY}{\sigma_n^2} }
=\|X\beta_n + \sigma_n \tilde{v}\|_{\Pi_X}(n\sigma_n^2)^{-1/2}\\
&\ge
\left|\|X\beta_n\|_{\Pi_X} - \|\sigma_n \tilde{v}\|_{\Pi_X}\right|(n\sigma_n^2)^{-1/2}\\
&=
\left|
\sqrt{\frac{\beta_n'\Sigma_n^{1/2}(\tilde{X}'\tilde{X}/n)\Sigma_n^{1/2}\beta_n}{\sigma_n^2}} - \sqrt{\frac{\tilde{v}'\Pi_X\tilde{v}}{n}}
\right|,
\end{align*}
where $\tilde{X}:= (\tilde{x}_1,\dots, \tilde{x}_n)' := X\Sigma_n^{-1/2}$ and $\Pi_X := X(X'X)^\dagger X'$.
Therefore, on the event 
$$
A_n(M) = \{\|\tilde{v}/\sqrt{n}\|_2^2 \le M^{1/2}, \lmin(\tilde{X}'\tilde{X}/n)> c_0^2(1-\sqrt{\kappa})^2/2>M^{-1/2}\},
$$ 
we have
$\tilde{v}'\Pi_X\tilde{v}(n\delta_n^2)^{-1} \le M^{-1/2}$ (because $\delta_n\ge\sqrt{M}$) and
$$
\beta_n'\Sigma_n^{1/2}(\tilde{X}'\tilde{X}/n)\Sigma_n^{1/2}\beta_n(\sigma_n^2\delta_n^2)^{-1} 
> 
c_0^2(1-\sqrt{\kappa})^2/2
>
M^{-1/2},
$$
so that on this event $t_n/\delta_n \ge c_0(1-\sqrt{\kappa})/\sqrt{2} - M^{-1/4}\ge0$.
Thus, turning back to \eqref{eq:boundQn} and using Markov's inequality, we obtain
\begin{align*}
&\P_n\left( \frac{p_n}{n} \frac{c_n}{t_n^2/\delta_n^2} \frac{\hat{\sigma}_n^2}{ \sigma_n^2} > M, t_n^2>0\right)
\le
\P_n\left( \frac{\hat{\sigma}_n^2}{ \sigma_n^2}> M t_n^2/\delta_n^2, t_n^2>0\right)\\
&\quad\le
\P_n\left( \frac{\hat{\sigma}_n^2}{ \sigma_n^2} > M t_n^2/\delta_n^2, A_n(M)\right)
+
\P_n\left( A_n(M)^c\right)\\
&\quad\le
\P_n\left(\frac{\hat{\sigma}_n^2}{ \sigma_n^2} > M\left(c_0(1-\sqrt{\kappa})/\sqrt{2} - M^{-1/4}\right)^2  \right)
+ \frac{2D+1}{M^{1/2}} \\
&\quad\quad+ \P_n\left( \lmin(\tilde{X}'\tilde{X}/n) \le c_0^2(1-\sqrt{\kappa})^2/2\right)
+ \P_n(c_0^2(1-\sqrt{\kappa})^2/2 \le M^{-1/2}),
\end{align*}
for sufficiently large $n$.
In view of Lemma~\ref{lemma:traceConv}\ref{lemma:traceConvA} in Appendix~\ref{sec:proofsaux} and $P_n$-boundedness of $\hat{\sigma}_n^2 /\sigma_n^2$ (Lemma~\ref{lemma:JSmisspec}), the limit superior of the upper bound is equal to a function $Q(M)\ge0$ that vanishes as $M\to\infty$. 
Therefore, we have shown that $\limsup_{n\to\infty} Q_n(M) \le Q(M) \to 0$ as $M\to\infty$.

To establish the claim about the stability of $\hat{\beta}_n(c_n)$ we proceed in a similar way. First, note that
\begin{align*}
&\|\Sigma_n^{1/2}(\hat{\beta}_n(c_n) - \hat{\beta}_n^{[1]}(c_n))\|_2/\sigma_n
=
\|(s_n- s_n^{[1]})\Sigma_n^{1/2}\hat{\beta}_n  + s_n^{[1]}\Sigma_n^{1/2}(\hat{\beta}_n - \hat{\beta}_n^{[1]})\|_2/\sigma_n\\
&\,\le
|s_n - s_n^{[1]}| \|\Sigma_n^{1/2} \hat{\beta}_n/\sigma_n\|_2 
 + |s_n^{[1]}|\|\Sigma_n^{1/2}(\hat{\beta}_n - \hat{\beta}_n^{[1]})/\sigma_n\|_2\\
&\,\le
|s_n - s_n^{[1]}| \|\Sigma_n^{1/2} (\hat{\beta}_n - \beta_n)/\sigma_n\|_2 
+ 
|s_n - s_n^{[1]}| \delta_n 
+ 
|s_n^{[1]}|\|\Sigma_n^{1/2}(\hat{\beta}_n - \hat{\beta}_n^{[1]})/\sigma_n\|_2,
\end{align*}
where $s_n^{[1]}$ is defined like $s_n$, but using $t_{n,[1]}$ and $\hat{\sigma}_{n,[1]}^2$. In view of Lemma~\ref{lemma:OLSmisspec}, it is easy to see that it remains to show that $|s_n - s_n^{[1]}|(1+\delta_n) = o_{P_n}(1)$. We argue along subsequences. Let $n'$ be an arbitrary subsequence of $n$. Then by compactness of the extended real line, there exists a further subsequence $n''$ of $n'$, such that $\delta_{n''}\to\delta\in[0,\infty]$. If we can show that for every $\eps>0$
$$
\P_{n''}(|s_{n''} - s_{n''}^{[1]}|(1+\delta_{n''})>\eps) \xrightarrow[n''\to\infty]{} 0,
$$
then the claim follows. For simplicity, we write $n$ instead of $n''$ and we distinguish the cases $\delta=\infty$ and $\delta\in[0,\infty)$. 

If $\delta=\infty$, then it suffices to show that $(s_n-s_n^{[1]})\delta_n$ converges to zero in $P_n$-probability. By Lemma~\ref{lemma:JSmisspec} we have $\P_n(t_n^2=0)\to0$ and the same holds for $t_{n,[1]}^2$, such that it suffices to show that
$$
\P_n(|s_n-s_n^{[1]}|\delta_n > \eps, t_n>0, t_{n,[1]}>0) \to 0.
$$
If $t_n>0$, set $r_n = \frac{p_n}{n}\frac{c_n}{t_n^2}\frac{\hat{\sigma}_n^2}{\sigma_n^2}$, such that $s_n = (1-r_n)_+$ on this event, and define $r_{n}^{[1]} = \frac{p_n}{n}\frac{c_n}{t_{n,[1]}^2}\frac{\hat{\sigma}_{n,[1]}^2}{\sigma_n^2}$, provided that $t_{n,[1]}>0$. Thus, if both $t_n$ and $t_{n,[1]}$ are positive, we have 
$$
|s_n-s_n^{[1]}|\delta_n\le |r_n - r_{n}^{[1]}|\delta_n
\le
\left|
\frac{\delta_n^2}{t_n^2}\frac{\hat{\sigma}_n^2}{ \sigma_n^2} - \frac{\delta_n^2}{t_{n,[1]}^2}\frac{\hat{\sigma}_{n,[1]}^2}{ \sigma_n^2}
\right|\frac{1}{\delta_n}.
$$
But in the first part of the proof we have already established that $t_n^2/\delta_n^2$ is lower bounded by $c_0^2(1-\sqrt{\kappa})^2/4$ with asymptotic probability one, provided that $\delta_n^2\to\infty$ (recall the case $\delta_n\ge M^{1/2}$ and the set $A_n(M)$, and let $M=\delta_n^2\to\infty$), and an analogous argument applies to $t_{n,[1]}^2/\delta_n^2$. Thus, it follows from the $P_n$-boundedness of $\hat{\sigma}_n^2 /\sigma_n^2$ and $\hat{\sigma}_{n,[1]}^2 /\sigma_n^2$ that the upper bound in the previous display converges to zero in $P_n$-probability.

If $\delta\in[0,\infty)$, it suffices to show that $|s_n-s_n^{[1]}|$ converges to zero in $P_n$-probability. As before, we restrict to the event $\{t_n>0, t_{n,[1]}>0\}$. Note that due to the positive part mapping in the definition of $s_n$, the absolute difference $|s_n-s_n^{[1]}|$ vanishes if both $r_n$ and $r_n^{[1]}$ are greater than or equal to $1$, and is otherwise bounded by $|r_n-r_n^{[1]}| \le \max(|r_n/r_n^{[1]}-1|, |r_n^{[1]}/r_n - 1|)$, provided that $r_n$ and $r_n^{[1]}$ are positive. Thus, it remains to verify that $r_n^{[1]}/r_n$ converges to $1$ in $\P_n$-probability and that both $\P_n(r_n=0)$ and $\P_n(r_n^{[1]}=0)$ converge to zero. The latter statement follows from Lemma~\ref{lemma:JSmisspec}, in fact it shows that $\P_n(r_n=0)=0=\P_n(r_n^{[1]}=0)$ provided that $n\ge p_n+2$. Finally, to show that $r_n^{[1]}/r_n \to 1$ in $P_n$-probability, define $S_{[1]} := \tilde{X}_{[1]}'\tilde{X}_{[1]} = \sum_{i=2}^n \tilde{x}_i\tilde{x}_i'$ and note that by the Sherman-Morrison formula (see also the proof of Lemma~\ref{lemma:OLSmisspec} below) we have
\begin{align*}
\hat{\beta}_n'X'X\hat{\beta}_n 
&=
\hat{\beta}_n^{[1]'}X_{[1]}'X_{[1]}\hat{\beta}_n^{[1]}
+
(x_1'\hat{\beta}_n^{[1]})^2 
+ 2x_1'\hat{\beta}_n^{[1]}(y_1-x_1'\hat{\beta}_n^{[1]})\\
&\quad\quad+
(y_1-x_1'\hat{\beta}_n^{[1]})^2 \frac{\tilde{x}_1'S_{[1]}^{-1}\tilde{x}_1}{1+\tilde{x}_1'S_{[1]}^{-1}\tilde{x}_1} \\
&\le \hat{\beta}_n^{[1]'}X_{[1]}'X_{[1]}\hat{\beta}_n^{[1]}
+ y_1^2,
\end{align*}
at least on the event $B_n:=\{\lmin(S_{[1]})>0\}$, which has asymptotic $\P_n$-probability one by Lemma~\ref{lemma:traceConv}.
Thus, on $B_n$, $t_n^2/t_{n,[1]}^2 = 1 + g_n$, where 
\begin{align}\label{eq:g_nBound}
|g_n| \le \frac{y_1^2}{\hat{\beta}_n^{[1]'}X_{[1]}'X_{[1]}\hat{\beta}_n^{[1]}}.
\end{align}
By Lemma~\ref{lemma:JSmisspec}, and since $\kappa>0$, $\hat{\beta}_n^{[1]'}X_{[1]}'X_{[1]}\hat{\beta}_n^{[1]}/(n\sigma_n^2) = t_{n,[1]}^2$ is bounded away from zero with asymptotic probability one. Thus, for the desired convergence of $t_n^2/t_{n,[1]}^2$ to $1$, it remains to show that the numerator in \eqref{eq:g_nBound} divided by $n\sigma_n^2$ converges to zero in $P_n$-probability. But this now follows from Condition~\ref{c.non-linmod}, assumption~\eqref{eq:JSMisspecBound} and the fact that $\delta<\infty$. The proof is finished if we can also show that $\hat{\sigma}_n^2/\hat{\sigma}_{n,[1]}^2$ converges to $1$ in $P_n$-probability. To this end, we apply the Sherman-Morrison formula once more to get
\begin{align*}
I_n - \Pi_X = I_n - \Pi_{\tilde{X}}  
= \begin{pmatrix}
\frac{1}{1+\tilde{x}_1'S_{[1]}^{-1}\tilde{x}_1}, &-\frac{\tilde{x}_1'S_{[1]}^{-1}\tilde{X} _{[1]}'}{1+\tilde{x}_1'S_{[1]}^{-1}\tilde{x}_1}\\
- \frac{\tilde{X} _{[1]}S_{[1]}^{-1}\tilde{x}_1}{1+\tilde{x}_1'S_{[1]}^{-1}\tilde{x}_1}, &I_{n-1} - \Pi_{\tilde{X} _{[1]}} + \frac{\tilde{X} _{[1]}S_{[1]}^{-1}\tilde{x}_1\tilde{x}_1'S_{[1]}^{-1}\tilde{X} _{[1]}'}{1+\tilde{x}_1'S_{[1]}^{-1}\tilde{x}_1}
\end{pmatrix},
\end{align*}
on the event $B_n$. Thus, on this event,
\begin{align*}
\hat{\sigma}_n^2(n-p_n) &= Y'(I_n-\Pi_X)Y  = Y_{[1]}'(I_{n-1}-\Pi_{X_{[1]}})Y_{[1]}\\
&\quad+ \frac{(y_1-x_1'\hat{\beta}_n^{[1]})^2}{1+\tilde{x}_1'S_{[1]}^{-1}\tilde{x}_1},
\end{align*}
such that $\frac{\hat{\sigma}_n^2}{\hat{\sigma}_{n,[1]}^2} \frac{n-p_n}{n-1-p_n} =: 1 + h_n$, where
$$
|h_n| \le \frac{(y_1-x_1'\hat{\beta}_n^{[1]})^2}{(n-1-p_n)\sigma_n^2}\frac{\sigma_n^2}{\hat{\sigma}_{n,[1]}^2}.
$$
But it is easy to see that the upper bound converges to zero in $P_n$-probability by Condition~\ref{c.non-linmod}, assumption~\eqref{eq:JSMisspecBound}, Lemmas~\ref{lemma:OLSmisspec} and \ref{lemma:JSmisspec}, and because $n-p_n\to\infty$ and $\delta<\infty$. \hfill\qed

\begin{proof}[Proof of Lemma~\ref{lemma:JSmisspec}]
We use the notation $e_i := \frac{\mu_n(x_i)-x_i'\beta_n}{\sigma_n}$, $v_i := \frac{y_i- \mu_n(x_i)}{\sigma_n}$, $\check{v} := e + v$, where $e = (e_1,\dots, e_n)'$ and $v = (v_1, \dots, v_n)'$, such that $Y=X\beta_n + \sigma_n\check{v} = X\beta_n + \sigma_ne + \sigma_nv$.
For the first claim simply observe that
\begin{align*}
\frac{\hat{\sigma}_n^2}{\sigma_n^2} = \frac{Y'(I_n-\Pi_X)Y}{(n-p_n)\sigma_n^2} = \frac{n}{n-p_n}\frac{\check{v}'(I_n-\Pi_X)\check{v}}{n}
\le \frac{n}{n-p_n}\left\|\frac{\check{v}}{\sqrt{n}}\right\|_2^2, 
 \end{align*}
and that $\E_n\|\check{v}\|_2^2 = n\E_n[\check{v}_1^2] \le n(2D+1)$, for sufficiently large $n$. For boundedness of the reciprocal we first note that $\P_n(\hat{\sigma}_n^2 = 0) = \E_n[\P_n(Y'(I_n-\Pi_X)Y=0\|X)]= \P_n(I_n-\Pi_X =0) = 0$, because the conditional distribution of $Y$ given $X$ under $\P_n$ is absolutely continuous with respect to Lebesgue measure and $n\ge p_n+2$. Similarly, $\P_n(t_n=0) = \P_n(\hat{\beta}_n=0) = \E_n[\P_n((X'X)^\dagger X'Y=0\|X)] = \P_n(X(X'X)^\dagger=0) = \P_n(X=0) = (\mathcal L_w(\{0\}))^{np_n} \to 0$. Next we show that $\hat{\sigma}_n^2/\sigma_n^2$ is bounded from below by $(1-\kappa)/2$ with asymptotic probability one. To this end, note that
$$
\frac{\hat{\sigma}_n^2}{\sigma_n^2}
=
\frac{Y'(I_n-\Pi_X)Y}{(n-p_n)\sigma_n^2}
\ge
2\frac{e'(I_n-\Pi_X)v}{n} + \left\|\frac{(I_n-\Pi_X)v}{\sqrt{n}}\right\|_2^2,
$$
where the conditional expectation of the mixed term given $X$ is equal to zero and its conditional variance converges to zero in $P_n$-probability because of $\E_n[e_i^2]\le 2D$, for sufficiently large $n$. The conditional expectation of the last term in the previous display is $\trace(I_n-\Pi_X)/n = \trace(I_n-\Pi_{\tilde{X}})/n = 1-p_n/n$, with asymptotic probability one in view of Lemma~\ref{lemma:traceConv}\ref{lemma:traceConvA}. Using independence of the $v_i$ and a little algebra, its conditional variance can be computed as
\begin{align}
\Var_n\left[\frac{v'(I_n-\Pi_X)v}{n}\Big\|X\right] 
&= \frac{2\trace((I_n-\Pi_X)^2)}{n^2} + \frac{(\E_n[v_1^4]-3)}{n^2} \sum_{i=1}^n(I_n-\Pi_X)_{ii}^2
\notag\\
&\le 2\frac{n}{n^2} + \frac{(C_0+3)n}{n^2} \to 0.\label{eq:CondVarHatsigma}
\end{align}
This establishes the boundedness of $\sigma_n^2/\hat{\sigma}_n^2$.
For the remaining statement about $t_n^2$, suppose that $\delta_n^2\to\delta\in[0,\infty)$ and note that
\begin{align*}
t_n^2 &= \frac{Y'\Pi_XY}{n\sigma_n^2}
=\left\|\frac{\tilde{X}\Sigma_n^{1/2}\beta_n}{\sqrt{n\sigma_n^2}} + \frac{\Pi_Xe}{\sqrt{n}} + \frac{\Pi_Xv}{\sqrt{n}}\right\|_2^2.
\end{align*}
Abbreviate $W_n := \frac{\tilde{X}\Sigma_n^{1/2}\beta_n}{\sqrt{n\sigma_n^2}} + \frac{\Pi_Xe}{\sqrt{n}}$ and observe $t_n^2 \ge 2W_n'\Pi_Xv/\sqrt{n} + \|\Pi_Xv/\sqrt{n}\|_2^2$. The conditional expectation of the mixed term $W_n'\Pi_Xv/\sqrt{n}$ given $X$ is equal to zero, and its conditional variance is bounded by $\|W_n/\sqrt{n}\|_2^2$. But $\|W_n\|_2^2$ is bounded in $P_n$-probability, in view of the facts that $\delta<\infty$, $\E_n[\tilde{X}'\tilde{X}/n] = I_n$ and $\E_n[e_i^2]\le 2D$, for sufficiently large $n$. Thus, the mixed term is $o_{P_n}(1)$. For $\|\Pi_Xv/\sqrt{n}\|_2^2$ one easily verifies that its conditional expectation given $X$ is $\trace(\Pi_X)/n = \trace(\Pi_{\tilde{X}})/n$, which converges to $\kappa\in[0,1)$ in $P_n$-probability, because $\P_n(\lmin(\tilde{X}'\tilde{X})=0)\to0$ by Lemma~\ref{lemma:traceConv}. Furthermore, as above, its conditional variance can easily be computed as
\begin{align*}
\Var_n\left[\frac{v'\Pi_Xv}{n}\Big\|X\right] 
&= \frac{2\trace(\Pi_X^2)}{n^2} + \frac{(\E_n[v_1^4]-3)}{n^2} \sum_{i=1}^n(\Pi_X)_{ii}^2\\
&\le \frac{2p_n}{n^2} + \frac{(C_0+3)p_n}{n^2} \to 0.
\end{align*}
Thus, $\|\Pi_Xv/\sqrt{n}\|_2^2$ converges to $\kappa$, in $P_n$-probability, which establishes the asymptotic lower bound on $t_n^2$.
The results about the leave-one-out quantities can be established analogously. 
\end{proof}

\begin{proof}[Proof of Lemma~\ref{lemma:OLSmisspec}]
Fix $n\in\N$ and $P_n\in\mathcal P_n$. For simplicity, we write $\mu = \mu_{P_n}$, $\Sigma = \Sigma_{P_n}$, $\beta = \beta_{P_n}$, $\sigma^2 = \sigma_{P_n}^2$ and $\P=\P_n$, and abbreviate $\tilde{X} := X\Sigma^{-1/2}$. For $\xi>0$, consider the event $A_n := A_n(\xi) := \{\lmin(\tilde{X}'\tilde{X}/n)>\xi\}$. On this event, we observe that $\Sigma^{1/2}(\hat{\beta}_n - \beta)/\sigma = (\tilde{X}'\tilde{X})^{-1}\tilde{X}'\check{v}$, where $\check{v} = (\check{v}_1,\dots, \check{v}_n)'$, $\check{v}_i = (\mu(x_i) - \beta'x_i)/\sigma + v_i$ and $v_i = (y_i - \mu(x_i))/\sigma$, for $i=1,\dots, n$. Thus, on $A_n$, $\|\Sigma^{1/2}(\hat{\beta}_n-\beta)/\sigma\|_2^2 = \check{v}'\tilde{X}(\tilde{X}'\tilde{X})^{-2}\tilde{X}'\check{v} = \check{v}'\tilde{X}(\tilde{X}'\tilde{X})^{-1/2}(\tilde{X}'\tilde{X})^{-1}(\tilde{X}'\tilde{X})^{-1/2}\tilde{X}'\check{v} \le \|\check{v}/\sqrt{n}\|_2^2 \|(\tilde{X}'\tilde{X}/n)^{-1}\|_2$. Hence, using Condition~\ref{c.non-linmod}, we obtain 
\begin{align*}
&\P(\|\Sigma^{1/2}(\hat{\beta}_n - \beta)/\sigma\|_2^2 > M)\\
&\quad\le
\P(\|\check{v}/\sqrt{n}\|_2^2 \|(\tilde{X}'\tilde{X}/n)^{-1}\|_2 > M,A_n(\xi)) + \P(A_n(\xi)^c)\\
&\quad\le
\P(\|\check{v}/\sqrt{n}\|_2^2 /\xi > M) + \P(A_n(\xi)^c)\\
&\quad\le
\frac{\E[\check{v}_1^2]}{M\xi} + \P(\lmin(\tilde{X}'\tilde{X}/n) \le \xi) .
\end{align*}
Since $\E[\check{v}_1^2] = \E[(\mu(x_0)-\beta'x_0)^2/\sigma^2] + 1$, in view of \ref{c.non-linmod}, and because $\P(\lmin(\tilde{X}'\tilde{X}/n) \le \xi) $ does not depend on the parameters $\beta$, $\Sigma$ and $\sigma^2$, Lemma~\ref{lemma:traceConv}\ref{lemma:traceConvA'} implies the first claim if we set $\xi = c_0^2(1-\sqrt{\kappa})^2/2>0$.

For the stability property, we abbreviate $S_{[1]} = \tilde{X}_{[1]}'\tilde{X}_{[1]}$, $\tilde{\beta}_n := \Sigma^{1/2}(\hat{\beta}_{n} - \beta)/\sigma$ and $\tilde{\beta}_n^{[1]} = \Sigma^{1/2}(\hat{\beta}_{n}^{[1]} - \beta)/\sigma$, and consider the event $B_n = \{\lmin(S_{[1]})>0\}$. On this event, also $\lmin(\tilde{X}'\tilde{X}) = \lmin(S_{[1]} + \tilde{x}_1\tilde{x}_1')> 0$, where $\tilde{X} = [\tilde{x}_1,\dots, \tilde{x}_n]'$ and $\tilde{X}_{[1]} = [\tilde{x}_2,\dots, \tilde{x}_n]'$, and the Sherman-Morrison formula yields
\begin{align*}
&\tilde{\beta}_n\; =\; (\tilde{X}'\tilde{X})^{-1}\tilde{X}'\check{v}\; =\; (S_{[1]} + \tilde{x}_1\tilde{x}_1')^{-1}(\tilde{X}_{[1]}'\check{v}_{[1]} + \tilde{x}_1 \check{v}_1)\\
&\quad= 
\left(S_{[1]}^{-1} - \frac{S_{[1]}^{-1}\tilde{x}_1\tilde{x}_1'S_{[1]}^{-1}}{1 + \tilde{x}_1'S_{[1]}^{-1}\tilde{x}_1} \right) (\tilde{X}_{[1]}'\check{v}_{[1]} + \tilde{x}_1 \check{v}_1) \\
&\quad=
\tilde{\beta}_n^{[1]} - \frac{S_{[1]}^{-1}\tilde{x}_1\tilde{x}_1'\tilde{\beta}_n^{[1]}}{1 + \tilde{x}_1'S_{[1]}^{-1}\tilde{x}_1} + S_{[1]}^{-1}\tilde{x}_1\check{v}_1 - S_{[1]}^{-1}\tilde{x}_1\check{v}_1\frac{\tilde{x}_1'S_{[1]}^{-1}\tilde{x}_1}{1 + \tilde{x}_1'S_{[1]}^{-1}\tilde{x}_1}\\
&\quad=
\tilde{\beta}_n^{[1]} + \frac{S_{[1]}^{-1}\tilde{x}_1(\check{v}_1-\tilde{x}_1'\tilde{\beta}_n^{[1]})}{1 + \tilde{x}_1'S_{[1]}^{-1}\tilde{x}_1},
\end{align*}
and thus, $\|\Sigma^{1/2}(\hat{\beta}_n - \hat{\beta}_{n}^{[1]})/\sigma\|_2^2 = (1+\tilde{x}_1'S_{[1]}^{-1}\tilde{x}_1)^{-2} \tilde{x}_1'S_{[1]}^{-2}\tilde{x}_1 (\check{v}_1 - \tilde{x}_1'\tilde{\beta}_n^{[1]})^2 \le 2 (\check{v}_1^2 + (\tilde{x}_1'\tilde{\beta}_n^{[1]})^2) \tilde{x}_1'S_{[1]}^{-2}\tilde{x}_1$. Clearly, the squared error term $\check{v}_1^2$ is $P_n$-bounded in probability because $\E[\check{v}_1^2] = \E[(\mu(x_0)-\beta'x_0)^2/\sigma^2]+1$, as above; $\E[(\tilde{x}_1'\tilde{\beta}_{n}^{[1]})^2\|\tilde{\beta}_{n}^{[1]}] = \|\tilde{\beta}_{n}^{[1]}\|_2^2$ is also $P_n$-bounded in probability, by the same argument as in the first paragraph, which implies that $(\tilde{x}_1'\tilde{\beta}_{n}^{[1]})^2$ is $P_n$-bounded in probability; and $\E[\tilde{x}_1'S_{[1]}^{\dagger2}\tilde{x}_1\|S_{[1]}] = \trace S_{[1]}^{\dagger2} \to 0$, in $P_n$-probability, by Lemma~\ref{lemma:traceConv}. Therefore, we have $\P(\|\Sigma^{1/2}(\hat{\beta}_n - \hat{\beta}_{n}^{[1]})/\sigma\|_2^2>\eps, B_n) \le \P(2O_{P_n}(1)o_{P_n}(1)>\eps, B_n) \to 0$. Moreover, $\P(B_n^c) = \P(\lmin(S_{[1]})=0) \to 0$, in view of Lemma~\ref{lemma:traceConv}.
\end{proof}


\subsection{Proof of Theorem~\ref{thm:PIlength}}

We begin by stating a few more results on the OLS estimator that hold in the linear model \ref{c.linmod}. The proof is deferred to the end of the subsection.

\begin{lemma}\label{lemma:OLScorrSpec}
Under the assumptions of Theorem~\ref{thm:PIlength}, the OLS estimator $\hat{\beta}_n = (X'X)^\dagger X'Y$, satisfies
\begin{align*}
&\left\|\Sigma_P^{1/2}(\hat{\beta}_n-\beta_P)/\sigma_P\right\|_2 \xrightarrow[n\to\infty]{} \tau \quad\text{in $P_n$-probability and} \\
&\left\|\Sigma_P^{1/2}(\hat{\beta}_n-\beta_P)/\sigma_P\right\|_4 = o_{P_n}(1).
\end{align*}
Here, $\tau = \tau(\mathcal L_l, \kappa)\in [0,\infty)$ depends only on $\mathcal L_l$ and $\kappa\in[0,1)$ and has the following properties: For any $\mathcal L_l$ as in \ref{c.linmod}, $\tau(\mathcal L_l, \kappa)=0$ if, and only if, $\kappa=0$. If $\mathcal L_l(\{-1,1\})=1$, then $\tau(\mathcal L_l, \kappa) = \sqrt{\kappa/(1-\kappa)}$.
\end{lemma}

The next result will be instrumental to establish convergence of the conditional law $\P(y_0-x_0'\hat{\beta}_n\le t\| T_n)$ to the distribution of $lN\tau + v$, for $l, N, \tau, v$ as in the statement of the theorem. Notice that this also establishes the statement about the asymptotic distribution of the prediction error $y_0-x_0'\hat{\beta}_n$. Its proof is also deferred until after the main argument is finished.

\begin{lemma} \label{lemma:UnifWeak}
Fix arbitrary positive constants $\tau\in[0,\infty)$, $\delta\in(0,2]$ and $c\in(0,\infty)$ and let $(p_n)_{n\in\N}$ be a sequence of positive integers. On some probability space $(\Omega, \mathcal A, \P)$, let $v_0$ and $l_0$ be real random variables and let $W_0 = (w_{0j})_{j=1}^\infty$ be a sequence of i.i.d. real random variables such that $W_0$, $v_0$ and $l_0$ are jointly independent, $|l_0|\ge c$, $\E[l_0^2]=1$, $\E[w_{01}]=0$, $\E[w_{01}^2]=1$ and $\E[|w_{01}|^{2+\delta}]<\infty$. For $n\in\N$ and $b\in\R^{p_n}$, define $w_n = (w_{01},\dots, w_{0p_n})'$, $G(t, b) = \P(l_0 w_n'b + v_0\le t)$ and $F(t) = \P(l_0 N \tau + v_0\le t)$, where $N \stackrel{\mathcal L}{=} \mathcal N(0,1)$ is independent of $(l_0,v_0)$. Consider positive sequences $g_1, g_2:\N \to (0,1)$, such that $g_j(n)\to 0$, as $n\to\infty$, $j=1,2$. 
Suppose that one of the following cases applies:
\begin{enumerate}[(i)]
\item \label{l:UnifWeakT0} $\tau=0$ and $t\mapsto \P(v_0\le t)$ is continuous. In this case, set 
$$B_n = \{b\in\R^{p_n}: \|b\|_2 \le g_1(n)\}.$$
\item \label{l:UnifWeakT>0} $\tau>0$ and $p_n\to\infty$ as $n\to\infty$. In this case, set 
$$
B_n = \{b\in\R^{p_n} : | \|b\|_2 - \tau| \le g_1(n), b \ne 0, \|b\|_{2+\delta}/\|b\|_2 \le g_2(n) \}.$$
\item \label{l:UnifWeakGauss} $\tau>0$ and $w_{01} \stackrel{\mathcal L}{=} \mathcal N(0,1)$. In this case, set 
$$
B_n = \{b\in\R^{p_n} : | \|b\|_2 - \tau| \le g_1(n)\}.$$
\end{enumerate}
Then, using the convention that $\sup \varnothing = 0$,
\begin{align} \label{eq:Lemma:UnifWeak}
\sup_{b\in B_n} \sup_{t\in\R} \left| G(t,b) - F(t) \right| \;\xrightarrow[n\to\infty]{}\; 0.
\end{align}
\end{lemma}

We now turn to the proof of Theorem~\ref{thm:PIlength}.
In order to achieve uniformity in $P_n\in\mathcal P_n^{{lin}}$, we consider sequences of parameters $\beta_n\in\R^{p_n}$, $\sigma_n^2 \in(0,\infty)$ and $\Sigma_n\in \mathcal S_{p_n}$ (where $\mathcal S_{p_n}$ is the set of all symmetric, positive definite $p_n\times p_n$ matrices). All the operators $\E$, $\Var$ and $\Cov$ are to be understood with respect to $\P_n$.

We have to show that, for arbitrary but fixed $\alpha\in[0,1]$, $\hat{q}_\alpha/\sigma_n$ converges in $P_n$-probability to $q_\alpha$, the $\alpha$ quantile of the distribution of $l N \tau + v$, the cdf of which we denote by $F$. Recall that for $\alpha\in[0,1]$, $\hat{q}_\alpha = \hat{F}_n^\dagger(\alpha) := \inf\{t\in\R:\hat{F}_n(t)\ge \alpha\}$. We treat the case $\alpha\in\{0,1\}$ separately at the end of the proof, because $q_1=-q_0=\infty$.
To deal with the empirical quantiles we use a standard argument. For $\alpha\in(0,1)$ and $\eps>0$, consider
$$
\P_n(|\hat{q}_\alpha/\sigma_n - q_{\alpha}| > \eps)
=
\P_n(\hat{q}_\alpha/\sigma_n > q_{\alpha} + \eps)
+
\P_n(\hat{q}_\alpha/\sigma_n < q_{\alpha} - \eps).
$$
To bound the first probability on the right, abbreviate $J_i := \mathbbm 1_{\{\hat{u}_i/\sigma_n > q_{\alpha} + \eps\}}$ and note that by definition of the OLS predictor, the leave-one-out residuals $\hat{u}_i = y_i - x_i'(X_{[i]}'X_{[i]})^\dagger X_{[i]}'Y_{[i]}$, $i=1,\dots, n$, and thus also the $J_i$, $i=1,\dots, n$, are exchangeable under $\P_n$. A basic property of the quantile function $\hat{F}_n^\dagger$ \citep[cf.][Lemma 21.1]{vanderVaart07} yields
\begin{align*}
\P_n(\hat{q}_\alpha/\sigma_n &> q_{\alpha} + \eps)
= \P_n\left(\alpha > \hat{F}_n(\sigma_n(q_{\alpha} + \eps))\right)\\
&= \P_n\left(1-\hat{F}_n(\sigma_n(q_{\alpha} + \eps)) > 1-\alpha \right)\\
&=
\P_n\left(\frac{1}{n}\sum_{i=1}^n (J_i - \E[J_1]) > 1-\alpha - \E[J_1] \right)\\
&=
\P_n\left(\frac{1}{n}\sum_{i=1}^n (J_i - \E[J_i]) >  F_n(q_{\alpha} + \eps) - \alpha \right),
\end{align*}
where $F_n(t) := \P_n(\hat{u}_1/\sigma_n\le t)$ is the marginal cdf of the scaled leave-one-out residuals. If we can show that
\begin{align}\label{eq:PIlengthToShow1}
F_n(t) \to F(t), \quad\quad\forall t\in\R,
\end{align}
as $n\to\infty$, then $F_n(q_\alpha+\eps) \to F(q_\alpha+\eps)>\alpha$, because $q_\alpha$ is unique, and thus the probability in the second to last display can be bounded, at least for $n$ sufficiently large, using Markov's inequality, by
\begin{align*}
&(F_n(q_{\alpha} + \eps) - \alpha)^{-2} \E\left[\left| \frac{1}{n}\sum_{i=1}^n (J_i - \E[J_i])\right|^2 \right]\\
&\quad=
(F_n(q_{\alpha} + \eps) - \alpha)^{-2} \left( 
\frac{1}{n} \Var[J_1] + \frac{n(n-1)}{n^2} \Cov(J_1,J_2)
\right),
\end{align*}
where the equality holds in view of the exchangeability of the $J_i$. An analogous argument yields a similar upper bound for the probability  $\P_n(\hat{q}_\alpha/\sigma_n \le q_{\alpha} - \eps)$ but with $(F_n(q_{\alpha} + \eps) - \alpha)^{-2}$ replaced by $(\alpha - F_n(q_{\alpha} - \eps))^{-2}$, and $J_i$ replaced by $K_i = \mathbbm 1_{\{\hat{u}_i/\sigma_n \le q_{\alpha} - \eps\}}$. The proof will thus be finished if we can establish \eqref{eq:PIlengthToShow1} and show that $\Cov(J_1,J_2)$ and $\Cov(K_1,K_2)$ converge to zero as $n\to\infty$.
We only consider $\Cov(J_1,J_2) = \Cov(1-J_1,1-J_2)$, as the argument for $\Cov(K_1,K_2)$ is analogous. Write $\delta = q_\alpha + \eps$ and 
$$
\Cov(1-J_1,1-J_2) = \P_n(\hat{u}_1/\sigma_n\le \delta, \hat{u}_2/\sigma_n \le \delta) - \P_n(\hat{u}_1/\sigma_n\le \delta)\P_n(\hat{u}_2/\sigma_n \le \delta).
$$ 
Now,
\begin{align*}
\begin{pmatrix}
\hat{u}_1/\sigma_n \\
\hat{u}_2/\sigma_n
\end{pmatrix}
=
\begin{pmatrix}
\hat{u}_{1[2[}/\sigma_n \\
\hat{u}_{2[1]}/\sigma_n
\end{pmatrix}
+
\begin{pmatrix}
\hat{e}_1 \\
\hat{e}_2
\end{pmatrix},
\end{align*}
where $\hat{u}_{i[j]} = y_i - x_i'\hat{\beta}_n^{[ij]}$, $\hat{\beta}_n^{[ij]} = (X_{[ij]}'X_{[ij]})^\dagger X_{[ij]}'Y_{[ij]}$, and $\hat{e}_i = (\hat{u}_i - \hat{u}_{i[j]})/\sigma_n = x_i'(\hat{\beta}_n^{[ij]} - \hat{\beta}_n^{[i]})/\sigma_n$, for $\{i,j\}=\{1,2\}$. Therefore, $\E[\hat{e}_i\|Y_{[i]}, X_{[i]}] = 0$, because $\E[x_i] = 0$, and $\E[\hat{e}_i^2 \| Y_{[i]},X_{[i]}] = \|\Sigma^{1/2}(\hat{\beta}_n^{[i]} - \hat{\beta}_n^{[ij]})/\sigma_n\|_2^2$, which converges to zero in $\P_n$-probability, by Lemma~\ref{lemma:OLSmisspec} (for a sample of size $n-1$ instead of $n$), which applies here because \ref{c.non-linmod} is satisfied under \ref{c.linmod}. Hence, $\hat{e}_1$ and $\hat{e}_2$ converge to zero in probability. The joint distribution function of $\hat{u}_{1[2]}/\sigma_n$ and $\hat{u}_{2[1]}/\sigma_n$ can be written as
\begin{align}
&\P_n(\hat{u}_{1[2]}/\sigma_n\le s, \hat{u}_{2[1]}/\sigma_n \le t)\label{eq:2CDF}\\
&\quad=
\E\left[ \P_n\left(x_1'(\beta_n-\hat{\beta}_n^{[12]})/\sigma_n + v_1 \le s, x_2'(\beta_n-\hat{\beta}_n^{[12]})/\sigma_n + v_2 \le t \Big\| Y_{[12]},X_{[12]}\right)\right]\notag\\
&\quad=
\E\left[ G_n\left(s, \Sigma^{1/2}(\beta_n-\hat{\beta}_n^{[12]})/\sigma_n\right) G_n\left(t, \Sigma^{1/2}(\beta_n-\hat{\beta}_n^{[12]})/\sigma_n\right) \right],\notag
\end{align}
where, for $t\in\R$ and $b\in\R^{p_n}$, $G_n$ is defined as $G_n(t, b) = P_n(b'\Sigma^{-1/2}x_0 + v_0 \le t)$. Note that $G_n$ depends only on $\mathcal L_l$, $\mathcal L_w$, $\mathcal L_v$ and on $n$, through $p_n$. If we abbreviate $\tilde{\beta}_n^{[12]} = \Sigma^{1/2}(\beta_n-\hat{\beta}_n^{[12]})/\sigma_n$ and $\tilde{\beta}_n^{[1]} = \Sigma^{1/2}(\beta_n-\hat{\beta}_n^{[1]})/\sigma_n$, we arrive at
$$
\Cov(1-J_1,1-J_2)
=
\E\left[ G_n\left(\delta,\tilde{\beta}_n^{[12]}\right)^2\right]
- \E\left[ G_n\left(\delta,\tilde{\beta}_n^{[1]}\right)\right]^2+ o(1),
$$
provided the bivariate distribution function in \eqref{eq:2CDF} converges pointwise to a continuous limit.
We finish the proof by showing that for all $t\in\R$, the bounded random variables $G_n(t,\tilde{\beta}_n^{[12]})$ and $G_n(t,\tilde{\beta}_n^{[1]})$ both converge to $F(t)$, in $\P_n$-probability, and hence, \eqref{eq:2CDF} converges to $F(s)F(t)$, which is continuous. Note that this also implies \eqref{eq:PIlengthToShow1}, because $F_n(t) = \E[ G_n(t,\tilde{\beta}_n^{[1]})]$.

To this end, we note that for an arbitrary measureable set $B_n\subseteq\R^{p_n}$ and for any $\eps>0$,
\begin{align*}
\P_n\left( \sup_{t\in\R}\left| G_n(t,\tilde{\beta}_n^{[1]}) - F(t)\right| > \eps \right) 
&\le
\P_n\left( \sup_{t\in\R}\left| G_n(t,\tilde{\beta}_n^{[1]}) - F(t)\right| > \eps, \tilde{\beta}_n^{[1]}\in B_n \right) \\
&\hspace{1cm}+
\P_n\left( \tilde{\beta}_n^{[1]}\notin B_n \right)\\
&\le
a_n(\eps) \;+\; \P_n\left( \tilde{\beta}_n^{[1]}\notin B_n \right),
\end{align*}
where $a_n(\eps) = 1$ if $\sup_{b\in B_n} \sup_{t\in\R} \left| G_n(t,b) - F(t) \right|>\eps$, and $a_n(\eps) = 0$, else. Now, we first consider the case $\kappa=0$. Thus, Lemma~\ref{lemma:OLScorrSpec}, which also applies to $\hat{\beta}_n^{[1]}$, yields $\|\tilde{\beta}_n^{[1]}\|_2 \to \tau = 0$, as $n\to\infty$, in $P_n$-probability. Therefore, the probability in the last line of the previous display converges to zero if we take $B_n  =\{b\in\R^{p_n}: \|b\|_2 \le g_1(n)\}$ and $g_1(n)\to0$ sufficiently slowly, as $n\to\infty$. Hence, Lemma~\ref{lemma:UnifWeak}\eqref{l:UnifWeakT0} applies and shows that also $a_n(\eps)\to 0$ as $n\to\infty$, for every $\eps>0$.
If $\kappa>0$, Lemma~\ref{lemma:OLScorrSpec} yields $\|\tilde{\beta}_n^{[1]}\|_2 \to \tau>0$ and $\|\tilde{\beta}_n^{[1]}\|_4 \to 0$, in $P_n$-probability, as $n\to\infty$. Thus, the probability in the last line of the previous display converges to zero if we take $B_n = \{b\in\R^{p_n} : b \ne 0, | \|b\|_2 - \tau| \le g_1(n), \|b\|_{4}/\|b\|_2 \le g_2(n) \}$ and sequences $g_1$ and $g_2$ that converge to zero sufficiently slowly. Now Lemma~\ref{lemma:UnifWeak}\eqref{l:UnifWeakT>0} shows that also $a_n(\eps)\to 0$ as $n\to\infty$, for every $\eps>0$. The same argument applies to $\tilde{\beta}_n^{[12]}$ instead of $\tilde{\beta}_n^{[1]}$, which finishes the proof in the case $\alpha\in(0,1)$.

Next, we treat the case $\alpha=0$. In either case of the theorem, we have $\lim_{\gamma\to 0}q_\gamma = q_0=-\infty$. By definition, $\hat{q}_0 \le \hat{q}_\gamma$, for any $\gamma\in(0,1)$. Thus, for any $M>0$, there exists a $\gamma\in(0,1)$, such that $q_\gamma<-2M$ and $\P_n(\hat{q}_0/\sigma_{P_n}<-M) \ge  \P_n(\hat{q}_\gamma/\sigma_{P_n}<-M) \to 1$, as $n\to\infty$, in view of the first part. In other words, $\hat{q}_0/\sigma_{P_n}$ converges to $-\infty = q_0$ in $P_n$-probability. A similar argument can be used to treat the case $\alpha=1$. \hfill\qed

\begin{proof}[Proof of Lemma~\ref{lemma:OLScorrSpec}]
On the event $\{\lmin(\tilde{X}'\tilde{X})>0\}$, which has asymptotic probability one in view of Lemma~\ref{lemma:traceConv}\ref{lemma:traceConvA'}, notice the identity 
$$\Sigma_P^{1/2}(\hat{\beta}_n - \beta_P)/\sigma_P = (\tilde{X}'\tilde{X})^\dagger\tilde{X}'v,$$ 
where $v = (v_1,\dots, v_n)' = (Y-X\beta_P)/\sigma_P$. Thus, the distribution under $P\in\mathcal P_n$ of the quantity of interest does not depend on the parameters $\beta_P$, $\sigma_P^2$ and $\Sigma_P$. Hence, without loss of generality, we assume for the rest of this proof that $\beta_P=0$, $\sigma_P^2=1$ and $\Sigma_P = I_{p_n}$. 
First, we have to show that $\|\hat{\beta}_n\|_2 \to \tau \in [0,\infty)$, in probability, for a $\tau = \tau(\mathcal L_l, \kappa)$ with the properties mentioned in the lemma. To this end, consider the conditional mean
$$
\E\left[\|\hat{\beta}_n\|_2^2\Big\| X\right] = \trace (X'X)^\dagger X'X (X'X)^\dagger = \trace (X'X)^\dagger \xrightarrow[]{a.s.} \tau^2,
$$
by Lemma~\ref{lemma:traceConv}\ref{lemma:traceConvC} and for $\tau$ as desired (cf. Remark~\ref{rem:tau}). From the same lemma we get convergence of the conditional variance
\begin{align*}
\Var\left[\|\hat{\beta}_n\|_2^2\Big\| X\right] &= \Var\left[v'X(X'X)^{\dagger2} X'v\Big\|X\right] =: \Var[v'Kv\|X] \\
&=
2\trace K^2 + (\E[v_1^4] - 3)\sum_{i=1}^n K_{ii}^2\\
&\le
2\trace K^2 + (\E[v_1^4] + 3)\sum_{i,j=1}^n K_{ij}^2
=
(\E[v_1^4] + 5)\trace K^2 \\
&= 
(\E[v_1^4] + 5)
\trace X(X'X)^{\dagger2}X'X(X'X)^{\dagger2}X' \\
&= (\E[v_1^4] + 5)
\trace (X'X)^{\dagger2} \xrightarrow[]{a.s.} 0. 
\end{align*}
For the second claim it suffices to show that $\|\hat{\beta}_n\|_4^4 \to 0$, in probability. Notice that for $M := (m_1, \dots, m_{p_n})' := (X'X)^\dagger X'$, we have
\begin{align*}
\|\hat{\beta}_n\|_4^4 = \|Mv\|_4^4 = \sum_{j=1}^{p_n} (m_j'v)^4
=
\sum_{j=1}^{p_n} \sum_{i_1,i_2,i_3,i_4=1}^n m_{ji_1}m_{ji_2}m_{ji_3}m_{ji_4}v_{i_1}v_{i_2}v_{i_3}v_{i_4}.
\end{align*}
After taking conditional expectation given $X$, only terms with paired indices remain and we get
\begin{align*}
\E\left[\|\hat{\beta}_n\|_4^4\Big\|X\right] &= \sum_{j=1}^{p_n}\left( \E[v_1^4] \sum_{i=1}^n m_{ji}^4 + 3 \sum_{i\ne k}^n m_{ji}^2m_{jk}^2\right)\\
&\le
\sum_{j=1}^{p_n}\left( \E[v_1^4] \sum_{i,k=1}^n m_{ji}^2m_{jk}^2 + 3 \sum_{i, k=1}^n m_{ji}^2m_{jk}^2\right)\\
&=
(\E[v_1^4] + 3) \sum_{j=1}^{p_n} (m_j'm_j)^2
\le
(\E[v_1^4] + 3) \trace \sum_{i,j=1}^{p_n} m_im_i'm_jm_j'\\
&=
(\E[v_1^4] + 3) \trace (M'M)^2 
= 
(\E[v_1^4] + 3) \trace (X'X)^{\dagger2} \xrightarrow[]{a.s.} 0,
\end{align*}
by Lemma~\ref{lemma:traceConv}\ref{lemma:traceConvB}.
\end{proof}

\begin{proof}[Proof of Lemma~\ref{lemma:UnifWeak}]
First, in the case \eqref{l:UnifWeakT0}, for every $n\in\N$, take $b_n \in B_n =\{b\in\R^{p_n}: \|b\|_2 \le g_1(n)\}$ and simply note that $l_0b_n'w_n \to 0$, in probability, and thus $G(t,b_n) \to F(t)$ weakly. Since the limit is continuous, Polya's theorem yields uniform convergence in $t\in\R$. Since the sequence $b_n\in B_n$ was arbitrary, we also get uniform convergence over $B_n$.

Next, we consider the Gaussian case \eqref{l:UnifWeakGauss}, so $B_n = \{b\in\R^p : | \|b\|_2 - \tau| \le g_1(n)\}$. For every $n\in\N$, let $b_n\in B_n$ be arbitrary, and note that $t\mapsto G(t,b_n)$ is the distribution function of $l_0 b_n'w_n + v_0$, where $w_n \stackrel{\mathcal L}{=} \mathcal N(0,I_{p_n})$, and $l_0, w_n, v_0$ are independent. Clearly, $l_0 b_n'w_n + v_0 \stackrel{\mathcal L}{=} l_0 N \|b_n\|_2 + v_0 \to l_0 N \tau + v_0$, weakly, and this limit has continuous distribution function $F$. Hence, by Polya's theorem, $\sup_t |\P(l_0 b_n'w_n + v_0 \le t) - F(t)| \to 0$, as $n\to \infty$. And since the sequence $b_n\in B_n$ was arbitrary, the result follows.

In the general case \eqref{l:UnifWeakT>0} first note that $B_n$ may be empty. By our convention that $\sup \varnothing = 0$ it suffices to restrict to the subsequence $n'$ for which $B_{n'}\ne \varnothing$. If this is only a finite sequence, then the result is trivial. For convenience, we write $n=n'$.
So let $b_n\in B_n$ and define the triangular array $z_{nj} := b_{nj} w_{0j}$, $j=1,\dots, p_n$, which satisfies $\E[z_{nj}]=0$ and $s_n^2 := \sum_{j=1}^p \E[z_{nj}^2] = \|b_n\|_2^2 \ne 0$. The Lyapounov condition is verified by
\begin{align*}
\sum_{j=1}^{p_n} s_n^{-(2+\delta)} \E[|z_{nj}|^{2+\delta}] 
\;&=\;
\E\left[|w_{01}|^{2+\delta}\right]\left(\frac{\|b_n\|_{2+\delta}}{\|b_n\|_2}\right)^{2+\delta}\\
\;&\le\;
\E\left[|w_{01}|^{2+\delta}\right] \left[ g_2(n)\right]^{2+\delta} 
\; \xrightarrow[n\to\infty]{} \; 0.
\end{align*}
Therefore, by Lyapounov's CLT \citep[][Theorem~27.3]{Billingsley95}, we have
$$
b_n'w_n/\|b_n\|_2 \;=\; \sum_{j=1}^{p_n} z_{nj}/s_n \;\xrightarrow[n\to\infty]{w}\; \mathcal N(0,1).
$$
Since $b_n\in B_n$, we must have $\|b_n\|_2 \to \tau$ as $n\to\infty$, and thus, $b_n'w_n = \|b_n\|_2 b_n'w_n/\|b_n\|_2 \xrightarrow[]{w} N\tau$, where $N \stackrel{\mathcal L}{=} \mathcal N(0,1)$, as $n\to\infty$, and, by independence, $l_0 b_n'w_n + v_0 \xrightarrow[]{w} l_0 N \tau + v_0$. Since the distribution function of this limit is continuous, Polya's theorem yields $\sup_t | G(t,b_n) - F(t)| \to 0$, as $n\to\infty$. Now the proof is finished because this convergence holds for arbitrary sequences $b_n\in B_n$.
\end{proof}


\section{Auxiliary results}
\label{sec:proofsaux}

\begin{lemma}\label{lemma:doubleSum}
If $a\in[0,1]^k$, $k\ge 2$, then
$$
f_k(a) := \frac{1}{k(k-1)} \sum_{i\neq j} a_i(1-a_j) \;\le\; \frac{k}{k-1} \cdot \frac14.
$$
\end{lemma}
\begin{remark}\normalfont
Note that in general the factor $\frac{k}{k-1}$ in the upper bound of Lemma~\ref{lemma:doubleSum} can not be removed. For instance, $f_2((1,0)) = \frac12$, so the bound is even tight for $k=2$, and $f_3((1,0,0)) = \frac13 > \frac14$.
\end{remark}

\begin{proof}
Simply write
\begin{align*}
f_k(a) 
&= \frac{1}{k(k-1)} \left( \sum_{i,j=1}^k a_i(1-a_j) - \sum_{i=1}^k a_i(1-a_i)\right)\\
&=
\frac{k}{k-1} \bar{a} \left( 1 - \bar{a} \right) - \frac{1}{k(k-1)}\sum_{i=1}^k a_i(1-a_i)\\
&\le
\frac{k}{k-1} \cdot \frac14,
\end{align*}
where $\bar{a} = \frac1k\sum_{i=1}^k a_i$.
\end{proof}

\begin{lemma}
\label{lemma:traceConv}
On a common probability space $(\Omega, \mathcal F, \P)$, consider an i.i.d. sequence $L_0 = \{l_i : i=1,2,\dots\}$ of random variables satisfying $|l_1|\ge c>0$, and a double infinite array $W_0 = \{w_{ij} : i,j=1,2,\dots\}$ of i.i.d. random variables with mean zero, unit variance and $\E[w_{11}^4]<\infty$, such that $L_0$ and $W_0$ are independent. For a sequence of positive integers $(p_n)$ with $p_n\le n$, consider the $n\times p_n$ random matrix $\tilde{X} = \Lambda W$, where $\Lambda = \diag(l_1,\dots, l_n)$ is diagonal and $W = \{w_{ij} : i=1,\dots, n; j=1,\dots, p_n\}$. Let $(\tilde{X}'\tilde{X})^\dagger$ denote the Moore-Penrose pseudo inverse of $\tilde{X}'\tilde{X}$. If $p_n/n\to \kappa\in[0,1)$ then the following holds:
\begin{enumerate}
	\renewcommand{\theenumi}{(\roman{enumi})}
	\renewcommand{\labelenumi}{{\theenumi}} 
	
	\item\label{lemma:traceConvA} $\liminf_{n\to\infty} \lmin (\tilde{X}'\tilde{X}/n) \ge c^2(1-\sqrt{\kappa})^2$, almost surely.
	\item\label{lemma:traceConvA'} $\lim_{n\to\infty} \P(\lmin (\tilde{X}'\tilde{X}/n) \le \eps) = 0$ for all $\eps < c^2(1-\sqrt{\kappa})^2$.
	\item\label{lemma:traceConvB} If $m>1$, then $\trace{(\tilde{X}'\tilde{X})^{\dagger m}} \to 0$, almost surely, as $n\to\infty$.
	\item \label{lemma:traceConvC} $\trace{(\tilde{X}'\tilde{X})^{\dagger}} \to \tau^2$ almost surely, as $n\to\infty$, for some constant $\tau\in [0,\infty)$ that depends only on $\kappa$ and on the distribution of $l_1^2$ and satisfies $\tau=0$ if, and only if, $\kappa=0$.
\end{enumerate}
\end{lemma}

\begin{proof}
Let $\lambda_1^{(n)}\le \dots \le \lambda_{p_n}^{(n)}$ and $\mu_1^{(n)}\le \dots \le \mu_{p_n}^{(n)}$ denote the ordered eigenvalues of $\tilde{X}'\tilde{X}/n$ and $W'W/n$, respectively. Then, 
\begin{align*}
\lambda_1^{(n)} = \inf_{\|t\|=1} t'W'\Lambda^2Wt/n 
\ge \left(\min_{i=1,\dots, n} l_i^2\right) \inf_{\|t\|=1} t'W'Wt/n 
\ge c^{2} \mu_1^{(n)},
\end{align*}
and from the Bai-Yin Theorem \citep{Bai93} it follows that $\mu_1^{(n)} \to (1-\sqrt{\kappa})^{2}>0$, almost surely, as $p_n/n\to \kappa \in[0,1)$ \citep[cf.][for the case $\kappa=0$]{Huber13}. This finishes the proof of part~\ref{lemma:traceConvA}. Part~\ref{lemma:traceConvA'} is now a textbook argument: Simply note that for $k\le n$, we have
$\P(\lambda_1^{(n)} \le \eps) \le \P(\inf_{r\ge k} \lambda_1^{(r)} \le \eps)$ and that $\inf_{r\ge k} \lambda_1^{(r)} \le \inf_{r\ge k+1} \lambda_1^{(r)}$ for all $k\in\N$. Thus
\begin{align*}
&\limsup_{n\to\infty} \P(\lambda_1^{(n)} \le \eps) 
=
\inf_{k\in\N}\sup_{n\ge k} \P(\lambda_1^{(n)} \le \eps) \\
&\quad
\le \inf_{k\in\N} \P\left(\inf_{r\ge k} \lambda_1^{(r)} \le \eps\right)
=
\lim_{k\to\infty} \P\left(\inf_{r\ge k} \lambda_1^{(r)} \le \eps\right) \\
&\quad
= 
\P\left(\forall k\in\N : \inf_{r\ge k} \lambda_1^{(r)} \le \eps\right)
=
\P\left(\liminf_{n\to\infty} \lambda_1^{(n)} \le \eps\right) = 0.
\end{align*}

Next, abbreviate $\lambda_j=\lambda_j^{(n)}$, $\mu_j=\mu_j^{(n)}$, for $m\ge1$ set $\alpha_m := c^{2m} (1-\sqrt{\kappa})^{2m}$ and, for $\alpha>0$, define the functions $h_0$ and $h_\alpha$ by $h_0(y) = 1/|y|$ if $y\ne 0$ and $h_0(0)=0$, and by $h_\alpha(y) = 1/|y|$, if $|y|> \alpha/2$ and $h_\alpha(y) = 2/\alpha$, if $|y|\le\alpha/2$. With this notation, and from the previous considerations, we see that the difference between
$$
\trace{(\tilde{X}'\tilde{X})^{\dagger m}} = n^{-m} \trace{(\tilde{X}'\tilde{X}/n)^{\dagger m}}
=  \frac{p_n}{n^m} \frac{1}{p_n}\sum_{j=1}^{p_n} h_0(\lambda_j^m),
$$
and 
$$
\frac{p_n}{n^m} \frac{1}{p_n}\sum_{j=1}^{p_n} h_{\alpha_m}(\lambda_j^m)
$$
converges to zero, almost surely, because $\lambda_j^m\ge \lambda_1^m \ge c^{2m}\mu_1^{m} \to \alpha_m > \alpha_m/2>0$, almost surely. But we have $n^{-m} \sum_{j=1}^{p_n} h_{\alpha_m}(\lambda_j^m) \le (p_n/n^m) (2/\alpha_m) \to 0$, if $m>1$, or if $m=1$ and $\kappa=0$. This finishes \ref{lemma:traceConvB} and the case $\kappa=0$ of part~\ref{lemma:traceConvC}.

For the remainder of part~\ref{lemma:traceConvC}, let $m=1$ and $\kappa>0$, and first note that the empirical spectral distribution function $F_n^{\Lambda^2}$ of $\Lambda^2$ is simply given by the empirical distribution function of $l_1^2,\dots, l_n^2$, and this converges weakly (even uniformly) to the distribution function of $l_1^2$, almost surely. Hence, from Theorem~4.3 in \citet{BaiSilv10}, it follows that, almost surely, the empirical spectral distribution function $F_n^{\tilde{X}'\tilde{X}/n}$ of $\tilde{X}'\tilde{X}/n$ converges vaguely, as $p_n/n\to\kappa\in(0,1)$, to a non-random distribution function $F$ that depends only on $\kappa$ and on the distribution of $l_1^2$. From the argument in the previous paragraph we know that $\lambda_1 \ge c^2\mu_1 \to c^2(1-\sqrt{\kappa})^2=\alpha_1>0$, almost surely, and thus the support of $F$ must be lower bounded by $\alpha_1$. Since $h_{\alpha_1}$ is continuous and vanishes at infinity, by vague convergence, we have \citep[cf.][relation (28.2)]{Billingsley95}
\begin{align*}
\frac{1}{p_n}\sum_{j=1}^{p_n} h_{\alpha_1}(\lambda_j) = \int\limits_{-\infty}^\infty &h_{\alpha_1}(y) dF_n^{\tilde{X}'\tilde{X}/n}(y) \\
&\xrightarrow[]{a.s.} \int\limits_{-\infty}^\infty h_{\alpha_1}(y) dF(y) 
= \int\limits_{-\infty}^\infty \frac{1}{y}\,dF(y) =: \tau_0^2 \in (0, 1/\alpha_1).
\end{align*}
Thus 
$$
\frac{p_n}{n} \frac{1}{p_n} \sum_{j=1}^{p_n} h_{\alpha_1}(\lambda_j) \quad\xrightarrow[]{a.s.}\quad \kappa \tau_0^2 \;=:\; \tau^2 >0.
$$
\end{proof}

\begin{remark}\label{rem:tau}\normalfont
If the $l_i$ in Lemma~\ref{lemma:traceConv} satisfy $|l_i|=1$, almost surely, then $\tau$ in part~\ref{lemma:traceConvC} is given by $\tau(\kappa)= \sqrt{\kappa/(1-\kappa)}$ \citep[cf.][Lemma~B.2]{Huber13}.
\end{remark}


\begin{lemma}\label{lemma:JSneg}
Let $\mathcal P_n^{(lin)}(\mathcal L_l, \mathcal L_w, \mathcal L_v)$ denote the class of distributions in Condition~\ref{c.linmod} and assume that $\mathcal L_l$ has a finite fourth moment. Furthermore, let $p_n/n\to\kappa>0$ and $n > p_n$ for all $n\in\N$. Then there exist probability measures $P_n\in\mathcal P_n^{(lin)}(\mathcal L_l, \mathcal L_w, \mathcal L_v)$, such that for every $c\in[0,1]$, every $\eta\in(0,\infty]$ and every $\eps\in(0,1)$, the James-Stein-type estimator $\hat{\beta}_n(c)$ satisfies
\begin{align*}
\P_n\left(
		\Big\|\Sigma_{P_n}^{1/2}\left(\hat{\beta}_n(c) - \beta_{P_n}\right)/\sigma_{P_n}\Big\|_{2+\eta}  \ge \eps c\sqrt{\kappa}/2
\right) \quad\xrightarrow[n\to\infty]{}\quad1.
\end{align*}
\end{lemma}

\begin{proof}
Consider a sequence $P_n\in\mathcal P_n$, such that $\beta_{P_n} = \sigma_{P_n}\Sigma_{P_n}^{-1/2}(\sqrt{\kappa},0,\dots, 0)'$, which implies $b := \Sigma_{P_n}^{1/2}\beta_{P_n}/\sigma_{P_n} = (\sqrt{\kappa}, 0,\dots,0)'\in\R^{p_n}$ and $\|b\|_2 = \|b\|_q = \sqrt{\kappa}$, for every $q\in(0,\infty]$. Simple relations of $\ell^q$-norms yield
\begin{align*}
&\left\|\Sigma_{P_n}^{1/2}\left(\hat{\beta}_n(c) - \beta_{P_n}\right)/\sigma_{P_n}\right\|_{2+\eta}  
\ge
\left\|\Sigma_{P_n}^{1/2}\left(\hat{\beta}_n(c) - \beta_{P_n}\right)/\sigma_{P_n}\right\|_{(2+\eta)\lor 4}\\
&\quad=
\left\|s_n\Sigma_{P_n}^{1/2}(\hat{\beta}_n - \beta_{P_n})/\sigma_{P_n} - (1-s_n)b\right\|_{(2+\eta)\lor 4}  \\
&\quad\ge
\left|
	|s_n|\left\|\Sigma_{P_n}^{1/2}(\hat{\beta}_n - \beta_{P_n})/\sigma_{P_n}\right\|_{(2+\eta)\lor 4} 
	-
	|s_n-1|\sqrt{\kappa}
\right|\\	
&\quad\ge
|s_n-1|\sqrt{\kappa} - |s_n|\left\|\Sigma_{P_n}^{1/2}(\hat{\beta}_n - \beta_{P_n})/\sigma_{P_n}\right\|_{(2+\eta)\lor 4},
\end{align*}
where $s_n$ is defined as before, i.e.,
\begin{align*}
s_n = \begin{cases}
\left( 1 - \frac{p_n}{n}\frac{c}{t_n^2}\frac{\hat{\sigma}_n^2}{\sigma_n^2}\right)_+, \quad&\text{if } t_n^2 >0,\\
1, &\text{else,}
\end{cases}
\end{align*}
and $t_n^2 = \hat{\beta}_n'X'X\hat{\beta}_n/(n\sigma_n^2)$, so that $\hat{\beta}_n(c) = s_n \hat{\beta}_n$. Clearly, we have $|s_n|\le 1$ and $\left\|\Sigma_{P_n}^{1/2}(\hat{\beta}_n - \beta_{P_n})/\sigma_{P_n}\right\|_{(2+\eta)\lor 4} \le \left\|\Sigma_{P_n}^{1/2}(\hat{\beta}_n - \beta_{P_n})/\sigma_{P_n}\right\|_4 \to 0$, in $P_n$-probability, by Lemma~\ref{lemma:OLScorrSpec}. Therefore, we see that
\begin{align}\label{eq:lowerB}
&\P_n\left( |s_n-1|\sqrt{\kappa} - o_{P_n}(1) \ge \eps c\sqrt{\kappa}/2 \right)\\
&\quad\le 
\P_n\left( \left\|\Sigma_{P_n}^{1/2}\left(\hat{\beta}_n(c) - \beta_{P_n}\right)/\sigma_{P_n}\right\|_{2+\eta} \ge \eps c\sqrt{\kappa}/2 \right).\notag
\end{align}
Now, as in the proof of Lemma~\ref{lemma:JSmisspec} but with $e=0$ (now the linear model is assumed correct), 
$$
t_n^2 = \left\| \frac{\tilde{X}\Sigma_{P_n}^{1/2}\beta_{P_n}}{\sqrt{n\sigma_{P_n}^2}} + \frac{\Pi_{\tilde{X}}v}{\sqrt{n}}\right\|_2^2
$$
and we showed already in that proof that the mixed term vanishes and $\|\Pi_{\tilde{X}}v/\sqrt{n}\|_2^2\to \kappa$ in $P_n$-probability. For the remaining quadratic form note that $\tilde{x}_i \stackrel{\P_n}{=} l_i (w_{i1},\dots, w_{ip_n})'$, where $l_i \thicksim \mathcal L_l$ and $w_{ij}\thicksim\mathcal L_w$ are all independent, for $i=1,\dots, n$, $j=1,\dots,p_n$, in view of Condition~\ref{c.linmod}. Thus $A:= \|\tilde{X}\Sigma_{P_n}^{1/2}\beta_{P_n}/\sqrt{n\sigma_{P_n}^2}\|_2^2 = \frac{1}{n} \sum_{i=1}^n (b'\tilde{x}_i)^2$. Clearly, $\E[A] = \|b\|_2^2 = \kappa$ and
\begin{align*}
\Var[A] &= \frac{1}{n} \Var[(b'\tilde{x}_1)^2]
=
\frac{1}{n} \Var\left[ l_1^2 \sum_{j_1,j_2}^{p_n} b_{j_1}b_{j_2} w_{1j_1}w_{2j_2}\right] \\
&\le
\frac{1}{n} \E\left[ l_1^4 \sum_{j_1,j_2,j_3,j_4}^{p_n} b_{j_1}b_{j_2}b_{j_3}b_{j_4}w_{1j_1}w_{2j_2}w_{1j_3}w_{2j_4}\right] \\
&=
\frac{1}{n} \E[l_1^4] \left[ \E[w_{11}^4] \sum_{j=1}^{p_n} b_j^4 + 3 \sum_{j_1\ne j_2}^{p_n} b_{j_1}^2b_{j_2}^2\right] \\
&\le
\frac{1}{n} \E[l_1^4] (\E[w_{11}^4] + 3) \|b\|_2^4 \xrightarrow[]{n\to\infty} 0.
\end{align*}
Thus, $t_n^2 \to \kappa + \kappa = 2\kappa$, in $P_n$-probability. Moreover, $\hat{\sigma}_n^2 /\sigma_n^2\to 1$, in $P_n$-probability, because its conditional mean given $X$ converges to $1$ and its conditional variance converges to zero (see the arguments in \eqref{eq:CondVarHatsigma} and in the lines immediately before that display). Thus $|s_n-1|\sqrt{\kappa} \to c \sqrt{\kappa}/2 > \eps c \sqrt{\kappa}/2$, such that the probability in \eqref{eq:lowerB} converges to $1$ and the proof is finished. 
\end{proof}


\begin{lemma}\label{lemma:0stable}
If $\hat{\mu}_n$ is symmetric and has $k$-stability coefficient $\eta_{n,k}(P)=0$ for all data generating distributions $P$ in some class $\mathcal P$ of distributions on $\mathcal Z$, then there exists a collection $\{g_P:P\in\mathcal P\}$ of measurable functions $g_P : \mathcal X\to \R$, such that for all $P\in\mathcal P$, 
\begin{align*}
&\P\left(\{(T_n,z_0)\in \mathcal Z^{n+1} : M_{n,p}(T_n,x_0) = g_P(x_0)\}\right)=1.
\end{align*}
\end{lemma}
\begin{remark}\normalfont
Note that the dependence of $g_P$ on $P$ in Lemma~\ref{lemma:0stable} can not be avoided. For example, suppose that $\mathcal P = \{P_0,P_1\}$ with disjoint supports $S_0\cap S_1 =\varnothing$ and consider the naive algorithm
\begin{align*}
M_{n,p}(T_n,x_0) := \begin{cases}
0, &\text{if } T_n \in S_0^n,\\
1, &\text{else.}
\end{cases}
\end{align*} 
Clearly, this algorithm has $\eta_{n,k}(P_0)=\eta_{n,k}(P_1)=0$, but it does depend on $T_n$. This is not in contradiction with Lemma~\ref{lemma:0stable}, because $M_{n,p}(T_n,x_0) = g_P(x_0)$, $\P$-almost surely, for any $P\in\mathcal P$, where
\begin{align*}
g_P(x_0) := \begin{cases}
0, &\text{if } P=P_0,\\
1, &\text{if } P=P_1.
\end{cases}
\end{align*} 
The paradox is resolved by noticing that since the supports $S_0$ and $S_1$ are disjoint we can perfectly discriminate between $P_0$ and $P_1$ using the training data.
\end{remark}

\begin{proof}[Proof of Lemma~\ref{lemma:0stable}]
Fix $P\in\mathcal P$. For $i=1,\dots, n$, let $z_i, z_i'\in\mathcal Z$ and $x_0\in\mathcal X$. For $l\in[k+1]$, let $C_l\in\mathcal Z^n$ be the $n$-vector with entries $(C_l)_i = z_i'$, if $i\in\bigcup_{\ell=1}^{l-1} K_\ell$, and $(C_l)_i = z_i$, if $i\in \bigcup_{\ell=l}^k K_\ell$, where empty unions are interpreted as the empty set. For $l\in[k]$, the $(n-m_l)$-vector $C_l^*$ is obtained from $C_l$ by removing the components with index $i\in K_l$, or, equivalently, from $C_{l+1}$ by again removing components with index $i\in K_l$. Note that
\begin{align*}
&\left| M_{n,p}((z_i)_{i\in[n]},x_0) - M_{n,p}((z_i')_{i\in[n]},x_0) \right|\\
&=
\left|\sum_{l=1}^k \left[M_{n,p}(C_l,x_0) - M_{n,p}(C_{l+1},x_0) \right]\right|\\
&\le
\sum_{l=1}^k \Big[\left| M_{n,p}(C_l,x_0) - M_{n-m_l,p}(C_l^*,x_0) \right|+\\
&\quad\quad \left|M_{n-m_l,p}(C_l^*,x_0) - M_{n,p}(C_{l+1},x_0) \right|\Big].
\end{align*}
Since $\eta_{n,k}(P)=0$, the integral of this upper bound with respect to $P^{2n+1}$ is equal to zero. Therefore, applying Lemma~\ref{lemma:0stableKEY} with $f= M_{n,p}$, $S = \mathcal Z^n$, $P_S = P^n$, $T = \mathcal X$ and $P_T$ equal to the $x$-marginal distribution of $P$, the claim follows.
\end{proof}

\begin{lemma} \label{lemma:0stableKEY}
Let $(S,\mathcal S, P_S)$ and $(T, \mathcal T, P_T)$ be two probability spaces, and let $f:S\times T \to \R$ be measurable w.r.t. the product sigma algebra $\mathcal S \otimes \mathcal T$ and the Borel sigma algebra on $\R$. 
If $$\int_{S^2\times T} | f(s_1,t) - f(s_2,t)|\,d P_S\otimes P_S\otimes P_T(s_1,s_2, t) = 0,$$
then there exists a measurable function $g:T\to\R$, such that 
$$P_S\otimes P_T\Big((s,t): f(s,t) = g(t)\Big) = 1.$$
\end{lemma}
\begin{proof}
By Tonelli's theorem we have
$$
\int_{T} | f(s_1,t) - f(s_2,t)|\,d P_T(t) = 0,
$$
for $P_S\otimes P_S$-almost all $(s_1,s_2)$, i.e., for all $(s_1,s_2)\in N^c\in\mathcal S\otimes\mathcal S$, where $P_S\otimes P_S(N)=0$. Furthermore, whenever $(s_1,s_2)\in N^c$, then $f(s_1,t) = f(s_2,t)$, for $P_T$-almost all $t$, i.e., for all $t \in M(s_1,s_2)^c\in\mathcal T$, with $P_T(M(s_1,s_2))=0$.
For $s_1\in S$, consider $N_{s_1} := \{s\in S: (s_1,s)\in N\}$, i.e., the $s_1$-section of $N$, and use Tonelli again, to see that there exists a $P_S$-null set $L\in\mathcal S$, such that $P_S(N_{s_1}) = 0$, for all $s_1\in L^c$.

Next, fix $s_1\in L^c$ and define the set 
$$A := A(s_1):= \{(s,t)\in S\times T : s\in N_{s_1}^c, t\in M(s_1,s)^c\},$$ 
as well as the function $g(t):= f(s_1,t)$, for $t\in T$.\footnote{Note that by construction, the function $g$ depends not only on $f$, but also on the null set $L$, and thus on both the probability spaces $(S,\mathcal S, P_S)$ and $(T, \mathcal T, P_T)$.} 
We therefore have $A \subseteq \{(s,t): f(s_1,t)=f(s,t)\} = \{(s,t): g(t)=f(s,t)\}$ and, for $s\in N_{s_1}^c$, $A_s = M(s_1,s)^c$ has $P_T$-probability one. To conclude, we use Tonelli again, to obtain
$$
P_S\otimes P_T(A) = \int_S P_T(A_s) \,dP_S(s)
=
\int_{N_{s_1}^c} P_T(A_s) \,dP_S(s)
=
P_S(N_{s_1}^c) = 1.
$$
\end{proof}

\end{appendix}

\end{document}